\documentclass[a4paper]{fmfibak}

\usepackage{ifpdf}                      % overovanie, ci robime pdf alebo nie
\usepackage{multirow}
\usepackage{fancyhdr}
\usepackage{amsmath}
\usepackage{amsfonts}
\usepackage{amssymb}
\usepackage{fancyhdr}
\usepackage{eucal}
\usepackage{graphicx}
\usepackage{colortbl}
\usepackage{wrapfig}
\usepackage{epsfig}
\usepackage{hyperref}
\usepackage{enumerate}

\usepackage[slovak, british]{babel}
\usepackage[cp1250]{inputenc}           % diakritika
\usepackage[T1]{fontenc}                % diakritika
\usepackage{latexsym}                   % dalsie symboly
\usepackage{amsmath,amsthm,amsfonts,amssymb}   % standardne symboly
\usepackage{wasysym, marvosym,color} 
\usepackage{verbatim}
\usepackage{pdfpages}
\usepackage[margin=3cm]{geometry}
\usepackage{multicol}
\usepackage{multirow}
\usepackage{longtable}
\usepackage{array}
\usepackage{float}
\newcounter{rowcount}
\newsavebox\ltmcbox

\let\epsilon = \varepsilon

\addto\captionsslovak{%
 \def\refname{References}%
 \def\abstractname{Abstract}%
 \def\contentsname{Contents}%
 \def\listfigurename{List of Figures}%
 \def\listtablename{List of Tables}%
 \def\indexname{Index}%
 \def\figurename{Figure}%
 \def\tablename{Table}%
 \def\partname{Part}%
 }

\ifpdf
  \usepackage[pdftex]{graphicx}          % grafika
  \DeclareGraphicsExtensions{.pdf,.png}  % pripony obrazkov
  	\usepackage{pgf}
  	\usepackage{tikz}
  \usepackage[pdftex]{hyperref}   %% tlac

 \hypersetup{%

 pdfcreator={pdfLaTeX},
 pdfkeywords={graph, srg, spectra},
 pdftitle={},
 pdfauthor={Kristina Kovacikova},
} % hypersetup
\else
  \usepackage{graphicx}
  \DeclareGraphicsExtensions{.eps}      % pripony obrazkov
  \usepackage{pgf}
    \usepackage{tikz}
    \usepackage{pstricks}
\fi

\usepackage[numbers]{natbib}

\katedra{Katedra algebry geometrie a didatktiky matemaky}
\department{Department of Algebra, Geometry, and Mathematical Education}
\specialization{Discrete mathematics}
\autor{Mgr. Kristína Kováčiková}
\veduci{Doc. RNDr. Martin Mačaj, PhD.}
\nazov{Indukované podgrafy v silno regul\'arnych grafoch}
\klucoveslova{silno regulárne grafy, indukované podrgafy, grupa automorfizmov}
\title{Induced Subgraphs in Strongly Regular Graphs}
\datum{28.5.2015}

\theoremstyle{plain}
\newtheorem{prop}{Proposition}[section]
\newtheorem{veta}{Theorem}[section]
\newtheorem{cor}{Corollary}[section]
\newtheorem{lema}{Lemma}[section]
\newtheorem{definicia}{Definition}[section]
\newtheorem*{rem}{Remark}

\begin{document}

\titulnestrany
\newpage

\podakovanie I want to thank my supervisor Doc. RNDr. Martin Ma\v caj, PhD.,
for his help, the materials he gave me and for all the
insightful answers to my questions.
\kpodakovania

\newpage

%\anotacia
%\par T\'ato pr\'aca sa venuje v\'yvoju teoretick\'eho a algoritmick\'eho n\'astroja, ktor\'y je schopn\'y odvodi\v t po\v cty indukovan\'ych podgrafov v silno regul\'arnom grafe a jeho n\'asledn\'ym aplik\'aci\'am. \v Speci\'alnu pozornos\v t venujeme jednej z podtried t\'ychto grafov, konkr\'etne beztrojuholn\'ikov\'ym SRGs. Tento konkr\'etny pr\'ipad obsahuje nekone\v cne ve\v la s\'ad pr\'ipustn\'ych parametrov pre SRG. Na druhej strane, existuje len sedem zn\'amych pr\'ikladov bez trojuholn\'ikov\'ych SRG.\par
%Vyvinuli sme algoritmus pre kon\v strukciu line\'arnych rovn\'ic popisuj\'ucich vz\v tahy medzi $o$ a $o-1$ vrcholov\'ymi podgrafmi v SRG. Takto dosiahnut\'e v\'ysledky sme pou\v zili napr\'iklad aj v pr\'ipade $srg(3250,57,0,1)$, ktor\'eho existencia je sl\'avny otvoren\'y probl\'em. Pre tento graf je po\v cet jeho indukovan\'ych podgrafov izomorfn\'ych dan\'emu grafu na desiatich vrcholoch z\'avisl\'y len od po\v ctu jeho indukovan\'ych Petersenov\'ych grafov. 
%\v Dalej prezentujeme nov\'e v\'ysledky o grupe automorfizmov pre $srg(3250,57,0,1)$ ako aj nov\'e ohrani\v cenia pre po\v cet indukovan\'ych $K_{3,3}$ v bez trojuholn\'ikov\'ych SRG. Na konci pr\'ace ukazujeme mo\v znos\v t uplatnenia n\'a\v sho pr\'istupu pri sk\'uman\'i takzvanej $t$-vrcholovej podmienky.
%\kanotacie

\anotaciaeng  
\par This thesis focuses on theoretical and algorithmic tools for determining the numbers of induced subgraphs in strongly regular graphs, SRGs, and on further applications of such numbers. We consider in more detail a restricted class of these graphs, specifically those with no triangles. In this special case, there are infinitely many feasible sets of parameters for SRGs. Despite this fact there are only seven known examples of such graphs.\par
we develop an algorithm which produces linear equations describing various relations between numbers of induced subgraphs of orders $o$ and $o-1$ in a SRG. We apply our results also on $srg(3250,57,0,1)$ (existence of which is a famous open problem). In this case, the number of induced subgraphs isomorphic to a given graph on $10$ vertices depends only on the number of induced Petersen graphs. 
Furthermore, we provide new insights about automorphisms of $srg(3250,57,0,1)$ as well as bounds for the numbers of induced $K_{3,3}$ in general triangle-free SRGs. At the end of the thesis we discuss possible extension of our approach for the study of so called $t$-vertex condition.

\kanotaciaeng

\thispagestyle{empty}

\tableofcontents

\newpage

\setcounter{page}{1}

%\newpage
%\addcontentsline{toc}{section}{\numberline{}List of Figures}
%\listoffigures

%\newpage
%\addcontentsline{toc}{section}{\numberline{}List of Tables}
%\listoftables

\addcontentsline{toc}{section}{\numberline{}Introduction}
\section*{Introduction}
Strongly regular graphs (also referred to as SRGs) are objects standing somewhere between highly symmetric and random graphs. They form an important class of graphs with applications in various fields such as coding theory, theoretical chemistry, or group theory.
\par SRGs appeared for the first time in the work of R. C. Bose \cite{Bos}, who in cooperation with D. M. Mesner explained their connection to linear algebra \cite{Bos_Mesner}. 
An SRG is a regular graph which can be described by four parameters: its number of vertices ($n$), its degree ($k$), the number of common neighbors for any pair of adjacent vertices ($\lambda$) and the number of common neighbors for any pair of non-adjacent vertices ($\mu$).
\par The strongest techniques in this area uses the theory of graph spectra. 
This fact is illustrated by the influential paper of J. Hoffman and R. Singleton \cite{HS} where authors show that there are only four feasible values for the degree of Moore graphs with diameter two (that is SRGs with $\lambda=0$ and $\mu=1$), namely, $2$, $3$, $7$ and $57$.
Hoffman an Singleton have successfully constructed three of these, but the existence of the Moore graph of valency $57$ remains unknown. Actually, this question became a famous open problem in graph theory. For illustration we refer to the monograph of A. E. Brouwer and W. H. Haemers \cite{BrHe}.
\par 
Spectral methods provide extremely strong necessary conditions for the parameters of an $SRG$. As a consequence, researchers usually apply combinatorial techniques on very limited systems of parameters sets. A notable exception is an important connection between SRGs and design theory \cite{Bos_Mesner}, \cite{Bos}. 
\par A useful combinatorial technique is the study of interplay of some small configurations in a given object
and deriving of the number of their occurrences can help to find new conditions for the existence of the object. This method is used in \cite{4conf} where $4$-cycle systems are explored.
\par We focus on theoretical and algorithmic tools for determining the numbers of induced subgraphs in SRGs from their parameters and on further applications of this method. Although the information about small induced subgraphs was efficiently used in the study of graphs with concrete parameter sets, for example in \cite{BrHe} and \cite{klin}, we are not aware of general results in this direction. Similar methods were used in the study of graphs satisfying the so called $t$-vertex condition.
\par For a given order $o$, we develop an algorithm that produces relations between numbers of induced subgraphs of orders $o$ and $o-1$ for an arbitrary SRG. Moreover, the algorithm can be easily modified to consider only triangle-free SRGs or Moore graphs. The relationships are given by a system of linear equations.
\par Our main result is an extensive analysis of solutions of these systems. For instance, it turns out that in a putative Moore graph $\Gamma$ of valency $57$, the number of induced subgraphs isomorphic to a given graph on $10$ vertices depends only on the number of induced subgraphs isomorphic to the Petersen graph in $\Gamma$. Among the applications of our methods, there are new results about automorphisms of order $7$ in a Moore graph of valency $57$ and new bounds on the numbers of induced $K_{3,3}$ in triangle-free SRGs with small number of vertices.
\par The thesis is organized as follows. The summary of basic notions and results about SRGs and triangle-free SRGs is discussed in Section 1 and Section 2. Section 3 contains detailed description of our algorithm together with output analysis. The last three sections are devoted to possible applications of our results. In Section 4 we combine ou results about triangle-free SRGs with the method recently developed by Bondarenko, Prymak and Radchenko \cite{klin} to derive new bounds on the number of $K_{3,3}$ in triangle-free SRGs. In Section 5 we improve upon results on automorphism group of Moore graph of valency $57$ from \cite{MacSir}. Finally, in section 6 we discuss extensions of our approach to the study of the $t$-vertex condition.

\newpage
\section{Strongly regular graphs}
In this chapter we introduce a formal definition of SRGs as well as some other terms which we will use later in the thesis. We also present some important results from this part of graph theory. 
\subsection{Basic definitions and known results}
Among the possible ways to define an SRG, we propose the following one:
\begin{definicia}
\label{def_SRG}
A \emph{strongly regular graph} with parameters $n, k, \lambda$ and $\mu$ is a k-regular graph on $n$ vertices with following properties:
\begin{enumerate}
\item Any two adjacent vertices have exactly $\lambda$ common neighbours.
\item Any two non-adjacent vertices have exactly $\mu$ common neighbours.
\end{enumerate}
We use the notation $srg(n,k,\lambda, \mu)$ for such a graph.
\end{definicia}

In fact, it is not necessary to use all four parameters. The following proposition shows a relations between them.

\begin{prop}\label{prop-par}
Let us consider $\Gamma=srg(n,k,\mu,\lambda)$. Then the following identity holds:
$$(k-\lambda-1)k=(n-k-1)\mu$$
\end{prop}

\begin{proof}
Let $v$ be some vertex in $\Gamma$. We denote the set of all its neighbours by $N$ and the set of remaining vertices by $R$. It is easy to see that the cardinalities of $N$ and $R$ are $k$ and $n-k-1$, respectively. Since $\Gamma$ is an $SRG$, every edge lies in $\lambda$ triangles. Therefore each vertex of $N$ has $\lambda$ neighbours in $N$ and $k-\lambda-1$ neighbours in $R$. On the other hand, each vertex in $R$ has $\mu$ common neighbours with $v$. Hence, we obtain the required equation.
\end{proof}

Proposition \ref{prop-par} says that the four parameters describing the structure of a given $SRG$ are not independent. It is sufficient to know values of parameters $k$, $\lambda$ and $\mu$. This triple already determines the value of the remaining parameter $n$.\par
The following proposition contains an obvious but important information. If we have some $SRG$, it immediately allows us to obtain another $SRG$.
\newpage
\begin{prop}
Let $\Gamma$ be an $srg(n,k,\lambda, \mu)$. The complement of $\Gamma$ is also strongly regular, with parameters:
\begin{enumerate}[ ]
	\item $\overline{n}=n$
	\item $\overline{k}=n-k-1$
	\item $\overline{\lambda }=n+\mu-2(k-1)$
	\item $\overline{\mu }=n+\lambda-2k$
	\end{enumerate}
\end{prop}

Note that any disconnected SRG has to be a disjoint union of complete graphs of the same order and it is considered to be of little interest. Since the complementary graph of any SRG is also strongly regular we consider only SRGs for which both their complements and themselves are connected. SRGs of this kind are called primitive. If some SRG is not primitive we call it imprimitive. Note that SRGs with disconnected complements are complete multi-partite graphs.

\begin{figure}[h!]
\begin{center}
\includegraphics[width=0.5\textwidth]{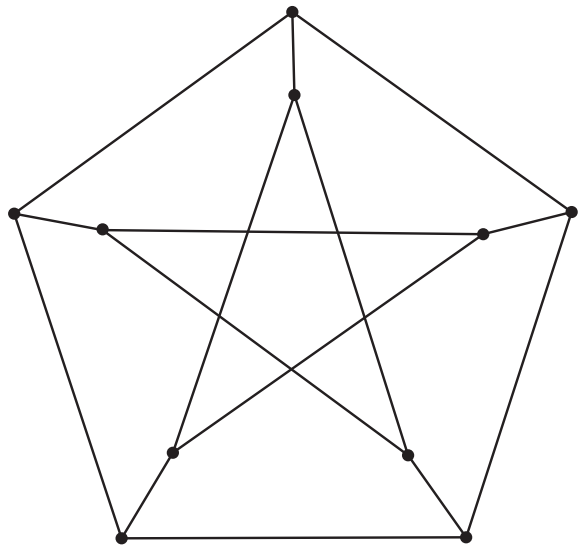}
\caption{The $srg(10,3,0,1)$ (Petersen graph)}
\label{petersen}
\end{center}
\end{figure}

\begin{figure}[h!]
\begin{center}
\includegraphics[width=0.85\textwidth]{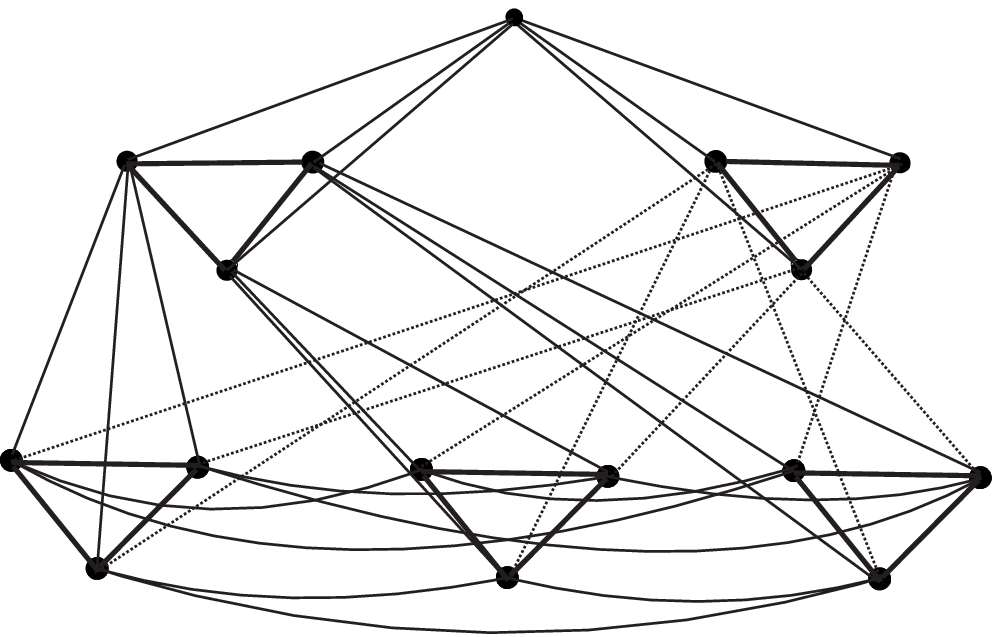}
\caption{A $srg(16,6,2,2)$ ($L_2(4)$ graph)}
\label{16_trojuh}
\end{center}
\end{figure}
%\newpage

A famous example of a primitive SRG is the Petersen graph (see figure \ref{petersen}) with parameters $(10,3,0,1)$. In fact, Petersen graph is the only $SRG$ with this parameter set. However, there are feasible parameter sets $n$, $k$, $\lambda$, $\mu$ with multiple SRGs. To illustrate this fact we choose two non-isomorphic SRGs with parameters $(16,6,2,2)$. Let us pick a vertex $v$ of this graph and denote the set of its neighbours by $N$. Since $\mu=2$, any edge lies in two different triangles. This condition has to be satisfied also by edges, which contain the vertex $v$. It follows that $N$ has to induce a union of cycles. Figure \ref{16_trojuh} presents the case when $N$ consists of two triangles.
The second possibility to arrange vertices of $N$ is to let them form a 6-cycle as shown in figure \ref{16_sestuh}. Note that if the set of neighbours for some vertex in a $srg(16,6,2,2)$ consists of two triangles then it has to be so for any vertex of this graph. Therefore there are no more SRGs with this set of parameters.

\begin{figure}[h!]
\begin{center}
\includegraphics[width=0.85\textwidth]{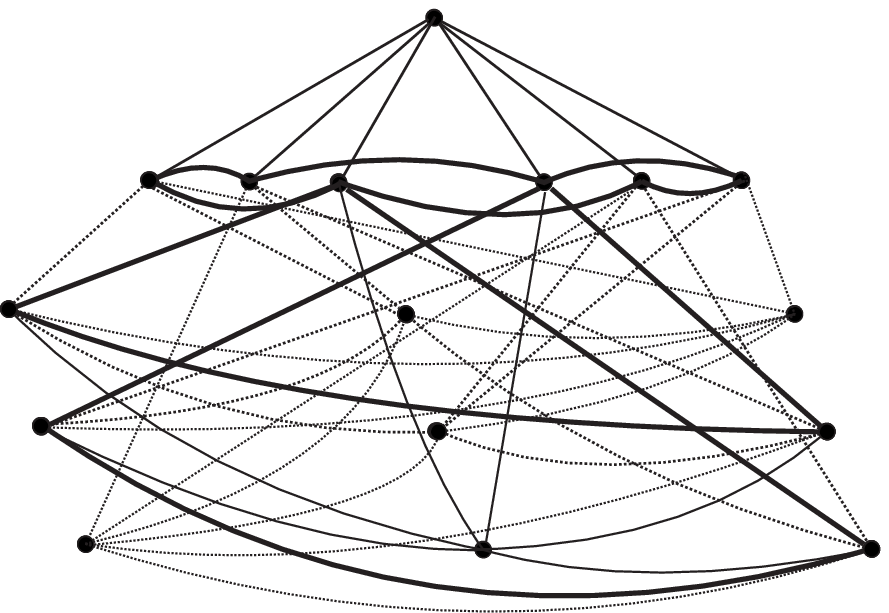}
\caption{$srg(16,6,2,2)$ (Shrikhande graph)}
\label{16_sestuh}
\end{center}
\end{figure}

There are even more extreme cases than $srg(16,6,2,2)$. For example, there exists a unique $srg(36,10,4,2)$. On the other hand, the computation of McKay and Spence \cite{mcs} shows that the number of graphs with parameters $(36,15,6,6)$ is 32548. This pattern continues as an infinite class of SRGs. There is a unique $srg(m^2,2(m-1),m-2,2)$ but more than exponentially many $srg(m^2,3(m-1),m,6)s$. This suggests that depending on the choice of parameters, SRGs can behave in either a highly structured or an apparently random manner.

\begin{figure}[h!]
\begin{center}
\includegraphics[width=1.0\textwidth]{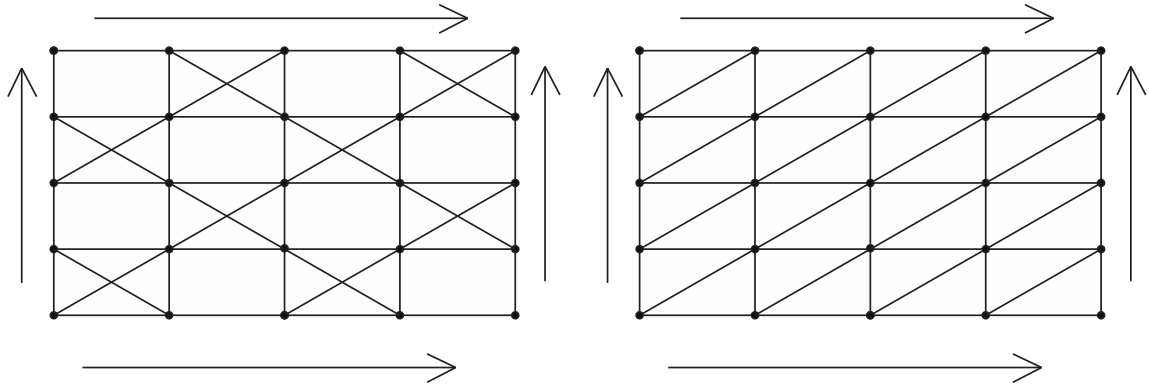}
\caption{A $srg(16,6,2,2)$ drown on the torus}
\label{petersen}
\end{center}
\end{figure}

\begin{definicia}
\emph{The adjacency matrix} for a graph with $n$ vertices is an $n \times n$ matrix whose $(i, j)^{th}$ entry is $1$ if the
$i^{th}$ vertex and $j^{th}$ vertex are adjacent, and 0 if they are not.\\
When speaking about graph spectra we will refer to the spectrum of its adjacency matrix.
\end{definicia}

Let $A$ be an adjacency matrix of some graph $\Gamma$. Then $\Gamma$ is an SRG with parameters $n$, $k$, $\lambda$ and $\mu$ if and only if $A$ satisfies
\begin{enumerate}[i)]
\item $A$ has row sum $k$
\item If $J$ denotes the all-one matrix and $I$ the identity matrix, then 
\begin{equation} \label{eq:vlc}
A^2 - (\lambda-\mu)A - (k-\mu)I=\mu J
\end{equation}
\end{enumerate}

Using the equation \ref{eq:vlc} together with basic knowledge of matrix theory we can obtain the following proposition.
 
\begin{veta}\label{prop:param}
The spectrum of any $SRG(n,k,\lambda,\mu)$ consists of eigenvalues $k$, $r$ and $s$, where

\begin{equation*}
r,s=\frac{\lambda-\mu \pm\sqrt{(\lambda-\mu)^2+4(k-\mu)}}{2}
\end{equation*}
Their respective multiplicities are $1$, $f$ and $g$ with $f$ and $g$ satisfying

\begin{equation*}
f,g=\frac{1}{2}\left( n-1\mp\frac{2k+(n-1)(\lambda-\mu)}{\sqrt{(\mu-\lambda)^2+4(k-\mu)}}\right)
\end{equation*} 
\end{veta}
\begin{proof}\cite{PJC}
Let $A$ be the adjacency matrix of $\Gamma=SRG(n,k,\lambda,\mu)$. We already know that $A$ satisfies the following

$$A^2 - (\mu-\lambda)A - (k-\mu)I=\mu J$$

Since $A$ has row sum $k$, the vector of ones $j$ is an eigenvector of $A$ with eigenvalue $k$. Let $v$ be some other eigenvector of $A$ with corresponding eigenvalue $x$.\\
The eigenvector $v$ has to be orthogonal to $j$. It follows that
\begin{align*}
A^2v - (\mu-\lambda)Av - (k-\mu)Iv&=\mu Jv\\
x^2 - (\mu-\lambda)x - (k-\mu)x&=0
\end{align*}
From this we can obtain the eigenvalues $r$, $s$ and (using the fact that trace of $A$ equals zero) also their respective multiplicities $f$, $g$. 
\end{proof}

Let $n$, $k$, $\lambda$ and $\mu$ be non-negative integers. Then a $srg(n,k,\lambda, \mu)$ exists only if the expression

$$\frac{(n-1)(\mu-\lambda)-2k}{\sqrt{(\mu-\lambda)^2+4(k-\mu)}}$$

is an integer with the same parity as $n-1$. The parameter sets that can be obtained from this criterion are called feasible.\\
On the basis of this result we can classify SRGs into two classes.

\begin{enumerate}[I.]
\item Conference graphs where $(n-1)(\mu-\lambda)-2k=0$. This implies that $\lambda=\mu-1$, $k=2\mu$, and $n=4\mu+1$. 
This class consists of SRGs with the same parameters as their complements. It is known that they exist only if $n$ is a sum of two squares.
\item Graphs for which $(\mu-\lambda)^2+4(k-\mu)=d^2$, where $d\in\mathbb{N}$ and divides\\
$(n-1)(\mu-\lambda)-2k$ with quotient congruent to $n-1\mod{2}$.
\end{enumerate}

A table of feasible values of parameters of SRGs on up to $1300$ vertices can be found on the home page of A. E. Brouwer \cite{table2}.

\subsection{Other feasibility criteria}
Except for the integral criteria there are few other restrictions on parameters of SRGs. In this section we describe two of these which can also be determined from properties of graph spectra, namely, the Krein condition and the absolute bound  \cite{PJC}. Some authors consider these two also as criteria of feasibility for SRGs.

\subsubsection{The Krein condition}

One of the most important conditions satisfied by parameters of an SRG is the Krein condition. It was first proved by Scott \cite{Scott}, using a result of Krein \cite{krc} from harmonic analysis.

\begin{prop}[The Krein condition]\label{Krein} \cite{Scott}
Let $\Gamma$ be an arbitrary $srg(n,k,\lambda,\mu)$ with eigenvalues $k$, $r$ and $s$. Then the following holds:

\begin{align*}
(r+1)(k+r+2rs)&\leq (k+r)(s+1)^2\\
(s+1)(k+s+2rs)&\leq (k+s)(r+1)^2
\end{align*}
\end{prop}

This condition rules out several possible parameter sets for $SRGs$. For example, a $srg(144,78,52,30)$ is excluded by the first Krein condition and a $srg(28,9,0,4)$ does not satisfy the second one. \par
An $SRG$ is called a \textit{Krein graph} if the equality holds in one of Krein's conditions. Let us fix a vertex $v$ in some Krein graph. Then the subgraphs induced by neighbours of $v$ and by its non-neighbours are both strongly regular. These are called \textit{linked graphs} \cite{subcon}. \\
For example, in the Higman-Sims graph ($srg(100,22,0,6)$) the subgraph induced by non-neighbours of any vertex is the Mesner graph ($srg(77,16,0,6)$).

In this section we consider even more specific subclass of these remarkable structures, namely, triangle-free Krein graphs.

\begin{lema}\label{krein} \cite{cam_lintCR}
Let $\Gamma$ be a triangle-free Krein graph with $2<\mu<k$. Then, $\Gamma$ has a parameter set $((r^2+3r)^2,r^3+3r^2+r,0,r^2+r)$ if and only if any triple of vertices inducing a $\overline{K}_3$ in $\Gamma$ have exactly $r$ common neighbours. 
\end{lema}

Note that the value $r$ from the above lemma fully describes a parameter set for a given Krein graph, such that one can simply use notation $Kr(r)$. Some authors use the notation $NL_r(r^2+3r)$ instead. The reason is that Krein graphs, as we describe them, are members of the family of negative Latin square graphs, $NL_r(n)$. It is easy to show that $NL_r(n)$ is a Krein graph if and only if $n=r^2+3r$.\par
It is well known that for $r\in\{1,2\}$ there are unique graphs $Kr(r)$, the Clebsch graph ($srg(16,5,0,2)$) and the Higman-Sims graph. There are no more known examples of $Kr(r)$ but the following statement has to be satisfied. 

\begin{prop}\cite{GavMak}\label{GavMak}
\begin{itemize}
	\item There is no $Kr(3)$.
	\item $Kr(r)$ does not contain $K_{r,r}$ as an induced subgraph for $r\geq 9$
\end{itemize}
\end{prop}

\subsubsection{The absolute bound}\label{abs-rep}
Let $\Gamma$ be an $srg(v,k,\lambda,\mu)$ and let $A$ be its adjacency matrix. We will use the notation $l$ for the number of vertices that are not adjacent to a given vertex. As it was already shown in the theorem \ref{prop:param} the matrix $A$ has exactly three eigenvalues:
\begin{enumerate}
	\item A trivial eigenvalue $k$ of multiplicity $1$.
	\item A positive eigenvalue $r$ of multiplicity $f$.
	\item A negative eigenvalue $s$ of multiplicity $g$.
\end{enumerate}

Any vertex $v_i$ of $\Gamma$ can be identified with the vector constructed by the $i^{th}$ row of the adjacency matrix $A$. As we consider only primitive SRGs, all eigenvalues of $A$ are non-zero. It follows that the set of vertices represented by vectors $\{v_1,\dots,v_n\}$ forms a basis of $\mathbb{R}$. Let $V_0$, $V_1$, $V_2$ denote eigenspaces for the eigenvalues $k$, $r$, $s$ respectively, and let $E_1$, $E_2$ and $E_3$ be their orthogonal projection onto the eigenspace $V_i$. The adjacency matrix $A$ can be then rewritten as the sum 

$$A=kE_0+rE_1+sE_2.$$

Now, let us restrict our attention to one of the non-trivial eigenspaces, $V_2$. We will consider here (normalized) projections of vectors $v_i$ into this eigenspace:

$$x_i=\frac{v_iE_2}{||v_iE_2||}$$

It is easy to verify the following property of these vectors.

\begin{prop} 
Let $\Gamma$ be a primitive $srg(n,k,\lambda, \mu)$ wit eigenvalues $k>r>s$ and for $i\in\{1,\dots,n\}$ let $x_i$ denote orthogonal projection of the vertex $v_i$ onto the eigenspace corresponding to the eigenvalue $s$. Then the inner product of the vectors $x_i$ and $x_j$ satisfies:

$$
					\begin{array}{ll*{3}{l}}
					x_i\cdot x_j &=& \left\{\begin{array}{ll}
									1 & \textrm{ if } i=j,\\
									p & \textrm{ if } i\sim j,\\
									q & \textrm{ otherwise, }
									\end{array}\right.
					\end{array}
$$

where $p=\frac{s}{k}$ and $q=-\frac{(s+1)}{(n-k-1)}$.\\
Moreover, as $\Gamma$ is connected and not complete multipartite, $x_i\neq x_j$ for $i\neq j$.
\end{prop}

The proposition above says that all the projected vectors lie on a unit sphere in a $g$-dimensional Euclidean vector space with a specific distance from each other. Delsarte et al. prove the following result:

\begin{veta}[The absolute bound \cite{del}]
Let $g\neq 1$ be the multiplicity of an eigenvalue of a $srg(n,k,\lambda,\mu)$. Then

$$n\leq {{g+2}\choose{2}}-1$$

\end{veta}

The absolute bound can be translated into an inequality on parameters of an SRG and gives a strong tool for handling feasible parameters obtained by spectral techniques. There are several parameter sets that have been ruled out using the absolute bound. Eight of them were for graphs on up to $100$ vertices, which illustrates the power of this method. The quadruple $(50,21,4,12)$ and the triangle free case $(64,21,0,10)$ are examples of these cases.

%\begin{prop}
%In the case of Krein graphs (lemma \ref{krein})  with parameters\\
%$((r^2+3r)^2,r^3+3r^2+r,0,r^2+r)$ the absolute bound is attained.
%\end{prop}

\subsubsection{Spherical $t$-designs }

\begin{definicia}
Let $X=\{x_1,x_2,\dots,x_n\}$ be a set of vectors lying on a unit sphere $S^{g-1}$ in an Euclidean space $\mathbb{R}^g$. Then, X is called a spherical $t$-design if, for any polynomial function $F$ of degree at most $t$, we have:

$$\frac{1}{n}\sum_{i=1}^{n}F(x_i)=\frac{1}{vol(S^{g-1})}\int_{^{g-1}}F(x)dx$$
\end{definicia}

In other words, the finite set $X$ "approximates the sphere up to degree $t$".\\
For small $t$, there is a mechanical interpretation. Place unit masses at the points of $X$. Then $X$ is a spherical $1$-design if and only if the centre of mass is at the origin, and is a spherical $2$-design if, in addition, the inertia ellipsoid is a sphere (that is, the moments of inertia are all equal and the products of
inertia are zero).

\begin{veta}
Let $\Gamma$ be a connected primitive SRG and let $X$ be the normalized projection of the vectors representing its vertices onto a non-trivial eigenspace. Then

\begin{enumerate}[i.)]
	\item $X$ is a spherical $2$-design.
	\item $X$ is a spherical $3$-design iff the Krein condition corresponding to this eigenspace is attained.
	\item $X$ is a spherical $4$-design iff the absolute bound is attained.
	\item $X$ is never a spherical $5$-design.
\end{enumerate}
\end{veta}
\newpage
 \section{Triangle-free SRGs}

As the last part of previous section indicated, the family of triangle-free strongly regular graphs (tfSRGs) is interesting in many ways. This section focusses exactly on these graphs. The triangle-free condition in our definition means that the parameter $\lambda$ of such SRG equals zero. Simplest examples of such graphs are complete bipartite graphs $K_{m,m}$. As these graphs are imprimitive we will not consider them further on.
\par The integral criterion for parameter $\mu \notin \{2,4,6\}$ of tfSRG leaves only finitely many admissible values of $k$. Since there are infinitely many values of $k$ which satisfy the integral condition \cite{elz} for each $\mu\in\{2,4,6\}$, tfSRGs are together represented by an infinite family of feasible parameter sets. 
\par On the other hand there are only seven known examples of tfSRGs. Three of them are members of the famous family of Moore graphs with diameter two. Their sets of parameters are $(5,2,0,1)$ (pentagon), $(10,3,0,1)$ (Petersen graph) and $(50,7,0,1)$ (Hoffman-Singleton graph). This family contains only one more parameters set, $(3250,57,0,1)$, but there is no known example of such a graph.
Parameters of the remaining four known tfSRGs are $(16,5,0,2)$ (Clebsch graph), $(56,10,0,2)$ (Sims-Gewirtz graph), $(77,16,0,4)$ (Mesners $M_{22}$ graph) and $(100,22,0,6)$ (Mesner, Higman-Sims graph). Each of these graphs is uniquely determined by its parameters and it is unknown whether there are more tfSRGs. 

\begin{table}[h!]
\begin{center}
\begin{tabular}{r|| c| c| c| c| c| c| c}
			&	Pentagon&	Petersen	&	Clebsch	&	HoSi	&	SimGe	&	Mesner		&	HiSim\\
\hline
\hline
Pentagon	&	$1$		&	$12$		&	$192$	&	$1260$		&	$8060$		&	$88704$		&	$443520$\\
\hline
Petersen	&			&	$1$			&	$16$	&	$525$		&	$13440$		&	$1921920$	&	$35481600$\\
\hline
Clebsch		&			&				&	$1$		&	$0$			&	$0$			&	$0$			&	$924000$\\	
\hline
HoSi		&			&				&			&	$1$			&	$0$			&	$0$ 		&	$704$\\	
\hline
SimGe		&			&				&			&				&	$1$			&	$22$		&	$1030$\\	
\hline
Mesner		&			&				&			&				&				&	$1$			&	$100$\\	
\hline
HiSim		&			&				&			&				&				&				&	$1$\\	
\end{tabular}
\end{center}
\caption{Number of tfSRGs inside of tfSRGs \cite{matan}}
\label{matan}
\end{table}

Note that all the known primitive $tfSRGs$ (except Higman-Sims graph) can be found as induced subgraphs of some larger one \cite{matan}. Known results are presented in the table \ref{matan}. The remarkable fact is that any known tfSRG does not contains $K_{3,3}$ as induced subgraph. On the other hand, each $Kr(r)$ with $r>2$ has to contain an induced $K_{3,3}$. According to lemma \ref{krein}, their number is determined uniquely. Hence, the question about induced $K_{3,3}$ in $tfSRGs$ are natural.\\
Obviously, the missing Moore graph does not contain an induced $K_{3,3}$. Yet there is another interesting open problem. It is not known whether it can contain an induced Petersen graph.
We have observed that feasible parameters for many tfSRGs of small orders have a specific form. They can be described as follows:
\begin{center}
\begin{align*}
k	&=N(2N+1+M(N+1)))=N(N(M+2)+(M+1)),\\
\mu	&=N(M+1),\\
s	&=-N(M+2),\\		
\end{align*}
\end{center}
where $N\in\mathbb{N}$, $M\in\mathbb{N}_0$ and $M\leq N$. This is summarized in Table \ref{tfsrg}. The cases that are ruled out by integral criterion are denoted by $*$.\\
An interesting observation is that the last two parameter sets of each column (with $N=M$ and $N=M+1$) represent a pair of linked graphs. The larger graph of a pair is always one of the Krein graphs described in the lemma \ref{krein}. According to the result of N. L. Bigs \cite{tfSRG}, there are no other pairs of such graphs:

\begin{veta}
The parameters of a linked pair $\Gamma$, $\Gamma'$ of tfSRGs must be of the form
\begin{align*}
k	&=N(N^2+3N+1),\\
\mu	&=N(N+1),\\
k'	&=N^2(N+2),\\
\mu'&=N^2,		
\end{align*}
where $N$ is a positive integer.
\end{veta}

Since there is no graph for parameters $(324,57,0,12)$ (proposition \ref{GavMak}), the existence of the smaller graph $\Gamma'$ of a linked pair may not guarantee the existence of the bigger one.

\begin{table}[h!]
\begin{center}
\scriptsize{
\begin{tabular}{r|| c| c| c| c| c| c}
			&	N=1				&	2				&	3				&	4				&	5				&	6\\
\hline
\hline
M=0			&	$(10,3,0,1)$	&	$(56,10,0,2)$	&	$(162,21,0,3)$	&	$(352,36,0,4)$	&	$(650,55,0,5)$	&	$(1080,78,0,6)$\\
\hline
1			&	$(16,5,0,2)$	&	$(77,16,0,4)$	&	$(210,33,0,6)$	&	$(442,56,0,8)^*$&	$(800,85,0,10)$	&	$(1311,120,0,12)^*$\\
\hline
2			&					&	$(100,22,0,6)$	&	$(266,45,0,9)$	&	$(552,76,0,12)$	&$(990,115,0,15)^*$	&	$(1612,162,0,18)^*$\\
\hline
3			&					&					&	$(324,57,0,12)$	&	$(667,96,0,16)$	&$(1190,145,0,20)$	&	$(1930.5,204,0,24)$\\
\hline
4			&					&					&					&	$(784,116,0,20)$&$(1394,175,0,25)$	&	$(2256,246,0,30)$\\
\hline
5			&					&					&					&					&$(1600,205,0,30)$	&	$(2585,288,0,36)$\\
\hline
6			&					&					&					&					&					&	$(2916,330,0,42)$\\
\end{tabular}
}
\end{center}
\caption{tfSRGs (* - $g$ is not integral)}
\label{tfsrg}
\end{table} 

There are only nine feasible parameter sets with $n\leq 1300$ that are not included in table \ref{tfsrg}, namely $(5,2,0,1)$, $(50,7,0,1)$, $(176,25,0,4)$, $(352,26,0,2)$, $(392,46,0,6)$, $(638,49,0,4)$, $(704,37,0,2)$, $(1073,64,0,4)$ and $(1276,50,0,2)$.
\newpage
\section{Small subgraphs in SRGs}

Let $\Gamma$ be a $srg(n,k,\lambda,\mu)$. For every graph $G$ on $t$ vertices there is a value representing the number of copies of $G$ as an induced subgraph in $\Gamma$.
This value is, for some $G$, constant and depends only on the parameters $n$, $k$, $\lambda$ and $\mu$. For example, it is easy to derive a formula for the number of triangles in $\Gamma$. We know that any pair of adjacent vertices $(x,y)$ in $\Gamma$ has $\lambda$ common neighbors. The vertices $x$ and $y$ together with any of these neighbors induce a triangle in $\Gamma$. Therefore the total number of induced triangles in $\Gamma$ equals $\frac{nk}{2}\frac{\lambda}{3}$. If we continue with calculating numbers of other $3$-vertex subgraphs we obtain that $\Gamma$ contains $\left[{{n}\choose{2}}-\frac{nk}{2}\right]\mu$ subgraphs isomorphic to $K_{1,2}$. The number of induced subgraphs on three vertices with one edge is $\frac{nk}{2}(n-2k+\lambda)$. All the remaining subgraphs on three vertices are isomorphic to $\overline{K_3}$ and their number together with the previous numbers sum up to ${{n}\choose{3}}$.

\begin{lema}\label{3v}
The number of occurrences of any $3$-vertex graph as induced subgraph in $srg(n,k,\lambda,\mu)$ is determined uniquely by the parameters $n$, $k$, $\lambda$ and $\mu$. The exact numbers are summarized in the following table:
\begin{table}[h!]    
\begin{center}
%\scriptsize
\scalebox{1}{
        \begin{tabular}{c || c}
        	$G$ of order $3$				& The number of $G$ in $srg(n,k,\lambda,\mu)$\\
        \hline
        \hline
  \includegraphics[scale=0.15]{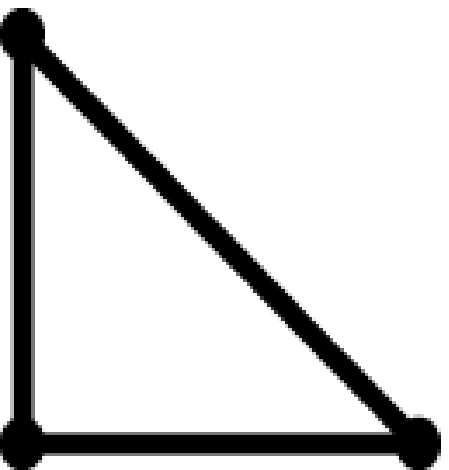}&      $\frac{nk}{2}\frac{\lambda}{3}$  \\
		\hline
  \includegraphics[scale=0.15]{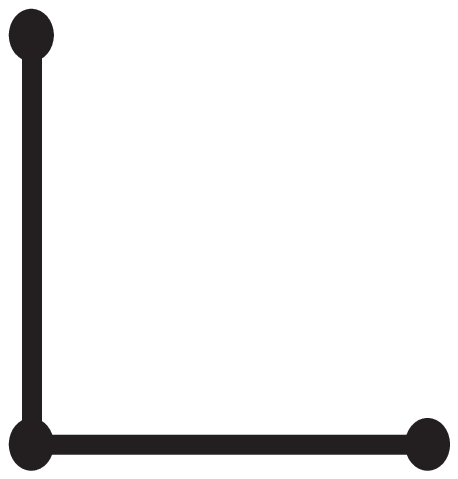}&	$\left[{{n}\choose{2}}-\frac{nk}{2}\right]\mu$	\\
        \hline
  \includegraphics[scale=0.15]{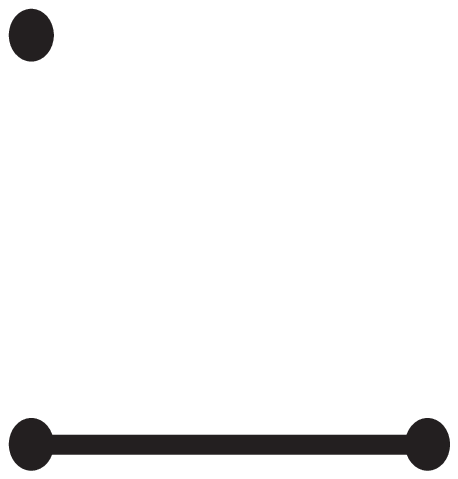}&		$\frac{nk}{2}(n-2k+\lambda)$	\\
        \hline
  \includegraphics[scale=0.15]{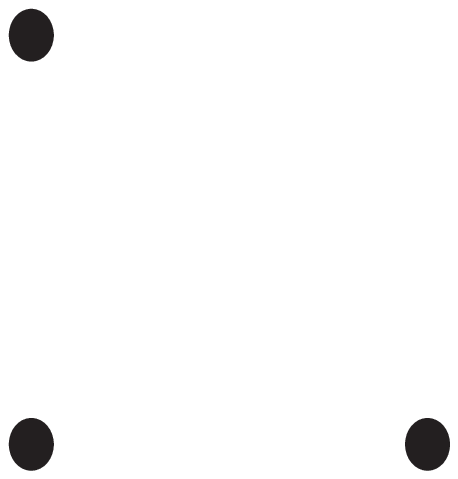}&	${{n}\choose{3}}-{{n}\choose{2}}\mu-\frac{nk}{2}\left[\frac{\lambda}{3}-\mu+(n-2k+\lambda)\right]$	\\
        \end{tabular}
}
\end{center}
    \end{table}
\end{lema}

With increasing the order of subgraphs, graphs whose numbers of copies in $\Gamma$ are not determined uniquely by $n$, $k$, $\lambda$ and $\mu$ start to appear. To illustrate this we look back into the two $srg(16,6,2,2)$ on figures \ref{16_trojuh}, \ref{16_sestuh}. The first of them contains eight $K_4$ subgraphs, while in the second graph there are none. 
% In general, if we know the value of $P_{K_4}$ for some $SRG$ then $P_G$ of any $G$ on four vertices is already determined uniquely \cite{tc}. 
%\par One part of our work is to find a method for identifying dependencies between numbers of $t$-vertex subgraphs for any tfSRG for small values of $t$. As $\lambda=0$, there are not so many dependencies between numbers of $t$-vertex subgraphs as in arbitrary case. The reason behind this is the fact that some of graphs on $t$ vertices cannot appear as induced subgraphs in tfSRG which also has effect on numbers of other $t$-vertex subgraphs.
\par In this chapter we begin our analysis of subgraphs in $tfSRGs$. However, the method can be extended also for general SRGs. 

To show that this strategy can be successful we refer to \cite{BrHe}, where the authors significantly use the number of induced $C_4$ to prove uniqueness of Gewirtz graph. 

\begin{prop}\label{prop:pre4}
Let $\Gamma$ be a $srg(n,k,0,\mu)$ and let $G$ be a graph on at most $4$ vertices. Then the number of induced subgraphs isomorphic to $G$ depends only on parameters $n$, $k$ and $\mu$.
\end{prop}

\begin{proof}
Let $G$ be a subgraph of $\Gamma$ with vertex set $V_G$.
The case where $|V_G|=3$ is already solved in lemma \ref{3v}.
Therefore we only need to show that the proposition holds for any $G$ with $|V_G|=4$. It is easy to see that there are seven non-isomorphic triangle-free graphs on four vertices. Let us denote them $G_1,G_2,\dots,G_7$ in correspondence with the figure \ref{4_vertex}. Then their respective occurrences in $\Gamma$ will be represented by $P_1, P_2, \dots, P_7$.

\begin{figure}[h!]
\begin{center}
\includegraphics[width=1\textwidth]{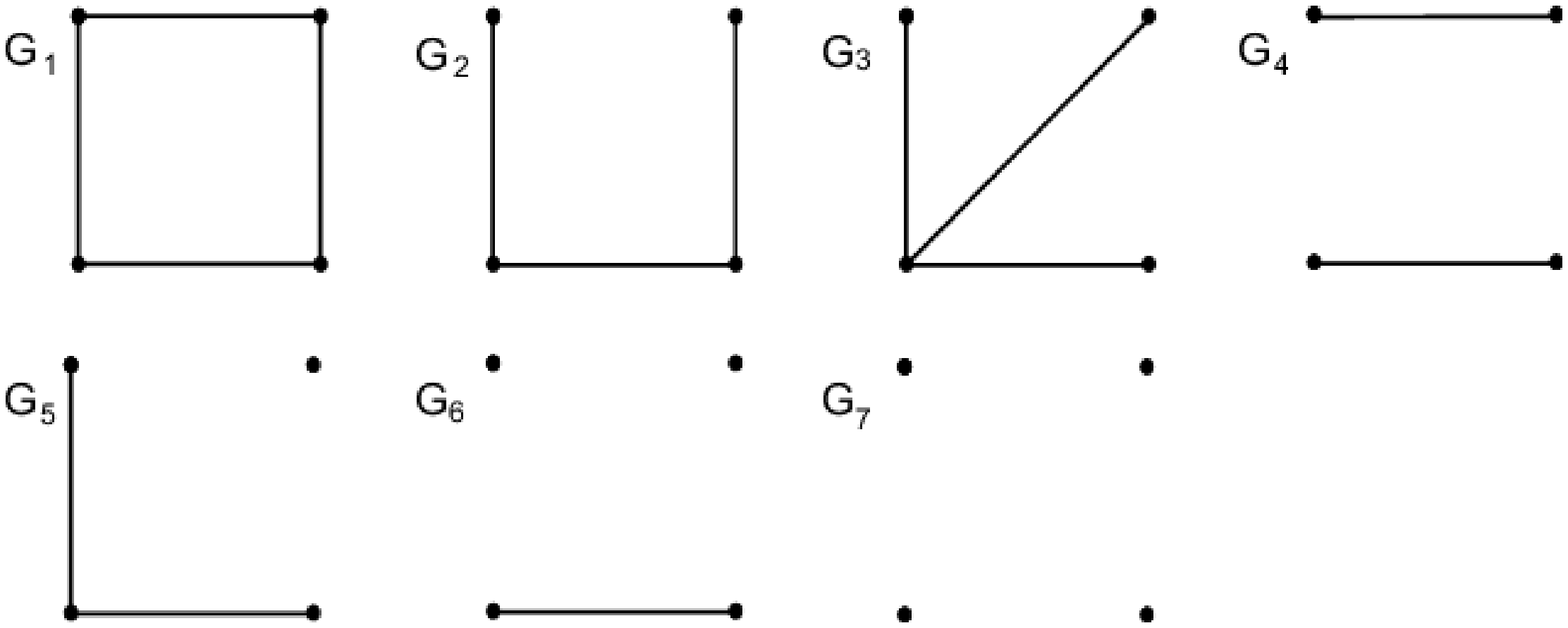}
\caption{ }
\label{4_vertex}
\end{center}
\end{figure}

When we compute values $P_i$ we use properties of all possible graphs $G_i-\{v\}$, where $v$ is some vertex of $G_i$. Since the order of $G_i-\{v\}$ is for any $v$ equal $3$ its number of copies is determined by parameters of $\Gamma$ (Lemma \ref{3v}). There are only three triangle-free graphs on three vertices, namely $K_{1,2}$, a graph with a single edge, and $\overline{K_3}$. Let us denote the numbers copies of these graphs in $\Gamma$ as $a$, $b$ and $c$, respectively.\\
Let $H$ be a fixed subgraph of $\Gamma$ isomorphic to $K_{1,2}$. Its vertex set $V_H=\{v_1,v_2,v_3\}$ is labeled such that $v_1$ represents a vertex of order $2$ in $H$. Now we can derive numbers of triangle-free graphs in $\Gamma$ which can be constructed by adding one more vertex into $K_{1,2}$ induced in $\Gamma$. There are four such non-isomorphic graphs.

\begin{enumerate}
\item The quadrangle ($G_1$). Let us focus on subgraph $H$ of $\Gamma$. There are exactly $\mu-1$ vertices in $\Gamma-H$ such that their addition $V_H$ gives a quadrangle. On the other hand, any quadrangle of $\Gamma$ contains four induced $K_{1,2}$. Therefore we obtain $a(\mu-1)=4P_1$, from which the total number of quadrangles in $\Gamma$ follows.

\item The path of length three ($G_2$). By adding a vertex adjacent either to $v_2$ or to $v_3$ into $V_H$ we create exactly a graph isomorphic to $G_2$. Each of vertices $v_2$, $v_3$ have $k-1$ neighbours in $\Gamma-H$. We have to remember that $\mu -1$ of them form together with $V_H$ a quadrangle and it is necessary to subtract them. As any $G_2$ contains two copies of $K_{1,2}$ it follows that $2(k-1)a-2\cdot 4P_1=2P_2$. Since we already know the value $P_1$, it is easy to derive also $P_2$.

\item The star $K_{1,3}$ ($G_3$). In this case we need to add a vertex adjacent to $v_1$ to $V_H$. Since $\Gamma$ is triangle-free the new vertex cannot have another neighbour in the created graph. There are exactly $k-2$ such vertices in $\Gamma$. Each $K_{1,3}$ contains $K_{1,2}$ three times. Hence, we obtain an equality $(k-2)a=3P_3$.

\item The disconnected case consisting of $K_{1,2}$ and one isolated vertex ($G_5$). Since the total number of $4$-vertex subgraphs containing a fixed $K_{1,2}$ in $\Gamma$ is $n-3$ we have an equality $(n-3)a=4P_1+2P_2+3P_3+P_5$. Hence, we can easily obtain also the remaining value $P_5$.
\end{enumerate}

There remain only three graphs for which we have not determined the numbers of occurrences in $\Gamma$ yet. Recall that the number of copies of 3-vertex graph with a single edge in $\Gamma$ is denoted by $b$. It contains one isolated vertex. Since the valency of $\Gamma$ is $k$ we obtain $kb=2P_2+4P_4$. This already gives the value $P_4$. The equation $(n-3)b=2P_2+4P_4+2P_5+2P_6$, that can be derived similarly to the step 4 for $K_{1,2}$, determines the value $P_6$. 
The last $4$-vertex graph for which the number of occurrences in $\Gamma$ is still unknown is $\overline{K_4}$ ($G_7$). To obtain it we use the equation ${{n}\choose{4}}=\sum_{i=1}^7P_i$ which counts the total number of $4$-vertex subgraphs in $\Gamma$.
The following table summarizes the numbers of induced $G_i$ in $\Gamma$, for each $i\in\{1,2,\dots,7\}$.

\begin{table}[h!]    
\begin{center}
%\scriptsize
\scalebox{1}{
        \begin{tabular}{c || c}
        	$G_i$				& The number of $G_i$ in $srg(n,k,0,\mu)$\\
        \hline
        \hline
  \includegraphics[scale=0.15]{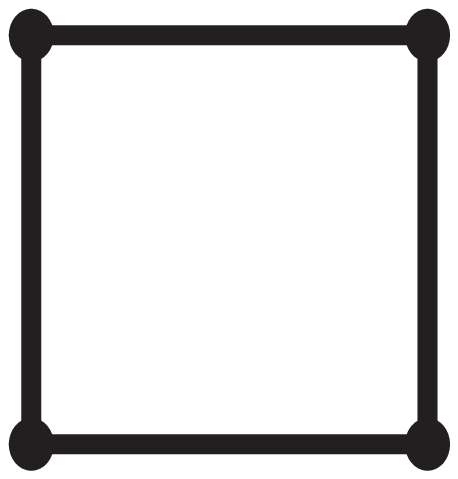}&    $\frac{(\mu-1)\mu}{4}\left[{{n}\choose{2}}-\frac{nk}{2}\right]$  \\
		\hline
  \includegraphics[scale=0.15]{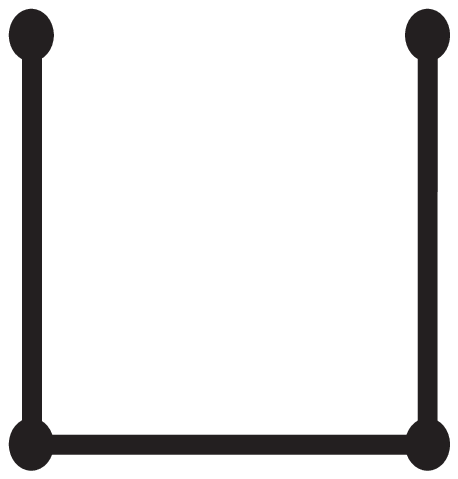}&	$(k-\mu)\mu \left[{{n}\choose{2}}-\frac{nk}{2}\right]$	\\
        \hline
  \includegraphics[scale=0.15]{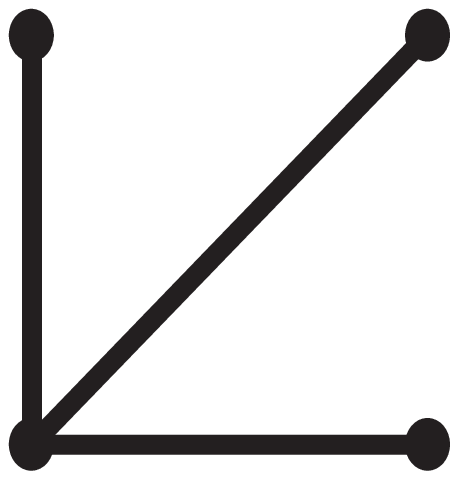}&	$\frac{(k-2)\mu}{3}\left[{{n}\choose{2}}-\frac{nk}{2}\right]$	\\
        \hline
  \includegraphics[scale=0.15]{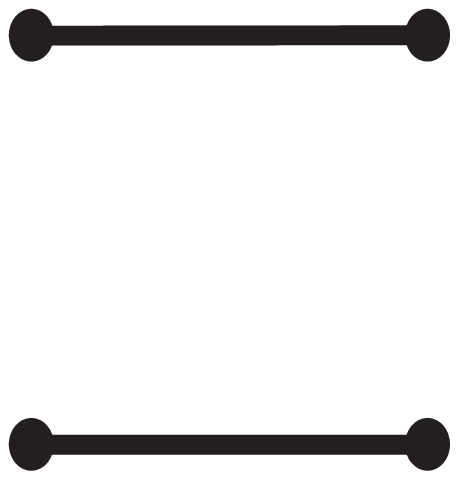}&	$\frac{nk^2}{8}(n-2k+\lambda)  
  													-\frac{(k-\mu)\mu}{2}\left[{{n}\choose{2}}-\frac{nk}{2}\right]$	\\
          \hline
  \includegraphics[scale=0.15]{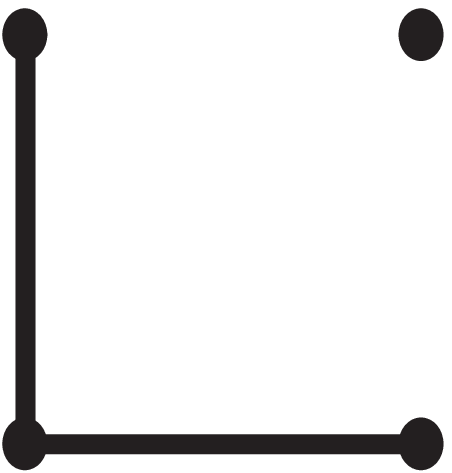}&	$(n-3k+\mu)\mu\left[{{n}\choose{2}}-\frac{nk}{2}\right]$	\\
            \hline
  \includegraphics[scale=0.15]{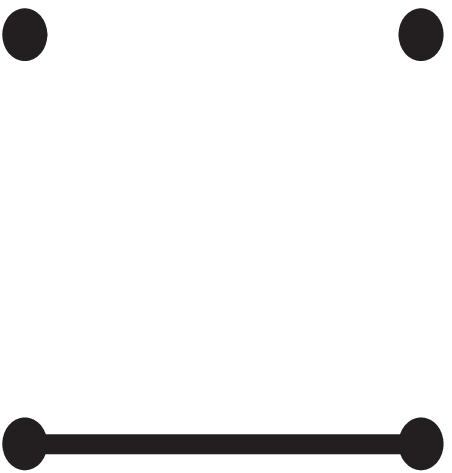}&	$\frac{nk(n-k-3)}{4}(n-2k+\lambda)
  													-(n-3k+\mu)\mu\left[{{n}\choose{2}}-\frac{nk}{2}\right]$	\\
        \hline
  \includegraphics[scale=0.15]{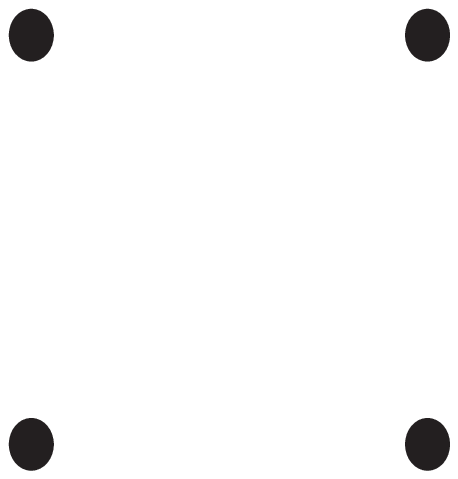}&	${{n}\choose{4}}-\frac{(10k-9\mu-11)\mu}{12}\left[{{n}\choose{2}}-\frac{nk}{2}\right]
  													+\frac{nk(k-2\mu+6)}{8}(n-2k+\lambda)$	\\
        \end{tabular}
}
\end{center}
    \end{table}

\end{proof}

\begin{table}[h!]    
\begin{center}
%\scriptsize
\scalebox{1}{
        \begin{tabular}{l || r|c|c|c|c|c|c|c}

	&	& \includegraphics[scale=0.15]{figures/4_1.eps} & \includegraphics[scale=0.15]{figures/4_2.eps} & \includegraphics[scale=0.15]{figures/4_3.eps} & \includegraphics[scale=0.15]{figures/4_4.eps} & \includegraphics[scale=0.15]{figures/4_5.eps} & \includegraphics[scale=0.15]{figures/4_6.eps} & \includegraphics[scale=0.15]{figures/4_7.eps} \\
&\centering{LHS} &	$G_1$	& $G_2$	&	$G_3$ &	$G_4$	&	$G_5$ &	$G_6$	&	$G_7$\\
			\hline\hline
%\hline
  \includegraphics[scale=0.15]{figures/3_1.eps}&	$(n-3)a$	&	$4$	& $2$	&	$3$ &		&	$1$ &		&	\\
        \hline
  \includegraphics[scale=0.15]{figures/3_1.eps}&	$(k-2)a$ 	&		&		&	$3$	&		&		&		&	\\
		\hline
  \includegraphics[scale=0.15]{figures/3_1.eps}&	$2(k-1)a$	&$8$	& $2$	&		&		&		&		&	\\
        \hline
  \includegraphics[scale=0.15]{figures/3_1.eps}&	$(\mu-1)a$	&	$4$	& 		&		&		&		&		&	\\
        \hline
  \includegraphics[scale=0.15]{figures/3_2.eps}&	$(n-3)b$	&		& $2$	&		&	$4$	&	$2$ &	$2$	&	\\
        \hline
  \includegraphics[scale=0.15]{figures/3_2.eps}&	$2(k-1)b$	&		& $2$	&		&		&	$2$	&		&	\\
        \hline
  \includegraphics[scale=0.15]{figures/3_2.eps}&	$kb$		&		&	$2$	&		&	$4$	&		&		&	\\
        \hline
  \includegraphics[scale=0.15]{figures/3_2.eps}&	$2\mu b$	&		&	$2$	&		&		&		&		&	\\
        \hline
  \includegraphics[scale=0.15]{figures/3_3.eps}&	$(n-3)c$	&		&		&	$1$	&		&	$1$	&	$2$	&$4$\\
        \hline
  \includegraphics[scale=0.15]{figures/3_3.eps}&	$3kc$	&		&		&	$3$	&		&	$2$	&	$2$	&	\\
        \hline
  \includegraphics[scale=0.15]{figures/3_3.eps}&	$3\mu c$	&		&		&	$3$	&		&	$1$	&		&	\\
        \hline
&${{n}\choose{4}}$ &	$1$	&	$1$	&	$1$	&	$1$	&	$1$	&	$1$	&$1$\\ %[0.3cm]
        \end{tabular}
}
\end{center}
\caption{Equations for $4$-vertex subgraphs in $SRG(n,k,0,\mu)$}
\label{4vrcholy}
    \end{table}
\newpage

Each step of the previous proof can be expressed using a single equation. For example, when we have computed numbers of graphs coming from $K_{1,2}$ we have used the equations:

\begin{enumerate}
	\item $a(\mu-1)=4P_1+0P_2+0P_3+0P_4+0P_5+0P_6+0P_7,$
	\item $2(k-1)a=8P_1+2P_2+0P_3+0P_4+0P_5+0P_6+0P_7,$
	\item $(k-2)a=0P_1+0P_2+3P_3+0P_4+0P_5+0P_6+0P_7,$
	\item $(n-3)a=4P_1+2P_2+3P_3+0P_4+1P_5+0P_6+0P_7.$
\end{enumerate}

There is a similar set of equations for each of the three triangle free graphs of order $3$. All of them are written down in the table \ref{4vrcholy}. It is created as a matrix of these equations. The first column represents the left side of equations. Remaining columns refer to the graph in the heading row and appropriate column. Each row is indexed by $3$-vertex graph to which a given equation corresponds. For example, first four rows are dedicated to the graph $K_{1,2}$ in the reverse order like they were constructed in the proof.
The last row displays the total number of $4$-vertex subgraphs in $srg(n,k,0,\mu)$. In the similar way we can understand all rows of the table. It is possible to divide equations of this system into three types (except the last one). First, there are equations for adding an arbitrary vertex. The second type of equations represents relations between graphs that appear by adding a neighbor of some vertex into given $3$-vertex graph. Equations of the last type describe situation of adding a common neighbor of a pair of non-adjacent vertices into given $3$-vertex graph.

Obviously, the solution of this system consists of numbers of occurrences of all seven $4$-vertex triangle-free graphs as induced subgraphs in $\Gamma$. As it was already shown in the proof of \ref{prop:pre4}, not all of these equations are necessary. The reason is that some of them are linear combinations of others. In this particular case, the rank of the system of created equations equals the number of triangle free graphs of order $4$. Hence, we obtain a unique solution. But this does not hold in general. Our aim is to generalize this method for an arbitrary SRG and for subgraphs of higher order where the structure is more complicated. We need therefore to work with all equations. Here, the rank of the system can be lower and its solution not unique. Clearly, the number of new parameters is equal to the number of graphs constructed by this method minus rank of the created system.

\subsection{Equations for a general SRG}
Generally, it is not easy to calculate the numbers of all $t$-vertex subgraphs of some $srg(n,k,\lambda,\mu)$.In the special case where $\lambda$ equals zero, these numbers are for $t>5$ not determined uniquely by $n$, $k$ and $\mu$ but one can observe some dependencies between occurrences of subgraphs. We develop a method for deriving all dependencies between numbers of $t$-vertex subgraphs for any small value of $t$.
\par The method we propose is a generalization of the one in the proof of \ref{prop:pre4}. It is based on a recursion. We already know that it is possible to compute the number of all $3$-vertex subgraphs in $\Gamma=srg(n,k,\lambda,\mu)$ using only numbers of edges and non-edges. 

Suppose that the numbers of all subgraps on $t-1$ vertices in $\Gamma$ are given.
Similarly to the triangle-free case for $t=4$ we use strong regularity to construct special equations for each of $(t-1)$-vertex graph.\\
Let $G$ be a graph on $t-1$ vertices. 
\begin{enumerate}
	\item We have one equation saying which types of subgraphs can be created by adding one arbitrary vertex from $\Gamma$ to subgraphs isomorphic to $G$. The left-hand side of this equation contains the number of occurrences of $G$ as induced subgraph in $\Gamma$ multiplied by $n-(t-1)$. The right-hand side consists of the sum of copies of all $t$-vertex graphs as induced subgraphs in $\Gamma$. Each of summands is multiplied by a constant saying how many times $G$ occurs in the appropriate graph.	
	\item For every vertex $v$ of $G$ there is an equation describing subgraphs, that can be obtained by adding a vertex adjacent to $v$.
	\item For any pair of adjacent vertices $x$ and $y$ in $G$ we have an equation for subgraphs that can be created by adding a common neighbor of $x$ and $y$.
	\item For any pair of non-adjacent vertices $x,y$ in $G$ there is an equation for subgraphs that can be created by adding their common neighbors.
\end{enumerate}

The cases 2., 3. and 4. are more complicated. It is necessary to care about situations when a new vertex has degree higher than $1$ (in the second case) or $2$ (the third and fourth case). This will be explained in more details in Section \ref{sec-alg}. This procedure has to be repeated for each graph of order $t-1$. At least we create an equation which says that the total number of $t$-vertex subgraphs of $SRG(n,k,0,\mu)$ is ${{n}\choose{t}}$. The solution of this system consists of numbers of all $t$-vertex graphs induced in $\Gamma$. In the case when the solution is not unique we need extra parameters, that describe relations between $t$-vertex subgraphs.

The construction of equations itself could be tedious for higher values of $t$, not to mention that this system has to be solved as well. Therefore we created an algorithm which implements the above ideas. 
\newpage

\subsection{The algorithm}\label{sec-alg}
\begin{definicia}
For a graph $G$ we denote by $P_G(H)$ the number of copies of $G$ in a graph $H$ as induced subgraph.
\end{definicia}

In the following text we repeatedly use a value $P_G(\Gamma)$, where $\Gamma$ is a fixed SRG and $G$ varies. In these cases we can omit $\Gamma$ and simplify the notation on $P_G$ which corresponds also to the notation in the previous section.\\
\\
The input:
\begin{itemize}
	\item Parameters of the graph $\Gamma=srg(n,k,\lambda,\mu)$
	\item The list $L_0=\{G_1,G_2,\dots,G_{|L_0|}\}$ of all graphs of order $o-1$
	\item The values $P_G$ for every graph $G\in L_0$ 
	\item The list $L$ of graphs of order $o$
\end{itemize}
The output:
\begin{itemize}
	\item System of linear equations, which are satisfied by values $P_H$, for $H\in L$
\end{itemize}
The procedure:
\begin{enumerate}[(1)]
	\item Initialize i=1
	\item Take $G_i\in L_0$
	\item Find all $H$ of order $o$ containing $G_i$ as induced subgraph
	\item Determine the equation that follows from parameter $n$ of $\Gamma$ \label{1}
	\item Determine all equations that follow from parameter $k$ of $\Gamma$ \label{2}
	\item Determine all equations that follow from parameter $\lambda$ of $\Gamma$ \label{3}
	\item Determine all equations that follow from parameter $\mu$ of $\Gamma$ \label{4}
	\item Increase $i$ by 1.
	\item If $i\leq |L_0|$, return to step $(2)$
	\item Return:
		\subitem The system of linear equations in the form $Mx=b$ satisfied by $P_H$, for $H\in L$.
\end{enumerate}

The algorithm works with parameters $n$, $k$, $\lambda$ and $\mu$ as with abstract variables, therefore the solution is universal for any SRG. However, the matrix $M$ contains just integers and does not depend on the choice of SRG. Since the right hand side of the output is derived using $n$, $k$, $\lambda$ and $\mu$, the solution (possible values of $P_H$, for $H\in L$) can be also expressed by these parameters.\\
Steps \ref{1}, \ref{2}, \ref{3}, \ref{4} are provided using statements in the following proposition:

\subsubsection{Propositions used in the algorithm}

This section contains several proofs that help describe the method for constructing equations in steps \ref{1}, \ref{2}, \ref{3}, \ref{4} of the algorithm.\\
Denote by $\Gamma$ an arbitrary but fixed $srg(n,k,\lambda,\mu)$. During this part we need to distinguish between two types of graphs of a given (small) order. The first one is an induced subgraph of $\Gamma$, which is labeled in correspondence with labeling of vertices in $\Gamma$. For this we reserve upper case letters of the Greek alphabet, $\Delta$ or $\Omega$.  The second type is an abstract graph and in this case we use the notation $G$ or $H$.\\
As defined above, $P_G$ will be the number of induced copies of graph $G$ in $\Gamma$. It follows that

$$P_G=\sum_{\substack{\Delta\subseteq \Gamma \\ \Delta\simeq G}}1$$ 
(as we consider only induced subgraph by $\Delta\subseteq \Gamma$ we always mean that $\Delta$ is an induced subgraph of $\Gamma$.)

Let $Aut(G)$ be the automorphism group of the graph $G$. 
The decompositions of the sets $V_G$, $E_G$, $\overline{E}_G$ (the set of non-edges in $G$) into orbits according to the action of $Aut(G)$ on them will be denoted in the following way:
\begin{itemize}
	\item $V_G=\cup_{i}V^i_G$ 
		\subitem Here, each $V^i_G$ is regarded as the system of one-element subsets of $V_G$.
	\item $E_G=\cup_{i}E^i_G$ 
	\item $\overline{E}_G=\cup_{i}\overline{E}^i_G$
\end{itemize}

All the statements in the following lemma are well known in algebraic graph theory.

\begin{lema}
Let $G$ and $G'$ represent two isomorphic graphs. Let $V_G^1, V_G^2,\dots, V_G^a$ be the orbits of $V_G$ with respect to $Aut(G)$ and $V_{G'}^1, V_{G'}^2,\dots, V_{G'}^{a'}$ be the orbits of $V_{G'}$ for $Aut(G')$. Then
\begin{enumerate}
	\item $Aut(G)\simeq Aut(G')$,
	\item $a=a'$,
	\item There exists a permutation $\psi\in S_a$, such that for each isomorphism $\phi:G\longrightarrow G'$:
				$$(V^i_G)_\phi={V}^{i_\psi}_G,$$
	\item For each $i\in \{1,\dots,a\}$ and for each pair of vertices $v\in V^i_G$ and $v'\in {V}^{i_\psi}_{G'}$ there exists an isomorphism $\phi:G\longrightarrow G'$ such that:
				$$v\mapsto v'.$$
\end{enumerate}
\end{lema}

We assume that the permutation $\psi$ for any pair of isomorphic graphs in the previous lemma is always the identity.\par
Let $G$ be a graph on $o-1$ vertices with orbits $V^1_G,\dots V^a_G$ of $V_G$ and let $\Delta$ be a subgraph of $\Gamma$ isomorphic to $G$. Let us fix a vertex $v\in V_\Delta^i$. Then there are $k-deg_\Delta(v)$ vertices in $\Gamma-\Delta$ that are adjacent to $v$. We will use the notation $\Omega$ for a subgraph $\Delta\cup u\subseteq \Gamma$ on $o$ vertices, where $v\sim u$.\\
The graph $\Gamma$ contains exactly $P_G$ copies of $G$. Each of them can be extended in $deg_G(v)|V_G^i|$ ways to a subgraph of $\Gamma$ of order $o$ by adding a vertex $u\sim v$, where $v\in V_\Delta^i$. Hence, the total number of created subgraphs of order $o$ is 

\begin{equation}\label{ls}
(k-deg_G(v))|V_G^i|P_G.
\end{equation}

Note that some subgraphs could be created more than once and we need to count the contribution of each subgraph $\Omega\subseteq\Gamma$ of order $o$ into the value \ref{ls}. If there is no vertex $u\in V_\Omega$ such that $\Omega-u\simeq G$, then the contribution of $\Omega$ is zero. In the other case, the contribution is the number of neighbors of $u\in V_\Omega$ in the orbit $V^i_G$, for each $u$ satisfying $\Omega-u\simeq G$. This value will be counted by a function $f(V^i_G,\Omega, u)$. Going through all subgraphs of order $o$ in $\Gamma$ we obtain the value:
\begin{equation}\label{rs}
\sum_{H}\sum_{j}f(V^i_G,H,u_j)|V_j(H)|P_{H},
\end{equation}
where $H$ is a graph of order $o$ and $u_j$ is a vertex from the orbit $V^j_H$ a the vertex set $V_H$. Clearly the values \ref{ls} and \ref{ls} have to be equal.\\
The same idea can be repeated also for orbits of $E_G$ and $\overline{E}_G$. Therefore, the function $f$ has to be defined also for these cases.\par
Let $S_G$ be an orbit of some of the sets $V_G$, $E_G$ and $\overline{E}_G$ with respect to $Aut(G)$ and let $H$ be a graph on $|V_G|+1$ vertices. Then the function $f$ will be defined as follows:

\bigskip

$$
					\begin{array}{ll*{3}{l}}
					f(S_G,H,u) &=& \left\{\begin{array}{ll}
					0 & \textrm{ if } G\not\simeq {H-u},\\
					\left|\{S\in S_G:u\sim v, \forall v \in S \}\right| & \textrm{ otherwise. }
					\end{array}\right.
					\end{array}
$$
\bigskip
\bigskip

The orbit of any subset of $V_G$ with respect to $Aut(G)$ is invariant under automorphisms of $G$. It follows that the function $f$ is well defined.\\
If we choose the set $S_G$ with no restrictions, $f$ could be ambiguous. For better understanding we can consider graphs $G$ and $H$ as in the figure \ref{fig_f}. Obviously, $G$ and ${H-u}$ are isomorphic. If $S_G=\{\{v_2\}\}$ the value $f(S_G,H,u)$ depends on the isomorphism $\theta$ between $G$ and $H-u$. In the case when $\theta(v_2)=u_2$ the value $f(S_G,H,u)$ equals $1$. But if $\theta(v_2)=u_1$ then $f(S_G,H,u)=0$. 

\begin{figure}[h!]
\begin{center}
\includegraphics[scale=0.45]{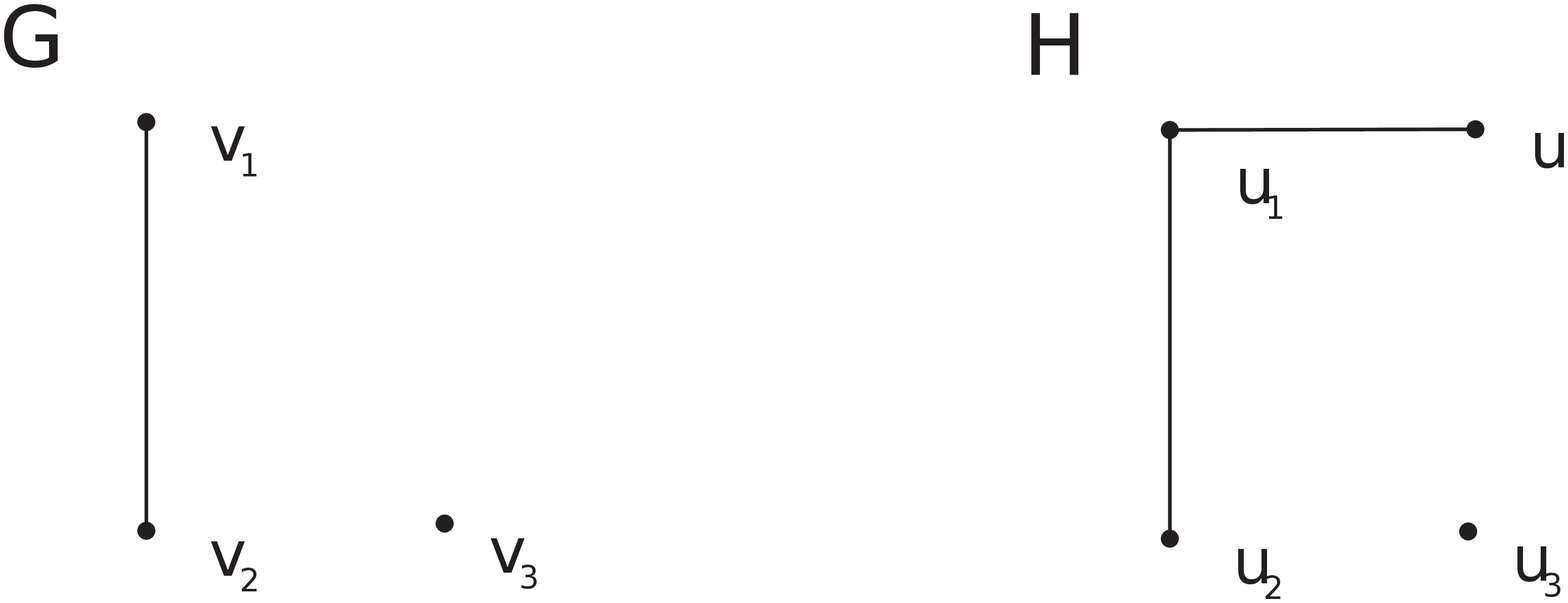}
\caption{}\label{fig_f}
\end{center}
\end{figure}

The degree of a set $S\subseteq V_G$ in $G$ is defined as the number of common neighbors for all vertices of $S$ in $G$:

$$deg_G(S)=|\{v\in V_G: u\sim v, \forall u \in S \}|$$

If some vertex $v$ is included in $\{v\in V_G: u\sim v, \forall u \in S \}$ we use the notation $v\sim S$. Hence, $deg_G(v)$ of a vertex $v\in V_G$ is just the degree of $v$ in the graph $G$.

\begin{lema}\label{lema-sumy}
Let us fix a subgraph $\Delta \subseteq \Gamma$ and let $S_\Delta$ be an orbit of some of sets $V_\Delta$, $E_\Delta$ and $\overline{E}_\Delta$. Then the following equality holds:

$$|S_\Delta|(\alpha-deg_\Delta(S))=\sum_{u\in V_\Gamma -V_\Delta} f(S_\Delta,\Delta\cup u,u),$$

where $\alpha\in\{k,\lambda,\mu\}$.
\end{lema}

\begin{proof}
Let us consider a fixed element $s\in S_\Delta$. Then, there are $\alpha -deg_\Delta(s)$ vertices in $\Gamma$ that are adjacent to each vertex of $s$, where $\alpha\in\{k,\lambda,\mu\}$. This gives an equality

$$\alpha-deg_\Delta(S)=
\sum_{\substack{u\in V_\Gamma-V_\Delta \\ u\sim S}} 1.$$

Since this holds for each set $S\in S_\Delta$, we can modify the equality as follows:

\begin{center}
\begin{align*}		
\sum_{S\in S_\Delta}(\alpha-deg_\Delta(S))&=\sum_{S\in S_\Delta}\left( \sum_{\substack{u\in V_\Gamma-V_\Delta \\ u\sim S}} 1 \right),\\
&\\
|S_\Delta|(\alpha-deg_\Delta(S))&=\sum_{S\in S_\Delta}\left( \sum_{\substack{u\in V_\Gamma-V_\Delta \\ u\sim S}} 1 \right).\\
\end{align*}
\end{center}

The contribution of each vertex $u\in V_\Gamma-V_\Delta$ to the double sum on the right-hand side of the equality is exactly $f(S_\Delta,\Delta\cup u,u)$, which eventually proves the statement of the lemma.
\end{proof}

Let us consider a fixed subgraph $\Delta$ of $\Gamma$ isomorphic to $G$. Let $H$ be a graph on $|V_G|+1$ vertices containing $G$ as an induced subgraph. Now we derive the relation between $\Delta$ and the value $P_H$. Clearly, the number of subgraphs in $\Gamma$ isomorphic to $H$, that contain $\Delta$ as induced subgraph depends on the choice of $\Delta$. Hence, this number can differ also for two subgraphs $\Delta_1$ and $\Delta_2$ both isomorphic to $G$. On the other hand, we can sum these numbers for all subgraphs $\Delta\subseteq \Gamma$ isomorphic to $G$ and derive a relation between this sum and the value $P_H$. Using the described idea and the strong regularity of $\Gamma$ we obtain the following proposition.

\begin{prop}\label{prop_rovnice}
Let $L$ be the set of all graphs on $o$ vertices. For $H\in L$, the vertex $u_j$ will be a representative of the orbit $V^j_H$. If $G$ is a graph of order $o-1$, then the following statements holds:
\begin{enumerate}[i.]
	\item $$(n-o+1)P_G=\sum_{H\in L}P_{G}(H)P_{H}$$ \label{rovnica1}
	
	\item For any $v\in V^i_G$ with $deg_G(v)\leq k$  
		$$(k-deg_G(v))|V^i_G|P_G=\sum_{H\in L}\sum_{j}f(V^i_G,H,u_j)|V^j_H|P_{H},$$ \label{rovnica2}
		
	\item For any edge $(v_1,v_2)\in E^G_i$ such that $\deg_G(v_1,v_2)\leq\lambda$
	
		$$(\lambda-deg_G(v_1,v_2))|E^i_G|P_G=\sum_{H\in L}\sum_{j}f(E^G_i,H,u_j)|V^j_H|P_{H},$$ \label{rovnica3}
				
	\item For any non-edge $(v_1,v_2)\in \overline{E}^i_G$ such that $\deg_G(v_1,v_2)\leq\mu$ 
	
		$$(\mu-deg_G(v_1,v_2))|\overline{E}^i_G|P_G=\sum_{H\in L}\sum_{j}f(\overline{E}^i_G,H,u_j)|V^j_H|P_{H}.$$ \label{rovnica4}
\end{enumerate}
\end{prop}

\begin{proof}
\ref{rovnica1}. The equation comes from an easy observation. Let us fix some induced subgraph $\Delta\in\Gamma$ isomorphic to $G$. Then there are exactly $n-(o-1)$ possibilities to extend $\Delta$ by adding one new vertex of $\Gamma$ into $V_{\Delta}$. Obviously, we obtain $n-o+1$ new subgraphs of $\Gamma$. This can be written as follows:

$$(n-o+1)=
\sum_{v\in{V_\Gamma-V_{\Delta}}}1=
\sum_{\substack{\Omega\subseteq \Gamma \\ |V_{\Omega}|=o}}\mathbf{1}_\Omega(\Delta),$$

where $\mathbf{1}_\Omega$ is a characteristic function of $\Omega$ defined as

$$
					\begin{array}{ll*{3}{l}}
					\mathbf{1}_\Omega(\Delta) &=& \left\{\begin{array}{ll}
					1 & \textrm{ if } \Delta\subseteq \Omega,\\
					0 & \textrm{ otherwise. }
					\end{array}\right.
					\end{array}
$$

At last, we sum the obtained equation over all subgraphs of $\Gamma$ isomorphic to $G$. Clearly, the left hand side is for any subgraph $\Delta\simeq  G$ the same. So it is just multiplied by the total number of subgraphs of $\Gamma$ isomorphic to $G$:

\begin{center}
\begin{align*}
(n-o+1)P_G	&=\sum_{\substack{\Delta\subseteq \Gamma \\ \Delta\simeq G}}\sum_{\substack{\Omega\subseteq \Gamma \\ |V_{\Omega}|=o}}\mathbf{1}_\Omega(\Delta)
			 =	 \sum_{\substack{\Omega\subseteq \Gamma \\ |V_{\Omega}|=o}}P_{G}(\Omega)\\
\end{align*}
\end{center}

After identifying all $\Omega\subseteq \Gamma$ that are mutually isomorphic we obtain the final equation:

\begin{center}
\begin{align*}		
(n-o+1)P_G	&=	 \sum_{H\in L}P_{G}(H)P_H.\\
\end{align*}
\end{center}

\ref{rovnica2}. Let us fix once more an induced subgraph $\Delta$ isomorphic to $G$ in $\Gamma$ and let a vertex $v$ be a representative of the orbit $V^i_{G}$ with $deg_{\Delta}(v)<k$ as . Since the degree of each vertex in $\Gamma$ equals $k$, $V^i_{G}$ satisfies the requirements of the lemma \ref{lema-sumy}, which gives:

$$
(k-deg_{\Delta}(v))|V^i_{G}|=\sum_{u\in V_{\Gamma}-V_\Delta} f(V^i_{G},\Delta\cup u,u).
$$

If we sum up the above equation over all subgraphs isomorphic to $G$, the left hand side will be just multiplied by value $P_G$:

$$
(k-deg_{\Delta}(v))|V^i_{G}|P_G	= \sum_{\substack{\Delta\subseteq \Gamma \\ \Delta\simeq G}}\sum_{u\in V_{\Gamma}-V_\Delta} f(V^i_{G},\Delta\cup u,u).
$$

This can by modified to obtain:

\begin{center}
\begin{align*}
(k-deg_{G}(v))|V^i_{G}|P_G	&=\sum_{\substack{\Omega\subseteq \Gamma \\ |V_\Omega|=o}}\sum_{j}f(V^i_{G},\Omega,u_j)|V^j_H|\\	
&\\
							&=\sum_{H\in L}\sum_{j}  f(V^i_{G},H,u_j)|V^j_H|P_H		\\
\end{align*}
\end{center}

\ref{rovnica3}. and \ref{rovnica4}. can be proven in a way similar to the part \ref{rovnica2}.

\end{proof}
\begin{figure}
\begin{center}
\includegraphics[scale=0.2]{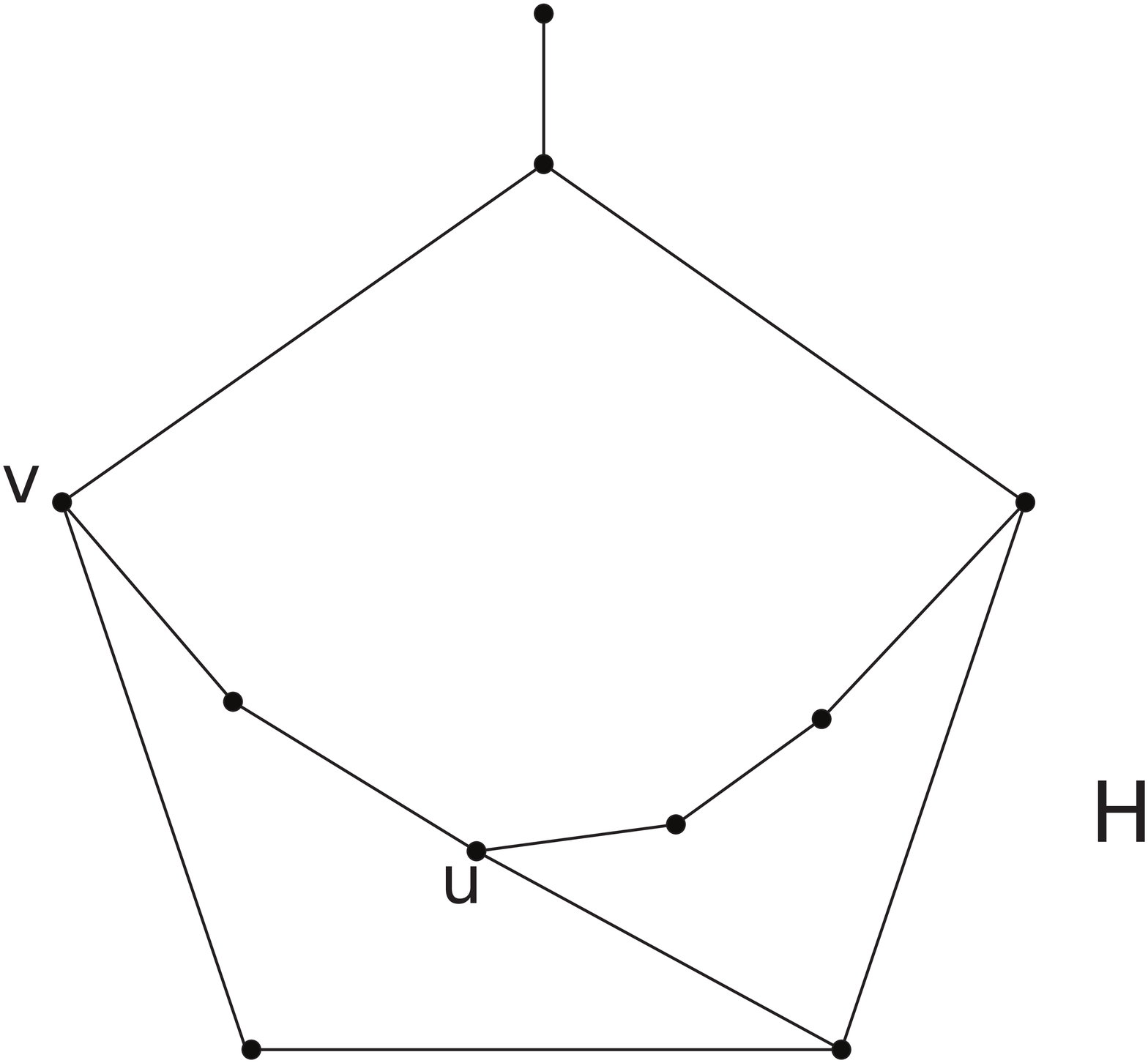}
\caption{}\label{fig_rovnice}
%\label{}
\end{center}
\end{figure}

\begin{rem}
In part \ref{rovnica1} of Proposition \ref{prop_rovnice} it is sufficient to use the value $P_{G}(H)$ in the right-hand side of the equation. This is different in the remaining parts of the Proposition. Here we need to distinguish between elements, which are not in the same orbit of $V_H$, $E_H$ or $\overline{E}_H$. For better understanding we illustrate this in Figure \ref{fig_rovnice}.\par
The vertices $u$ and $v$ in \ref{fig_rovnice} are obviously members of different orbits of $V_H$. In the figure \ref{fig2_rovnice} there are two graphs obtained by omitting $u$ or $v$. They are isomorphic to the same graph $G$. Now let $V^i_G$ be the set of vertices of degree two that are adjacent to a list in the graph $G$. There is just one such vertex.  $f(V^i_G,H,v)=1$ and $f(V^i_G,H,u)=0$.
\end{rem}

\begin{figure}[h!]
\begin{center}
\includegraphics[scale=0.2]{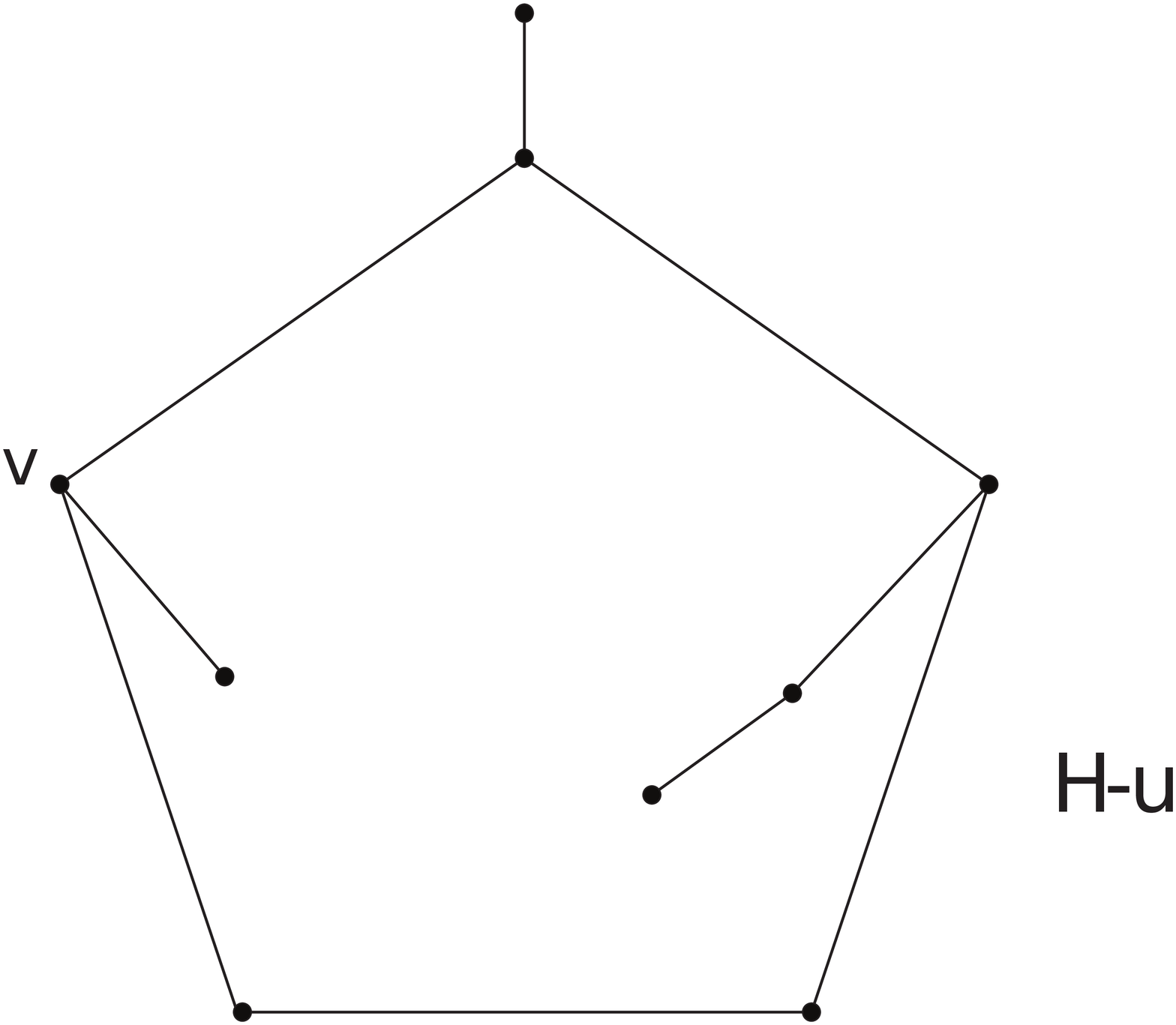}
\includegraphics[scale=0.2]{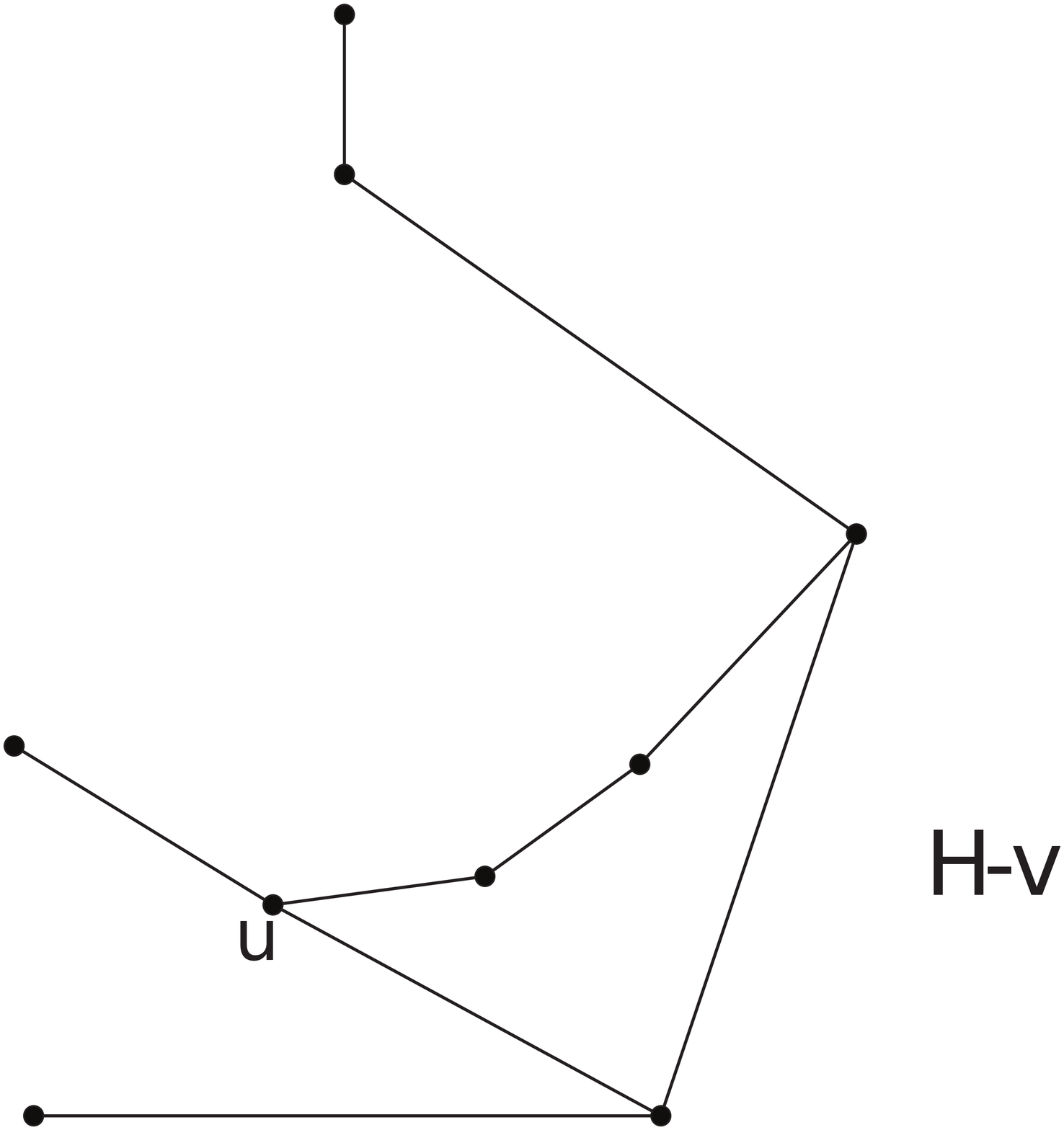}
\caption{Two isomorphic subgraphs of $H$}\label{fig2_rovnice}
\end{center}
\end{figure}

\begin{prop}
Let $\Gamma$ be a $srg(n,k,\lambda,\mu)$. Then, the number of all induced subgraph of order $o$ in $\Gamma$ equals to ${n\choose o} $ and the equation 

$${n\choose o}=\sum_{H\in L}P_H$$

is a linear combination of equations of the type \ref{rovnica1}. in \ref{prop_rovnice}.
\end{prop}

\begin{proof}
We just need to sum equation \ref{rovnica1}. through all subgraphs $G$ of order $o-1$ in $\Gamma$ and modify the result.
\begin{center}
\begin{align*}
\sum_{G}(n-o+1)P_G					&=\sum_{G}\sum_{H\in L}P_{G}(H)P_{H}\\
					&\\
(n-o+1){n\choose {o-1}}				&=\sum_{H\in L}oP_{H}\\
\end{align*}
\end{center}

\begin{center}
\begin{align*}
\frac{(n-o+1)}{o}{n\choose {o-1}}	&=\sum_{H\in L}P_{H}\\
					&\\
{n\choose {o}}						&=\sum_{H\in L}P_{H}\\
\end{align*}
\end{center}
\end{proof}

Note that there is another equation describing relations between induced subgraphs. It can be from characteristic polynomial of $\Gamma$. However, in each case when we included also this equation to our computations, we did not obtain any new result.

\subsection{Results}

This section summarize the results that follow from outputs of our algorithm. The solution of the system of equations described in the previous part gives us total numbers of all induced subgraphs of a given order in some $srg(n,k,\lambda,\mu)$. This solution can be for some $SRGs$ unique (if and only if the rank of the system is equal to the number of created graphs) but in general it is not. Hence we are obtaining a set of solutions with the integrability requirement. In this case there are dependences between numbers of occurrences of subgraphs. These we represent using new parameters ($P_1,P_2,\dots$) denoting occurrences of appropriate subgraphs in a given $SRG$. Clearly, there are always more options to distribute parameters $P_i$ and we describe our result by choosing one of them. The numbers graphs of small order in the following table are taken from \cite{orderly}.

\begin{table}[h!]    
\begin{center}
\begin{tabular}{c||c|c|c}
%\cline{2-4}
& \multicolumn{3}{ c }{Type of graphs} \\ \hline
order & general & $C_3$-free & $(C_3,C_4)$-free  \\ \hline
$1$	  &	$1$		&	$1$		& $1$	\\ \hline
$2$	  &	$2$		&	$2$		& $2$	\\ \hline
$3$	  &	$4$		&	$3$		& $3$	\\ \hline
$4$	  &	$11$	&	$7$		& $6$	\\ \hline
$5$	  &	$34$	&	$14$	& $11$	\\ \hline
$6$	  &	$156$	&	$38$	& $23$	\\ \hline
$7$	  &	$1044$	&	$107$	& $48$	\\ \hline
$8$	  &	$12346$	&	$410$	& $114$	\\ \hline
$9$	  &	$274668$&	$1897$	& $293$	\\ \hline
$10$  &$12005168$&	$12172$	& $869$	\\ \hline
$11$  &$1018997864$&$105071$& $2963$\\ 
\end{tabular}
\end{center}
\caption{Numbers of non-isomorphic graphs of small orders}
\label{poctygrafov}
\end{table}
\newpage

\subsubsection{Subgraphs of triangle free $SRGs$}

In this section, $\Gamma$ always denotes a $srg(n,k,0,\mu)$. It was already shown that the numbers of all subgraph in $\Gamma$ of order at most $4$ depend only on $n$, $k$ and $\mu$. We present here results for subgraphs of higher orders.

\begin{prop}
The system of equations provided by our algorithm for $srg(n,k,0,\mu)$ and $o\in\{5,6,7,8\}$ has the following parameters:\\
\begin{table}[h!]
\begin{center}

\begin{tabular}{r|| c| c| c| c| c}
			& number of		&	number of	&	rank of		&	number of		&	number of 		\\
	$o$		& equations 	&	 variables	&	the system	&	 free variables	& new parameters	\\
\hline
\hline
$5$			&	$30$		&	$14$		&	$14$		&	$0$				& $0$	\\
\hline
$6$			&	$86$		&	$38$		&	$37$		&	$1$		 		& $1$	\\
\hline
$7$			&	$301$		&	$107$		&	$106$		&	$1$				& $2$	\\
\hline
$8$			&	$1238$		&	$410$		&	$402$		&	$8$				& $10$	\\
\end{tabular}
\end{center}
\end{table} 
\end{prop}

The more detailed analysis of the results from the previous proposition allows us to obtain new information about relation between subgraphs of $srg(n,k,0,\mu)$. These are summarized in following theorems.

\begin{veta}\label{tf5}
The value $P_G$ is, for any graph $G$ on at most $5$ vertices, determined uniquely by parameters $k$ and $\mu$.
\end{veta}

\begin{veta}
There are $12$ graphs on $6$ vertices for which their numbers of induced copies in $\Gamma$ are determined uniquely only by parameters $k$ and $\mu$. The values $P_G$ for remaining $26$ cases depend also on the number of induced $K_{3,3}$ subgraphs in $\Gamma$, which is represented by the parameter $P_1$.
\end{veta}

\begin{veta}
There are $15$ graphs on $7$ vertices for which their numbers of occurrences in $\Gamma$ are determined uniquely by values $k$ and $\mu$. Among the remaining $92$ cases there are $91$ graphs whose occurrences depend also on $P_{K_{3,3}}=:P_1$ and occurrences of $76$ of them (including $K_{3,4}$) depend on $P_{K_{3,4}}=:P_2$.
\end{veta}

\begin{veta}
There exist eight graphs on $8$ vertices, each containing induced $C_4$, with occurrences in $\Gamma$ denoted by $P_3,P_4,\dots P_{10}$, such that the number of occurrences of any graph of order $8$ depends on parameters $k$, $\mu$ and on $P_1,P_2,\dots,P_{10}$.
\end{veta}

\subsubsection{Induced subgraphs of $srg(3250,57,0,1)$}
The missing Moore graph is specific because of its small value of $\mu$. This implies that the number of occurrences of many graphs as induced subgraphs is in this particular case equal to zero. Clearly, a $srg(3250,57,0,1)$ cannot contain an induced $4$-cycle. Hence, also its subgraphs have to be triangle-free with no $4$-cycle, which corresponds to the last column of Table \ref{poctygrafov}.\\
It follows from theorems from the previous section that the numbers of occurrences of all graphs of orders at most $8$ in $srg(3250,57,0,1)$ are determined uniquely. Calculations for this special case confirm this result independently.

\begin{prop}
The systems of equations provided by our algorithm for\\ 
a $srg(3250,57,0,1)$ and $o\in\{9,10\}$:
\begin{enumerate}
	\item If $o=9$ then the system contains $1234$ equations on $293$ variables. The rank of this system is $293$, therefore it has a unique solution.
	\item If $o=10$ then the system contains $4221$ equations on $869$ variables. The rank of this system is $868$, therefore it has $1$ free variable.
\end{enumerate}
\end{prop}

\begin{veta}
The number of induced copies of any graph on at most $9$ vertices in $\Gamma$ is constant.
\end{veta}

\begin{veta}
There are $595$ graphs of order $10$ whose number of occurrences in $\Gamma$ is constant. Occurrences of remaining $274$ cases depend on the number of induced Petersen graphs in $\Gamma$.
\end{veta}

The number of induced subgraph of any type has to be a non-negative integer. This fact together with the solution of system of equations for $\Gamma$ and $o=10$ give the upper and lower bound for the number of induced Petersen graph $\Gamma$.

\begin{veta}
$\Gamma$ contains at most $266$ $266$ $000$ induced Petersen  graphs. The lower bound for the number of induced Petersen graphs in $\Gamma$ is equal to zero.
\end{veta}

\subsubsection{Subgraphs of arbitrary $SRGs$}

As we will see in the section \ref{geometry}, relationships among subgraphs of an arbitrary SRG can also be utilized. Moreover, they can be used also as a test of correctness of results obtained for triangle free case.

\begin{prop}
The system of equations provided by our algorithm for $srg(n,k,\lambda,\mu)$ and $o$ equal to $4$ and $5$ satisfies:
\begin{enumerate}
	\item If $o=4$ then the system contains $17$ equations on $11$ variables. The rank of this system is $10$, therefore it has $1$ free variable.
	\item If $o=5$ then the system contains $60$ equations on $34$ variables. The rank of this system is $31$, therefore it has $3$ free variables.
\end{enumerate}
\end{prop}

\begin{veta}\label{order4}
For any graph $G$ of order $4$ the value $P_G$ for $\Gamma=srg(n,k,\lambda,\mu)$ depends on parameters $k$, $\lambda$, $\mu$ and on the value $P_1:=P_{K_{1,3}}$. 
\end{veta}

The previous theorem says that there is no graph on $4$ vertices for which its number of occurrences would be determined uniquely only by parameters of $\Gamma$. On the other hand, if we would know the number of occurrences of any graph on $4$ vertices then this value would be known for each graph of order $4$. An equivalent result can also be found in \cite{tc} where authors use mostly combinatorial tools.\par

\begin{prop} \label{order5}
Let us assume that the numbers of all induced subgraphs on $4$ in a $\Gamma=(n,k,\lambda,\mu)$ vertices are known. Then $P_G$ for any graph on $5$ vertices is determined by $k$, $\lambda$, $\mu$ and $3$ new parameters $P_2$, $P_3$ and $P_4$ denoting the number of induced copies of $K_{1,4}$, $K_{1,4}$ with one extra edge and $K_{2,3}$ in $\Gamma$, respectively. %($P_1$-$18$ times, $P_2$-$30$ times, $P_3$-$30$ times, $P_4$-$28$ times)
\end{prop} 

In the case when $\Gamma$ is triangle-free, the number $K_{1,3}$ contained in $\Gamma$  equals $n{k \choose 3}$. So, the numbers of induced subgraphs in $\Gamma$ are determined uniquely by its parameters which is in correspondence with Proposition \ref{prop:pre4}. \\
If $P_2$, $P_3$ and $P_4$ denote the same graphs as in the Proposition \ref{order5} and $\Gamma$ is triangle-free, then $P_2=n{k \choose 4}$ and $P_3=0$. Each non-edge of $\Gamma$ lies in ${\mu \choose 3}$ induced $K_{2,3}$ subgraphs. So the value $P_4$ is determined uniquely as well. Thus we obtain an independent proof of Theorem \ref{tf5}.

%\newpage
%\input{algorithm}
\newpage
\section{Geometry of the spectra of $srg$} \label{geometry}
This chapter brings two approaches for dealing with $SRG$s together. The first comes from the algorithm that we designed. The second follows from the geometrical representation of SRGs.
\par The idea comes from the work of Bondarenko, Prymak and Radchenko \cite{geometry}. They use a geometrical representation of vertices of SRG on the unit sphere $S^{g-1}$ in the Euclidean vector space $\mathbb{R}^g$ (introduced in Section \ref{abs-rep}), extended by vectors representing edges of SRG as a normalized sum of vertices that form a given edge. This can be easily generalized to representation of any induced subgraph of a SRG on $S^{g-1}$. The inner product of two representatives of subgraphs follows from the inner product of vertices inducing them and from their mutual positions. Therefore they derive relations between numbers of induced subgraphs of SRGs on up to four vertices. 
As we have already proven, these numbers are uniquely determined by parameters $n$, $k$, $\lambda$ and $\mu$ for all subgraphs of order at most $3$. In the case of order $4$, it is sufficient to know the number of occurrences of any subgraph on four vertices in a given SRG. This corresponds to the result in \cite{geometry}. 
The Riesz representation theorem allows the authors to convert vectors of $S^{g-1}$ into special polynomials with inputs from the sphere of the same dimension. The inner product of two such polynomials is then given by the inner product of vectors of the original representation. Here they can use classical properties of spaces with an inner product such as Cauchy-Schwartz inequality, which leads to the lower bound for number of induced $K_4$ in SRG. In the case of a $srg(76,30,8,14)$, this lower bound is greater than zero. This step allows them to show that it has to contain one of the three larger induced subgraphs. Each of these cases leads to a contradiction.
\par The technique was not developed for subgraphs of higher orders. The reason may have been increasing complexity of dependencies between induced subgraphs of SRG. Since our algorithm derives these dependencies, we can generalised the idea just described. Most interesting for us is to derive the bound for the number of induced $K_{3,3}$ subgraphs in tfSRGs. In this case we need to consider mutual positions between subgraphs of order at most $3$.

\subsection{Zonal spherical harmonics}
This section describes an alternative view on the geometrical representation of $SRGs$ \cite{geometry}. Instead of working directly with vectors of the sphere $S^{g-1}$ we will translate the whole problem into the world of special polynomials with inputs from $S^{g-1}$. The inner product of these polynomials will depend on the inner product of "vertices" projected into the eigenspace $V_2$. This connection allows us to use properties of polynomials for SRGs.\par
The tools leading to this representation of SRGs are outside of the area of our study, so we include only definitions and statements of important results without proofs. For more details we recommend \cite{harm} and \cite{harm2}.

\begin{veta}[Riesz representation theorem]\label{riesz}
Let $H$ be a Hilbert space with inner product $\langle .,. \rangle$ and let $H^*$ be the dual space of $H$. For any functional $g\in H^*$ there exists a unique $y\in H$ such that

$$g(x)=\langle x,y \rangle,\textrm{ for all } x\in {H}$$

\end{veta}

We say that a function $f:\mathbb{R}^g\longrightarrow \mathbb{R}$ is homogeneous of degree $t$ if and only if $f(ax)=a^tf(x)$ for all $x\in \mathbb{R}^g$ and $a>0$.

\begin{definicia}
Let $\mathcal{P}_t(g)$ denote the space of homogeneous polynomials of degree $t$ in $g$ variables with real coefficients. Then $\mathcal{P}_t(S^{g-1})$ is the restriction of polynomials in $\mathcal{P}_t(g)$ to the sphere $S^{g-1}$.\\
Let $\mathcal{H}_t(g)$ denote the space of (real) harmonics polynomial defined as
	
	$$\mathcal{H}_t(g)=\{P\in\mathcal{P}_t(g)|\Delta P=0\}.$$
	
The symbol $\Delta$ in the definition above represents the Laplace operator and its acting on $P$ satisfies 
	
$$\Delta P=\sum_{i=1}^{g}\frac{\partial^2P}{x_i}.$$
	
Finally, let $\mathcal{H}_t(S^{g-1})$ be the set of restrictions of harmonic polynomials in $\mathcal{H}_t(g)$ to the sphere $S^{g-1}$.
\end{definicia}

Elements of $\mathcal{H}_t(S^{g-1})$ are called (real) \textit{spherical harmonics} and form a Hilbert space with the usual inner product

$$\left\langle P, Q\right\rangle=\int_{S^{g-1}}P(x)Q(x)d\omega_g(x),$$
where $P,Q\in \mathcal{H}_t(S^{g-1})$ and $\mu_g$ is a Lebesque measure on the unit sphere $S^{g-1}$. 
%(normalized by $\mu_g(S^{g-1})=1$). 

We are now prepared to define polynomials properties of which are useful for our purpose.

\begin{definicia}
For a fixed vector $x\in S^{g-1}$, the \textit{zonal spherical harmonic} $Z^t_x$ of degree $t$ is defined as the dual (Riesz) representation of the mapping $P\mapsto P(x)$ where $P\in \mathcal{H}_t(S^{g-1})$. In other words, $Z^t_x(y)$ satisfies the following reproducing property: 

$$P(x)=\int_{S^{g-1}}Z^t_x(y)P(y)d\omega_g(y),$$

for all $P\in \mathcal{H}_t(S^{g-1})$.
\end{definicia}

The following property follows from the uniqueness of the element $y$ in Theorem \ref{riesz}.

\begin{lema}\label{ip}
The zonal spherical harmonics satisfy the equality:

$$\left\langle Z^t_y, Z^t_x\right\rangle=\int_{S^{g-1}}Z^t_y(\xi)Z^t_x(\xi)d\omega_g(\xi)=Z^t_y(x)$$
\end{lema}

Each polynomial $Z^t_x$ is an element of $\mathcal{H}_t(S^{g-1})$, which is a Hilbert space. Hence we can use the Cauchy-Schwartz inequality for $Z^t_x$ and $Z^t_y$.

$$\langle Z^t_x,Z^t_y\rangle ^2\leq \langle Z^t_x,Z^t_x\rangle \langle Z^t_y,Z^t_y\rangle $$

Now If $X=\{x_1,x_2,\dots\}$ and $Y=\{y_1,y_2,\dots\}$ are two finite subsets of points from the unit sphere $S^{g-1}$, then:

\begin{center}
\begin{align*}
\left(\sum_{i,j}\langle Z^t_{x_i},Z^t_{y_j}\rangle \right)^2&=\left(\left\langle \sum_{i}Z^t_{x_i},\sum_{j}Z^t_{y_j}\right\rangle \right)^2\\
&\\
						   	   &\leq \left\langle \sum_{i}Z^t_{x_i},\sum_{i}Z^t_{x_i}\right\rangle  \left\langle \sum_{j}Z^t_{y_j},\sum_{j}Z^t_{y_j}\right\rangle \\
						   	   &\\
						       &= \sum_{i,i'}\langle Z^t_{x_i},Z^t_{x_{i'}}\rangle  \sum_{j,j'}\langle Z^t_{y_j},Z^t_{y_{j'}}\rangle 
\end{align*}
\end{center}

Using lemma \ref{ip}, which uniquely determines the value of the inner product of two zonal harmonics , we have: 

\begin{equation}\label{ineqZ}
\left(\sum_{i,j}Z^t_{x_i}(y_j) \right)^2
\leq
\sum_{i,i'} Z^t_{x_i}(x_{i'}) \sum_{j,j'} Z^t_{y_j}(y_{j'})
\end{equation}

For our purpose the definition of zonal spherical harmonics that use Gegenbauer polynomial is more useful. The reason will be obvious as soon as we present this alternative. 

\begin{definicia}
The \textit{Gegenbauer polynomial} $C^\alpha_t(x)$ is defined recursively by
\begin{center}
\begin{align*}
C_0^{\alpha}(x)&=1,\\
C_1^{\alpha}(x)&=2\alpha x,\\
			   &\vdots \\
C_t^{\alpha}(x)&=\frac{1}{t}\left[2x(\alpha+t-1)C_{t-1}^{\alpha}(x)-(2\alpha+t-2)C_{n-2}^{\alpha}(x)\right].
\end{align*}
\end{center}

This can be written parametrically as 

$$C_t^{\alpha}(x)=\sum_{i=0}^{t/2}\frac{(-1)^i\alpha_{t-i}(2x)^{t-2i}}{i!(t-2i)!},$$
where $\alpha_t=\alpha(\alpha+1)\dots(\alpha+t-1)$.\\

The Gegenbauer polynomial is even if the parameter $t$ is even number (in the other case it is odd).
\end{definicia}

\begin{prop}\label{propZC}
The zonal spherical harmonic $Z^t_x(y)$ can be expressed as follows.

$$Z^t_x(y)=\frac{1}{c_{g,t}}C^\alpha_t(x\cdot y),$$
where $\alpha=\frac{g-2}{2}$ and $x\cdot y$ is the usual inner product of vectors $x,y\in\mathbb{R}^g$. The value $c_{g,t}$ is constant and satisfies

$$c_{g,t}=\frac{1}{\omega_{g-1}}\frac{2t+g-2}{g-2},$$
where $\omega_{g-1}$ denotes the surface area of the $g-1$ dimensional sphere.
\end{prop}

After the transition to new definition of zonal harmonics we obtain an easy way to compute their values for given vectors $x$ and $y$. Moreover, the constant $c_{g,t}$ does not depends on the choice of $x$ and $y$. It follows that the inequality \ref{ineqZ} can be rewritten in a more usable form.

\begin{prop}\label{propG}
Let $X=\{x_1,x_2,\dots\}$ and $Y=\{y_1,y_2,\dots\}$ be finite subsets of points of the unit sphere $S^{g-1}$.
Then the Gegenbauer polynomial $C^\alpha_t$ with $\alpha=\frac{g-2}{2}$ has the following property

$$\left(\sum_{i,j}C^\alpha_t(x_i\cdot y_j) \right)^2
\leq
\sum_{i,i'} C^\alpha_t(x_i\cdot x_{i'}) \sum_{j,j'} C^\alpha_t(y_j\cdot y_{j'})$$
\end{prop}

\begin{proof}
The property comes straightforwardly from inequality \ref{ineqZ} and proposition \ref{propZC}.
\end{proof}

\subsection{Gegenbauer polynomials and induced subgraphs in $SRGs$}

Let us recall the geometric representation of $\Gamma=srg(n,k,\lambda,\mu)$ from the section \ref{abs-rep} which uses the spectra of $\Gamma$. Each vertex $v_i\in V_\Gamma$ can be represented by the vector $x_i\in S^{g-1}$, where $g$ is a multiplicity of non-trivial eigenvalue of $\Gamma$. The angle between two representatives $x_i$ and $x_j$ depends only on the adjacency between $v_i$ and $v_j$:

$$
					\begin{array}{ll*{3}{l}}
					x_i\cdot x_j &=& \left\{\begin{array}{ll}
									1 & \textrm{ if } i=j,\\
									p & \textrm{ if } i\sim j,\\
									q & \textrm{ otherwise, }
									\end{array}\right.
					\end{array}
$$

Let $\Delta$ be a subgraph of $\Gamma$ induced by vertices $\{v^\Delta_1,v^\Delta_2,\dots,v^\Delta_o\}$. The vector $x^\Delta$ constructed as

$$x^\Delta=\frac{\sum_{i=1}^{o}x^\Delta_i}{||\sum_{i=1}^{o}x^\Delta_i||}$$
will represent the subgraph $\Delta$ on the sphere $S^{g-1}$.\\ 
Let us fix two graphs $G$ and $H$, and let $X$ and $Y$ be systems of all subgraphs of $\Gamma$ isomorphic to $G$ and $H$, respectively. Hence, the inner product of representatives of graphs $\Delta\in X$ and $\Omega\in Y$ satisfies:

$$ x^\Delta \cdot x^\Omega=\frac{\sum_{i,j}x^\Delta_i\cdot x^\Omega_j}{|\sum_{i}x^\Delta_i||\sum_{j}x^\Omega_j|}.$$

The value $x^\Delta_i\cdot x^\Omega_j$ equals $1$, $p$ or $q$ depending on adjacency of vertices $v^\Delta_i$ and $ v^\Omega_j$. In other words, if we denote $e(\Delta,\Omega)$ the cardinality of the set $\{(u,v)| u\in V_\Delta, v\in V_\Omega,u\sim v \}$ and $\overline{e}(\Delta,\Omega)$ the cardinality of the set $\{(u,v)| u\in V_\Delta, v\in V_\Omega,u\nsim v \}$ then

$$ x^\Delta \cdot x^\Omega=\frac{|V_\Delta\cap V_\Omega|+ e(\Delta,\Omega)p+ \overline{e}(\Delta,\Omega)q}
{\sqrt{\left(|V_\Delta|+2p|E_\Delta|+2q|\overline{E}_\Delta| \right) \left(|V_\Omega|+2p|E_\Omega|+2q|\overline{E}_\Omega|\right)}}.$$

Note that this expression depends only on the type of a subgraph induced by vertices $V_\Delta\cup V_\Omega$. This fact will be important later.\\
Let us consider a system $L$ of such graphs $W$ for which there exist two subsets $V^1_{W}, V^2_W$ of $V_W$ with following properties. First, the union $V^1_W\cup V^2_W$ is the whole set $V_W$. Second, $V^1_{W}$ induces a subgraph isomorphic to $G$ in $W$ and $V^2_{W}$ induces a subgraph isomorphic to $H$ in $W$.

\begin{lema}\label{lemaG}
Let us continue with notation $L$, $X$ and $Y$ in correspondence with the discussion above. Then 

\begin{align*}
&\sum_{\substack{\Delta\in X \\ \Omega\in Y}} C(x^\Delta\cdot x^\Omega)\\
&\parallel \\
&\sum_{W\in L}P_W(\Gamma) 
\sum_{\substack{\Delta,\Omega\subset W \\ V_\Delta \cup V_\Omega=V_W}} 
\frac{|V_\Delta\cap V_\Omega|+ e(\Delta,\Omega)p+ \overline{e}(\Delta,\Omega)q}
{\sqrt{\left(|V_\Delta|+2p|E_\Delta|+2q|\overline{E}_\Delta| \right) \left(|V_\Omega|+2p|E_\Omega|+2q|\overline{E}_\Omega|\right)}},
\end{align*}
where $\Delta$ and $\Omega$ represent fixed subgraphs in $\Gamma$ or in $W$ that are isomorphic to $G$ and $H$, respectively. $P_W(\Gamma)$ is  the number of induced subgraphs of $\Gamma$ that are isomorphic to $W$.
\end{lema}
\subsection{Results}

Lemma \ref{lemaG} together with Proposition \ref{propG} give us a new tool to handle numbers of induced subgraphs of a given SRG with respect to each other.\par
In \cite{geometry}, the set $X$ is chosen as the entire vertex set of a SRG. The members of the set $Y$ are representatives of edges in a SRG.	Hence, in this case it is necessary to know all possibilities for subgraphs that can be induced by two vertices, by a vertex-edge pair and by two edges in an SRG. Numbers of all these substructures can be expressed by $k$, $\lambda$, $\mu$ and the number of induced $K_4$ in a SRG. By setting $t=4$, the Proposition \ref{propG} gives for these $X$ and $Y$ the lower bound of the number of induced $K_4$ in an SRG.\par
\begin{comment}
Since, our algorithm is able to produce numbers of occurrences also for graphs of higher orders in SRG, we can change sets $X$ and $Y$ according to necessity. The most interesting results we obtain for the case when $Y$ contains all representatives of $K_{1,2}$ subgraphs in tfSRG. For this case we used values $P_G$ for graphs $G$ on at most $6$ vertices.

By appropriate choice of elements in $X$ and $Y$ we are able to use inequality in the Proposition \ref{propG} for different types subgraphs in a SRG. The changing of $t$ in the inequality switches between lower and upper bound of the number of an exploring induced subgraph in a SRG.\par

\end{comment}
In the triangle-free case, the lowest order such that the numbers $P_G$s are not determined uniquely by parameters $k$ and $\mu$ is $6$. However, they are determined by parameters of a tfSRG and by the number of its induced $K_{3,3}$. Therefore, if we chose $Y$ as the set of representatives of induced $K_{1,2}$ and $X$ as the set of representatives of some induced subgraph of order at most $2$, we obtain a bound on the number of induced $K_{3,3}$ in a tfSRG. In most cases the best results were obtained for $X$ containing representatives of edges in a tfSRG. In all cases that we tested the inequality gave a lower bound of induced $K_{3,3}$ for even values of $t$ and an upper bound for the odd values of $t$. The lower bounds are all negative and therefore they do not give any new result. On the other hand, odd values of $t$ lead for some tfSRGs to new upper bounds for their induced $K_{3,3}$. In the case of Krein graphs $Kr(r)$ the upper bound is exactly the same as the one coming from the statement of the Proposition \ref{krein}, which we state as a lemma.

\begin{lema} \label{lemaKR}
Let us consider Krein graph $Kr(r)$. The number of induced $K_{3,3}$ in this $SRG$ is equal to 
$\frac{1}{2}P_{\overline{K}_3}{r\choose {3}}$
\end{lema}

All results for tfSRGs are summarized in Table \ref{results-tfSRG}. The first column of the table represents the parameter set of the considered tfSRG. The second column belongs to the upper bound coming from the solution of the system of linear equations obtained by our algorithm for triangle-free case. The heading row describes the choice of $X$, $Y$ and $t$. The illustration of the results for even $t$ can be found in the last column, where the lower bound is negative.\\
The first four parameter sets represented in the table are Krein graphs. One sees that the results obtained geometrical approach are better for those graphs, where the Krein inequality is tighter.\\

\begin{table}[h!]
\centering
\footnotesize{
\begin{tabular}{c c c|| r|| l l l| l l l| l l l}
\multicolumn{2}{ r }{tfSRG} &		& 				& $X$	 & $Y$		 &	$t$ &$X$	 & $Y$		 &	$t$ &$X$	 & $Y$		 &	$t$ \\ 
$n$ & $k$ & $\mu$ & lin. equations & $K_{1}$ & $K_{1,2}$ & $5$ & $K_{1,1}$ & $K_{1,2}$ & $5$ & $K_{1,1}$ & $K_{1,2}$ & $6$\\
\hline % inserts single-line
\hline
100&	22&	6& 51 333.33&  \multicolumn{3}{ |r| }{0.00}&\multicolumn{3}{ |r| }{0.00}&	\multicolumn{3}{ |r }{-2 189.08}\\
\hline 
324& 57& 12&	15 800 400.00&	\multicolumn{3}{ |r| }{1 580 040.00}&\multicolumn{3}{ |r| }{1 580 040.00}&\multicolumn{3}{ |r }{-717 648.25}\\
\hline
784& 116&	20&	894 206 880.00&\multicolumn{3}{ |r| }{99 356 320.00}&	\multicolumn{3}{ |r| }{99 356 320.00}&	\multicolumn{3}{ |r }{-171 378 969.69}\\
\hline 
1600& 205& 30& 21 129 322 667.00&\multicolumn{3}{ |r| }{2 263 856 000.00}& \multicolumn{3}{ |r| }{2 263 856 000.00}&	\multicolumn{3}{ |r }{-384 925 3316.17}\\
\hline
77&		16&	4&	3 080.00&		\multicolumn{3}{ |r| }{534.63}&			\multicolumn{3}{ |r| }{552.27}&	\multicolumn{3}{ |r }{	-378.11}\\
\hline
162&	21&	3&	1 890.00&		\multicolumn{3}{ |r| }{124 452.56}&	\multicolumn{3}{ |r| }{	42 409.60}&		\multicolumn{3}{ |r }{-36 672.31}\\
\hline
176&	25&	4&	17 600.00&		\multicolumn{3}{ |r| }{218 854.15}&		\multicolumn{3}{ |r| }{74 550.28}&		\multicolumn{3}{ |r }{-60 480.70}\\
\hline
210&	33&	6&	246 400.00&	\multicolumn{3}{ |r| }{524 614.17}&		\multicolumn{3}{ |r| }{210 718.23}&	\multicolumn{3}{ |r }{-136 110.87}\\
\hline
266&	45&	9&	2 867 480.00&	\multicolumn{3}{ |r| }{1 186 988.94}&		\multicolumn{3}{ |r| }{730 655.25}&	\multicolumn{3}{ |r }{-350 790.83}\\
\hline
352&	36&	4&	73 920.00&		\multicolumn{3}{ |r| }{20 793 353.15}&	\multicolumn{3}{ |r| }{1 851 582.50}&	\multicolumn{3}{ |r }{-2 252 418.82}\\
\hline
392&	46&	6&	901 600.00&	\multicolumn{3}{ |r| }{46 206 927.87}&	\multicolumn{3}{ |r| }{3 943 905.78}&	\multicolumn{3}{ |r }{-4 572 373.08}\\
\hline
552&	76&	12&	48 070 000.00&	\multicolumn{3}{ |r| }{202 641 567.42}&	\multicolumn{3}{ |r| }{24 938 590.19}&	\multicolumn{3}{ |r }{-25 958 718.19}\\
\hline
638&	49&	4&	250 096.00&	\multicolumn{3}{ |r| }{546 336 456.15}&		\multicolumn{3}{ |r| }{17 861 596.71}&	\multicolumn{3}{ |r }{-35 169 414.63}\\
\hline
650&	55&	5&	965 250.00&	\multicolumn{3}{ |r| }{78 212 8175.56}&		\multicolumn{3}{ |r| }{22 335 864.64}&	\multicolumn{3}{ |r }{-46 682 434.69}\\
\hline
667&	96&	16&	248 390 800.00&	\multicolumn{3}{ |r| }{286 405 500.67}&	\multicolumn{3}{ |r| }{57 611 096.19}&	\multicolumn{3}{ |r }{-70 157 307.76}\\
\hline
800&	85&	10&	45 696 000.00&	\multicolumn{3}{ |r| }{3 196 736 243.94}&		\multicolumn{3}{ |r| }{70 873 609.74}&	\multicolumn{3}{ |r }{-153 636 540.68}\\
\hline
1073&	64&	4&	721 056.00&	\multicolumn{3}{ |r| }{7 483 465 849.76}&		\multicolumn{3}{ |r| }{128 237 091.68}&	\multicolumn{3}{ |r }{-331 449 507.48}\\
\end{tabular}
}
\caption{Bounds for induced $K_{3,3}$ in $tfSRGs$ on up to $1100$ vertices}
\label{results-tfSRG}
\end{table} 

We have tried to obtain new bounds also for the number of induced Petersen graphs in the missing Moore graph by this method. The upper bound following from our algorithm equals $266$ $266$ $000$ and the lower bound is $0$. The bounds obtained from the geometrical approach are for all cases which we tested much worse. Some of them are presented in the Table \ref{bound-Moore}. The set $Y$ consists of all representatives for induced $C_5$ in $srg(3250,57,0,1)$. The notation $P4$ in the table is used for a path of length $4$. Even values of $t$ give the lower bound and odd values the upper bound for induced Petersen graphs.

\begin{table}[h!]
\centering
\footnotesize{
\begin{tabular}{r|| r| r| r}
		&\multicolumn{1}{|c|}{$X:$  $K_1$}&\multicolumn{1}{|c|}{$K_{1,1}$}&\multicolumn{1}{|c}{$P4$}\\
\hline\hline
$t=5$	& 22 694 158 422 336.21		& 4 835 831 221 300.89			& 2 403 005 672 027.93				\\
\hline
$6$		&-18 674 394 364 791.40		&-4 326 814 304 388.56			&-20 580 486 113 004.44 	\\
\hline
$7$		& 27 779 257 927 613.66		& 8 370 869 868 988.79			& 3 331 657 191 483.46 		\\
\end{tabular}
}
\caption{Bounds for induced Petersen graphs in $srg(3250,57,0,1)$}
\label{bound-Moore}
\end{table}

\newpage
%\section{Automorphism group(s) of the missing Moore graph(s)}

\section[Automorphism group(s) of the missing Moore graph(s)]{\texorpdfstring{Automorphism group(s) of the missing\\ Moore graph(s)}{Automorphism group(s) of the missing Moore graph(s)}}\label{sec:Moore}

The family of Moore graphs is a special class of $tfSRGs$ with $\mu=1$. According to \cite{HS}, there are exactly four feasible parameter sets for such graph. These are namely $(5,2,0,1)$ (pentagon), $(10,3,0,1)$ (Petersen graph), $(50,7,0,1)$ (Hoffman-Singleton graph) and $(3250,57,0,1)$. The proof of the feasibility of parameters is analogous to \ref{prop:param}. Hoffman and Singleton proved uniqueness of graphs with the first three parameter sets. The existence of the fourth case is a famous open problem in algebraic graph theory. For the rest of this section let $\Gamma$ denote a $srg(3250,57,0,1)$.
\par The automorphism group of $\Gamma$ is well studied. In 1971 Aschbacher proved that $\Gamma$ is not rank three graph \cite{Asb} and Higman showed in his unpublished notes that $\Gamma$ cannot be vertex transitive (for the proof see Cameron monograph \cite{Perm}). Later in 2001, Makhnev and Paduchikh showed that if $\Gamma$ has an involutive automorphism then $|Aut(\Gamma)|\leq 550$ \cite{MakPad}. Finally in 2009 Ma\v caj and \v{S}ir\'{a}\v{n} proved that if $|Aut(\Gamma)|$ is odd then $|Aut(\Gamma)|\leq 275$ and $|Aut(\Gamma)|\leq 110$ otherwise \cite{MacSir}.
\par We apply our results about small subgraphs in tfSRG to give a new information about automorphism of $\Gamma$ of order $7$. We begin by introducing relevant notions and known results.

\begin{prop}\cite{Asb}
Let $X$ be a group of automorphisms of $\Gamma$ and let $Fix(X)$ denote the subgraph induced by the set of all fixed points. Then $Fix(X)$ is the empty graph, an isolated vertex, a pentagon, the Petersen graph, the Hoffman-Singleton graph, or a star $K_{1,m}$ for some $m>0$.
\end{prop}
Since $\Gamma$ has diameter $2$, the image $v^x$ of any vertex $v$ under the automorphism $x$ is either itself or a neighbour of $v$, or a vertex at the distance two from $v$. Let $a_i(x)=|\{v\in \Gamma;d(v,v^x)=i\}|$ where $d$ denotes the distance between two vertices in $\Gamma$ and $i\in\{1,2,3\}$.
\newpage

\begin{prop}\cite{MacSir}\label{tabulka}
Let x be an automorphism of graph $\Gamma$ of a prime order $p$. Then the value $a_1(x)$ satisfies
\begin{table}[h!]
\centering
\begin{tabular}{| p{4cm} | p{4cm} | p{4cm} |}
\hline
	$a_0(x)$	&	$p$	&  $a_1(x)$ \\
\hline
	$0$		&	$5$		& $50+75k\leq 500$\\
	$0$		&	$13$	& $65+195k\leq 500$\\
	$1$		&	$3$		& $27+45k= 0$\\
	$1$		&	$19$	& $57+285k\leq 500$\\
	$5$		&	$5$		& $10+75k\leq 500$\\
	$5$		&	$11$	& $55+165k\leq 500$\\
	$10$	&	$3$		& $0$\\
	$50$	&	$5$		& $25+75k\leq 350$\\
	$56$	&	$2$		& $112$\\
	$2$		&	$7$		& $49+105k\leq 500$\\
	$9$		&	$7$		& $98+105k\leq 500$\\
	$16$	&	$7$		& $42+105k\leq 500$\\
	$23$	&	$7$		& $91+105k\leq 500$\\
	$30$	&	$7$		& $35+105k\leq 500$\\
	$37$	&	$7$		& $84+105k\leq 392$\\
	$44$	&	$7$		& $28+105k\leq 260$\\
	$51$	&	$7$		& $77$\\
\hline
\end{tabular}
\end{table}
\end{prop}

\begin{lema}
If $x$ is an automorphism of a prime order $p$ of $\Gamma$, then any graph $G$ satisfies the following:

$$P_G\equiv |\{\Delta\subseteq\Gamma; \Delta\simeq G, \Delta^x=\Delta\}|\mod{p}.$$

\end{lema}
\begin{proof}
If $\Delta\simeq G$ is an induced subgraph of $\Gamma$ then $\Delta^x$ is an induced subgraph of $\Gamma$ isomorphic to $G$ as well. Hence, the set $O_x(\Delta)=\{\Delta^{x^i};i\in\{0,1,\dots,(p-1)\}\}$ consists either of $p$ induced subgraphs of $\Gamma$ isomorphic to $G$, or it contains only $G$. This way we decompose the set of subgraphs of $\Gamma$ isomorphic to $G$ into several $p$-tuples and individuals. Hence, if the number of occurrences of $G$ in $\Gamma$ is congruent to $a \mod{p}$ where $a\in\{0,1,\dots,p\}$, then the number of induced subgraphs of $\Gamma$ isomorphic to $G$ that are fixed by automorphism $x$ equals $a \mod{p}$.
\end{proof}

\subsection{Automorphisms of order $7$}
The following result follows from the output of our algorithm.

\begin{prop}
The number of induced copies of any $7$-vertex graph except $\overline{K_7}$ in $\Gamma$ is a multiple of $7$. The number of induced $\overline{K_7}$ in $\Gamma$ is equal to $2 \mod{7}$.
\end{prop}
\begin{cor}
Let $x$ be an automorphism of order $7$ of $\Gamma$. 
\begin{itemize}
\item Let $G$ be a graph on $7$ vertices not isomorphic to $\overline{K_7}$. Then the number of induced subgraphs of $\Gamma$ isomorphic to $G$ and fixed by $x$ is congruent to $0\mod{7}$.
\item The number of induced subgraph $G$ in $\Gamma$ isomorphic to $\overline{K_7}$ and fixed by $x$ is congruent to $2\mod{7}$.
\end{itemize}
\end{cor}

\begin{lema}\label{lemmaCK}
Let $x$ be an automorphism of $\Gamma$ of order $7$ and let $v\notin Fix(x)$ be a vertex in this graph. The subgraph induced by the set $O_x(v):=\{v^{x^i};i\in\{0,1,\dots,6\}\}$ is either $\overline{K_7}$ or $C_7$.
\end{lema}
\begin{proof}
If $v$ has no neighbors in the set $O_x(v)$ then the graph induced by this set has to be $\overline{K_7}$. Let us suppose that $v$ and $v^{x^i}$ for $i\in\{1,2,\dots,6\}$ are neighbors ($v\sim v^{x^i}$). Therefore also $v^{x^i}\sim v^{x^{2i}}, v^{x^{2i}}\sim v^{x^{3i}},\dots,v^{x^{6i}}\sim v$, which forms a $C_7$. The vertex $v$ can not be joined to other vertices from the set $O_x(v)$ as there would be a triangle or a quadrangle in $\Gamma$.
\end{proof}

This way the automorphism $x$ decomposes vertices of $\Gamma$ into $(3250-a_0(x))/7$ sets, each representing a set $O_x(v)$ for some vertex $v\notin Fix(x)$. There are three possibilities how the $O_x(v)$ can induce $C_7$ subgraph (figure \ref{7orb}) and each of them occurs in the same number \cite{MacSir}. 
With this knowledge we are able to derive the number of $O_x(v)$ of the shape $C_7$ using only column $a_1(x)$ of the table in proposition \ref{tabulka}.

\begin{figure}[h!]
\begin{center}
\includegraphics[width=1\textwidth]{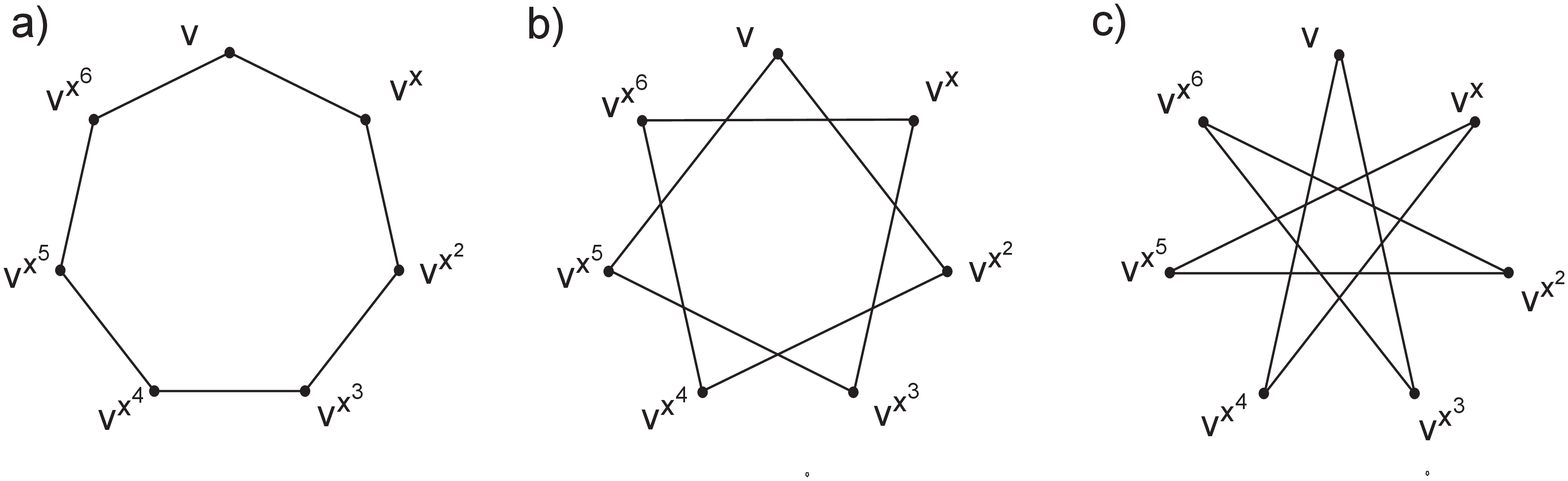}
\caption{}
\label{7orb}
\end{center}
\end{figure}
\newpage

\begin{lema}
Let $x$ be an automorphism of $\Gamma$ of order $7$. Then the number of $O_x(v)$ forming $C_7$ is $\frac{3a_1(x)}{7}$, which has to be a multiple of seven.
\end{lema}
Now we are ready to update the table from proposition \ref{tabulka} in the following way.

\begin{veta}
Let x be an automorphism of $srg(3250, 57,0, 1)$ of order $7$. Then the value $a_1(x)$ is a multiple of $49$.
\begin{table}[h!]
\centering
\begin{tabular}{| p{3cm} | p{3cm} |}
\hline
	$a_0(x)$	&  $a_1(x)$ \\
\hline
	$2$			& $49$\\
	$9$			& $98$\\
	$16$		& $147$\\
	$23$		& $196$\\
	$30$		& $245$\\
	$37$		& $294$\\
	$44$		& $-$\\
	$51$		& $-$\\
\hline
\end{tabular}
\end{table}
\end{veta}
\begin{proof} 
Since $\frac{3a_1(x)}{7}$ is a multiple of $7$, $x$ can by an automorphism of order $7$ of $\Gamma$ only if $a_1(x)$ is a multiple of $49$. For each possible value $a_0(x)$ there is at most one such $a_1(x)$ and there are two cases, specifically $a_0(x)\in\{44,51\}$, when there is no feasible value for $a_1(x)$. Each of the possible values of $a_1(x)$ satisfies also the condition which follows from the number of subgraphs isomorphic to $\overline{K_7}$.
\end{proof}

An analogy of this method can be used for any automorphism of prime order. We have used it also for automorphism of order $5$ but it led to no new results.\\

\subsection{A $srg(3250,57,0,1)$ with no induced Petersen graph}

The case when a $srg(3250,57,0,1)$ does not contain Petersen graph as induced subgraph is not excluded by our algorithm. Therefore the information about its automorphism group is naturally interesting.\\ 
Let $x$ be an automorphism of a prime order $p$ of $srg(3250,57,0,1)$ with no induced Petersen graph. Then the set $Fix(x)$, containing all vertices that are fixed by $x$, can induce neither Petersen graph nor Hoffman-Singleton graph. If subgraph $\Delta$ is fixed by $x$ and its vertex-set is not a subset of $Fix(x)$, then there has to be an automorphism of order $p$ of $\Delta$. So, if a value $P_G$ for some $G$ with no automorphism of order $p$ is congruent to $a \mod{p}$, then the set $Fix(x)$ contains $a \mod{p}$ copies of $G$. By analyzing induced subgraphs of higher orders we obtain the following.

\begin{lema}
A $srg(3250,57,0,1)$ with no induced Petersen graph does not have an automorphism of order $3$ and automorphism of order $5$ fixing the Hoffman-Singleton graph.
\end{lema}

\begin{prop}
The number of induced copies of any graph $G$ of order $10$ in a $srg(3250,57,0,1)$ with no induced Petersen graph is congruent to $0 \mod{p}$ for each $p\in\{7,11,13,19\}$.
\end{prop}

Hence, obtained results correspond to already known properties of a $srg(3250,57,0,1)$. We can see that the assumption that a $srg(3250,57,0,1)$ does not contains induced Petersen graph does not restrict the automorphism group of this graph. Therefore it can be interesting to look for Petersen-free Moore graphs.

\newpage
 \section{t-vertex condition in $SRG$}
In this chapter we introduce the concept of a $t$-vertex condition together with known results on this topic. We show a connection with our approach, which gives a way to extend our algorithm to solve problems in this area.
\par Let us start with a remark that each known tfSRG is highly symmetric. By this we mean that for all pairs of edges (non-edges) $(u,v)$, $(u',v')$ in a $tfSRG$ $\Gamma$ there exists an automorphism of this graph mapping $u$ onto $u'$ and $v$ onto $v'$, which is equivalent to saying that the full automorphism group is a rank-three group. Such graphs are called rank-three graphs.

\begin{prop}
Every rank-three graph is strongly regular.
\end{prop}
\begin{proof}
The number of common neighbors of any couple of vertices in graph $\Gamma$ has to be the same as the number of common neighbors of their images under some automorphism of $\Gamma$.
As $\Gamma$ is a rank-three graph, its automorphism group acts transitively on its edges and non-edges as well, from which follows that $\Gamma$ is an SRG. 
\end{proof}

This fascinating symmetry of existing tfSRG is a motivation for introducing the concept of a $t$-vertex condition, which is in a way a generalization of properties of rank-three graphs. This idea is discussed in \cite{tc} and we present a sketch here.
\par First, we need to define which subgraphs of $\Gamma$ will be considered equivalent. We say that two subgraphs of $\Gamma$ are of the same type with respect to a pair $(x,y)$ of distinct vertices if both contain vertices $x$, $y$ and there exists an isomorphism of one onto the other mapping $x$ to $x$ and $y$ to $y$. In the figure \ref{t_con_hr} there are all possible types of $3$-vertex subgraphs with respect to an edge $(x,y)$. Clearly, the last one cannot be contained in a tfSRG, but we consider also this case for full illustration.

%\begin{figure}
%	\begin{center}
%		\includegraphics[width=0.5\textwidth]{figures/t_con_hr.eps}
%		\caption{}
%		\label{t_con_hr}
%	\end{center}
%\end{figure}

\begin{figure}[h!]
	\centering
 \includegraphics[width=1\textwidth]{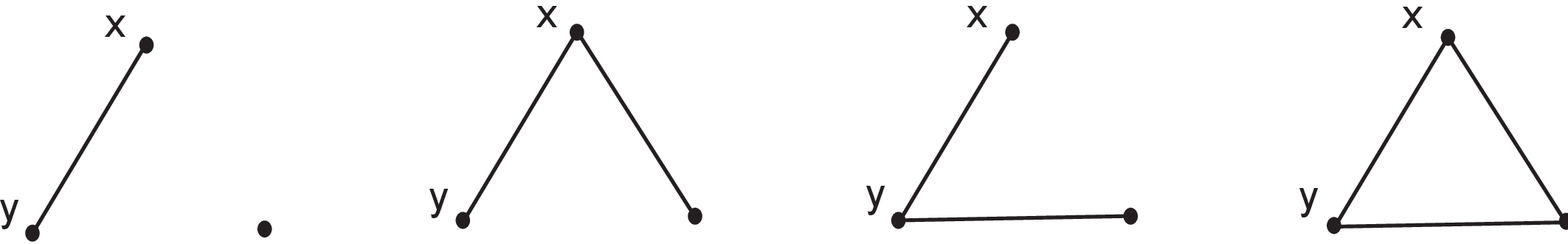}
	\caption{} 
	\label{t_con_hr}
\end{figure}

%\begin{figure}[!hb]
%	\centering
%	\includegraphics[width=0.5\textwidth]{figures/t_con_hr.eps}
%	\caption{}   
%	\label{t_con_hr}
%\end{figure}
%
%\bigskip

%\begin{minipage}{\linewidth}
%	\centering{
%		\includegraphics[width=.6\linewidth]{figures/t_con_hr.eps}
%		\captionof{figure}{}
%		\label{t_con_hr}
%	}
%\end{minipage}

\begin{definicia}
Let $t$ be an integer greater than one. We say that a graph $\Gamma$ satisfies a \emph{$t$-vertex condition} if the number of its $t$-vertex subgraphs of any type with respect to a given pair of distinct vertices $(x,y)$ depends only on whether $(x,y)$ is an edge or not.
\end{definicia}

It is easy to verify that each rank-three graph satisfies a $t$-vertex condition for all $t\geq 2$. The long-standing conjecture by M. Klin \cite{t-v-con} suggests that it could by interesting to consider graphs which are not rank three graphs and satisfy the $t$-vertex condition for large $t$. The conjecture predicts a number $t_0$ such that the only graphs satisfying the $t_0$-vertex condition are rank-three graphs. Since S. Reichard shows in \cite{t-v-sven} that for any fixed $t$ the $t$-vertex condition can be checked in polynomial time, a proof of Klin's conjecture would have the remarkable consequence that rank-three graphs can be recognised combinatorially in polynomial time.

Since not all SRGs are rank three graphs we look individually at SRGs and $t$-vertex condition for these graphs. Let us reformulate the definition of SRG in the following way. The first condition says that any edge of a $srg(n,k,\lambda,\mu)$ lies in $\lambda$ triangles. The second one determines the number of $K_{1,2}$ in which any non-edge occurs. Now, note that it is also possible to compute numbers of remaining subgraphs with respect to an edge or a non-edge using only $n$, $k$, $\lambda$ and $\mu$. For example the number of $K_{1,2}$ with respect to an edge $(x,y)$ with $y$ of order two is $k-\lambda-1$. Therefore the $3$-vertex condition holds for any SRG. However, the $t$-vertex condition for $t>3$ is not satisfied in general. In \cite{tc} the authors show the following property of $4$-vertex subgraphs of any SRG.

\begin{veta}\label{t4}
Let $\Gamma$ be a SRG. If the number of $4$-vertex subgraphs of $\Gamma$ of one type with respect to an edge (non-edge) $(x,y)$ is independent of the choice $(x,y)$, then so is the number of $4$-vertex subgraphs of each type with respect to $(x,y)$.
\end{veta}

Recently, S. Reichard improved in \cite{t-v-sven} this result in the following way. 

\begin{veta}
Let $\Gamma$ be an SRG satisfying a $t-1$-vertex condition. Then, in order to check that $\Gamma$ satisfies a $t$-vertex condition it suffices to consider graph types $G$ with respect to a pair of vertices $(x,y)$, in which each vertex from $V_G-\{x,y\}$ has valency at least $3$.
\end{veta}

Finally, C. Pech continued with this approach \cite{pech} and further reduced the list of graph types whose count matters when checking whether a graph satisfies the $5$-vertex and $6$-vertex condition.

As the number of induced $K_4$ in any tfSRG equals zero, it follows that the $4$-vertex condition holds in this case. Looking at the list provided by C. Pech it follows that any tfSRG satisfies the $5$-vertex condition as well. As some of the $6$-vertex graph types from \cite{pech} do not contain triangles the existence of a tfSRG which does not satisfy the $6$-vertex condition still remains an open problem. In a similar way it can be proven that any Moore graph of valency $57$ satisfies the $9$-vertex conditions. Let us note that these numbers are the same as our maximum orders for which the numbers of induced subgraphs of each type depend only on $k$ and $\mu$ in tfSRGs and only on $k$ in Moore graphs.

The main idea in the proof of Theorem \ref{t4} is very similar to our approach described in Section 3. The authors derive relations between induced subgraph of order $4$ with respect to a pair of vertices in an SRG. Therefore, at least in principle it is possible to derive all the numbers of subgraphs of order $4$ of given type with respect to a pair of vertices. On the other hand, S. Reichard and C. Pech obtained their results using other approaches. In particular, using their approach it is not possible to derive all the numbers of subgraphs of given type with respect to a pair of vertices or to provide different systems of graphs which are sufficient for checking the $t$-vertex conditions.

Therefore, one can expect that modification of our techniques will provide more detailed results than in \cite{pech}, \cite{t-v-sven}.
Moreover, the bounds on the number of induced subgraphs in an SRG further restrict possible numbers of subgraphs of given type with respect to a pair of vertices in graphs satisfying $t$-vertex condition for a given $t$.

\newpage
\newpage
\section*{Conclusion}
Strongly regular graphs are described by two simple conditions. They are regular and the number
of common neighbors of a pair of vertices in a given SRG depends only on whether they
are adjacent or not. In spite of the simplicity of this characterization there are significant applications of SRGs in combinatorics, algebra,
statistics, and group theory, including discovery of several sporadic simple groups.

The goal of this thesis was to find theoretical and algorithmic tools for determining the numbers of induced subgraphs in strongly regular graphs just from their parameters $n$, $k$, $\lambda$, $\mu$ and further applications of such numbers. We considered in more detail a restricted class of these graphs, specifically those with no triangles. In this special case, there are infinitely many feasible sets of parameters for SRGs. Despite this fact there are only seven known examples of such graphs.

One of the most important parts of our work is an algorithm which produces linear equations describing various relations between numbers of induced subgraphs of orders $o$ and $o-1$ in an $SRG$. The analysis of the systems of equations provided by our algorithm gives rise to the main result of the thesis. Namely, we are able to prove that in any $srg(n,k,0,\mu)$ the number of induced subgraphs isomorphic to a given graph $G$ on up to $7$ vertices depends only on parameters $k$ and $\mu$ and on the number of induced $K_{3,3}$s and $K_{3,4}$s. Moreover, in any $srg(3250,57,0,1)$, the number of induced subgraphs isomorphic to a given graph on up to $10$ vertices depends only on the number of induced Petersen graphs in this SRG. 

Among the applications of our methods, there were new results about automorphisms of $srg(3250,57,0,1)$ and new bounds for the numbers of induced $K_{3,3}$ in triangle-free SRGs. At the end of the thesis we discussed possible extension of our approach for the study of $t$-vertex condition.

Despite our best effort we were not able to show that Moore graph of valency $57$ has to contain a Petersen graph as induced subgraph. We feel that it is worthwhile to look for Petersen-free Moore graphs.

As it was observed in Section \ref{geometry} in the case of small triangle free Krein graphs, our computations provide tight bounds for induced $K_{3,3}$ subgraphs in these SRGs. It would be nice to have the confirmation of this observation for all triangle-free Krein graphs.
\addcontentsline{toc}{section}{\numberline{}Conclusion}

%\newpage

\newpage
\bibliographystyle{plain}	
\bibliography{bib_literatura}

\begin{thebibliography}{10}

\bibitem{Asb}
M.~Aschbacher.
\newblock {The nonexistence of rank three permutation group of degree $3250$
  and subdegree 57}.
\newblock {\em J. Algebra}, 19(3):538--540, 1971.

\bibitem{tfSRG}
N.~L. Biggs.
\newblock Strongly regular graphs with no triangles.
\newblock arXiv:0911.2160v1., September 2009.

\bibitem{klin}
A.~V. Bondarenko, A.~Prymak, and D.~Radchenko.
\newblock Non-existence of $(76, 30, 8, 14)$ strongly regular graph and some
  structural tools.
\newblock preprint.

\bibitem{Bos}
R.~C. Bose.
\newblock {Strongly Regular Graphs, Partial Geometries and Partially Balanced
  Designs}.
\newblock {\em Pacific J. Math.}, 13:389--419, 1963.

\bibitem{Bos_Mesner}
R.~C. Bose and D.~M. Mesner.
\newblock {On Linear Associative Algebras Corresponding to Association Schemes
  of Partially Balanced Designs}.
\newblock {\em Annals. Math. Stat.}, 36:21--38, 1959.

\bibitem{table2}
A.~E. Brouwer.
\newblock A table of parameters of strongly regular graphs.
\newblock www.win.tue.nl/~aeb.

\bibitem{BrHe}
A.~E. Brouwer and W.~H. Haemers.
\newblock {The Gewirtz Graph: An Exercise in the Theory of Graph Spectra}.
\newblock {\em Europ. J. Combinatorics}, 14:397--407, 1993.

\bibitem{spectra}
Andries~E. Brouwer and Willem~H. Haemers.
\newblock {\em Spectra of Graphs}.
\newblock Springer.

\bibitem{4conf}
D.~Bryant, M.~Grannell, T.~Griggs, and M.~Ma\v{c}aj.
\newblock {Configurations in 4-Cycle System }.
\newblock {\em Graphs and Combinatorics}, 20:161--179, 2004.

\bibitem{MacSir}
M.~Ma\v caj and J.~\v{S}ir\'a\v n.
\newblock {Search for properties of the missing Moore graph}.
\newblock {\em Linear Algebra and Its Applications}, pages 223--241, 2009.

\bibitem{Perm}
P.~Cameron.
\newblock {Permutation Groups}.
\newblock {\em Cambridge University Press}, 1999.

\bibitem{blog}
P.~Cameron.
\newblock The prehistory of the higman-sims graph.
\newblock Peter Cameron`s blog, November 2011.

\bibitem{PJC}
P.~Cameron.
\newblock Strongly regular graphs.
\newblock Queen Mary, University of London, April 2011.

\bibitem{subcon}
P.~J. Cameron, J.-M. Goethals, and J.~J. Seidel.
\newblock Strongly regular graphs with strongly regular subconstituent.
\newblock {\em J. Algebra}, 55:257--280, 1978.

\bibitem{cam_lintCR}
P.~J. Cameron and J.~H.~Van Lint.
\newblock {\em {Designs, Graphs, Codes and Their Links}}.
\newblock Cambridge University Press, 1981.

\bibitem{harm}
F.~Dai and Y.~Xu.
\newblock {\em {Approximation Theory and Harmonic Analysis on Spheres and
  Balls}}.
\newblock Springer, 2013.

\bibitem{del}
P.~Delsarte, J.-M. Goethals, and J.J.~Seidel and.
\newblock Spherical codes and designs.
\newblock {\em Geometriae Dedicata}, 6:363--388, 1977.

\bibitem{elz}
R.~J. Elzinga.
\newblock Strongly regular graphs: values of $\lambda$ and $\mu$ for which
  there are only finitely many feasible $(v, k, \lambda, \mu)$.
\newblock {\em Electron. J. Linear Algebra}, 10:232--239, 2003.

\bibitem{GavMak}
A.~L. Gavrilyuk and A.~A. Makhnev.
\newblock On krein graphs without triangles.
\newblock {\em Doklady Mathematics}, 72(1):591--594, 2005.

\bibitem{harm2}
C.~Helanow.
\newblock Spherical harmonics: a theoretical and graphical study.
\newblock Stockholms Univesitet, 2009.

\bibitem{tc}
M.~D. Hestenes and D.~G. Higman.
\newblock {Rank 3 Groups and Strongly Regular Graphs}.
\newblock {\em AMS}, 4:141--159, 1970.

\bibitem{HS}
A.J. Hoffman and R.R. Singleton.
\newblock {On Moore Graphs with Diameter 2 and 3}.
\newblock {\em IBM J. Res. Develop.}, 4:497--504, 1960.

\bibitem{Scott}
L.~L.~Scott Jr.
\newblock {A condition on Higmans parameters}.
\newblock {\em Notices Amer. Math. Soc.}, 20(A-97), 1973.

\bibitem{geometry}
M.~H. Klin and A.~J. Woldar.
\newblock The strongly regular graph with parameters $(100,22,0,6)$: hidden
  history and beyond.
\newblock preprint.

\bibitem{krc}
M.~G. Krein.
\newblock Hermitian-positive kernels ii.
\newblock {\em Amer. Math. Soc. Transl.}, 35(2):109--164, 1963.

\bibitem{MakPad}
A.~A. Makhnev and D.~V. Paduchikh.
\newblock {Automorphisms of Aschbachers graphs}.
\newblock {\em Algebra Logic}, 40(2):69--74, 2001.

\bibitem{orderly}
B.~D. McKay.
\newblock Isomorphism-free exhaustive generation.
\newblock {\em J Algorithms}, 26:306--324, 1998.

\bibitem{mcs}
B.~D. McKay and E.~Spence.
\newblock Classification of regular two-graphs on 36 and 38 vertices.
\newblock {\em Australasian J. of Combin.}, 24:293--300, 2001.

\bibitem{matan}
M.Klin and M.~Ziv-Av.
\newblock Computer algebra investigation of known primitive triangle-free
  strongly regular graphs.
\newblock Ben-Gurion University of the Negev, 2013.

\bibitem{pech}
C.~Pech.
\newblock On highly regular strongly regular graphs.
\newblock May 2014.

\bibitem{t-v-sven}
S.~Reichard.
\newblock Strongly regular graphs with the $7$-vertex condition.
\newblock Technische Universitat Dresden, Jan. 2014.

\bibitem{t-v-con}
I.~A.~Farad\v zev, M.~H., and M.~E. Muzychuk.
\newblock Cellular rings and groups of automorphisms of graphs. in
  investigations in algebraic theory of combinatorial objects.
\newblock {\em Math. Appl. (Soviet Ser.)}, 84:1--152, 1994.

\end{thebibliography}
\newpage
\newpage
\section*{Appendix}
For illustration of the form of our results, we present number of induced subraphs for selected orders. Upon interest, we are able to offer also other numbers in some more appropriate format.
\subsection*{tfSRG}
order $4$:
\begin{multicols}{4}
\scriptsize{
\setbox\ltmcbox\vbox{
\makeatletter\col@number\@ne
\begin{longtable} {@{\stepcounter{rowcount}\therowcount.\hspace*{\tabcolsep}}l}
$\begin{matrix}  &  &  &  &  & \\
0 & 0 & 1 & 1 \\
0 & 0 & 1 & 1 \\ 
1 & 1 & 0 & 0 \\
1 & 1 & 0 & 0 \end{matrix}$\\
$\begin{matrix}  &  &  &  &  & \\ 
0 & 0 & 0 & 1 \\ 
0 & 0 & 0 & 1 \\ 
0 & 0 & 0 & 1 \\
1 & 1 & 1 & 0 \end{matrix}$\\
$\begin{matrix}  &  &  &  &  & \\ 
0 & 0 & 1 & 0 \\
0 & 0 & 0 & 1 \\
1 & 0 & 0 & 1 \\
0 & 1 & 1 & 0 \end{matrix}$\\
$\begin{matrix}  &  &  &  &  & \\
0 & 0 & 0 & 0 \\
0 & 0 & 0 & 1 \\
0 & 0 & 0 & 1 \\
0 & 1 & 1 & 0 \end{matrix}$\\
$\begin{matrix}  &  &  &  &  & \\
0 & 0 & 0 & 1 \\
0 & 0 & 1 & 0 \\
0 & 1 & 0 & 0 \\
1 & 0 & 0 & 0 \end{matrix}$\\
$\begin{matrix}  &  &  &  &  & \\
0 & 0 & 0 & 0 \\
0 & 0 & 0 & 0 \\
0 & 0 & 0 & 1 \\
0 & 0 & 1 & 0 \end{matrix}$\\
$\begin{matrix}  &  &  &  &  & \\
0 & 0 & 0 & 0 \\
0 & 0 & 0 & 0 \\
0 & 0 & 0 & 0 \\
0 & 0 & 0 & 0 \end{matrix}$\\
\end{longtable}
\unskip
\unpenalty
\unpenalty}
\unvbox\ltmcbox
}
\end{multicols}
\setcounter{rowcount}{0}

\scriptsize{
\begin{longtable}{@{\stepcounter{rowcount}\therowcount.\hspace*{\tabcolsep}}p{15cm}}
$1/8 (\mu +k \mu +k^2-k) k (k-1)/\mu  (\mu -1)$\\
$1/6 (\mu +k \mu +k^2-k) k (k-1)/\mu  (k-2)$\\
$1/2 (\mu +k \mu +k^2-k) k (k-1) (k-\mu )/\mu $\\
$1/2 (\mu +k \mu +k^2-k) k (k-1) (-2 k \mu +k^2-k+\mu +\mu ^2)/\mu ^2$\\
$1/8 (\mu +k \mu +k^2-k) k (3 k \mu -3 k^2 \mu +k^3-k^2+2 k \mu ^2-2 \mu ^2)/\mu ^2$\\
$1/4 (\mu +k \mu +k^2-k) k (-k \mu -3 \mu  k^3+k^4-2 k^3+k^2+4 k^2 \mu ^2+4 k^2 \mu -4 k \mu ^2-2 k \mu ^3+2 \mu ^3)/\mu ^3$\\
$1/24 (\mu +k \mu +k^2-k) k (-3 k \mu -k^2-2 \mu ^2+3 k^2 \mu +2 k \mu ^2+3 k^3+3 k \mu ^3+3 \mu  k^3-6 k^2 \mu ^3-3 k^4 \mu +k^5-3 k^4+6 k^3 \mu ^2-6 k^2 \mu ^2+3 \mu ^4 k-3 \mu ^4+3 \mu ^3)/\mu ^4$\\
\end{longtable}
}
\setcounter{rowcount}{0}
\normalsize
order $5$:

\begin{multicols}{4}
\scriptsize{
\setbox\ltmcbox\vbox{
\makeatletter\col@number\@ne
\begin{longtable} {@{\stepcounter{rowcount}\therowcount.\hspace*{\tabcolsep}}l}
$\begin{matrix}  &  &  &  &  & \\0 & 0 & 0 & 1 & 1 \\ 
0 & 0 & 0 & 1 & 1 \\ 
0 & 0 & 0 & 1 & 1 \\ 
1 & 1 & 1 & 0 & 0 \\ 
1 & 1 & 1 & 0 & 0 \end{matrix}$\\
$\begin{matrix}  &  &  &  &  & \\ 
0 & 0 & 0 & 0 & 1 \\ 
0 & 0 & 0 & 1 & 1 \\ 
0 & 0 & 0 & 1 & 1 \\ 
0 & 1 & 1 & 0 & 0 \\ 
1 & 1 & 1 & 0 & 0 \end{matrix}$\\
$\begin{matrix}  &  &  &  &  & \\ 
0 & 0 & 0 & 0 & 0 \\ 
0 & 0 & 1 & 1 & 0 \\ 
0 & 1 & 0 & 0 & 1 \\ 
0 & 1 & 0 & 0 & 1 \\ 
0 & 0 & 1 & 1 & 0 \end{matrix}$\\
$\begin{matrix}  &  &  &  &  & \\ 
0 & 0 & 0 & 0 & 1 \\ 
0 & 0 & 0 & 0 & 1 \\ 
0 & 0 & 0 & 0 & 1 \\ 
0 & 0 & 0 & 0 & 1 \\ 
1 & 1 & 1 & 1 & 0 \end{matrix}$\\
$\begin{matrix}  &  &  &  &  & \\ 
0 & 0 & 0 & 0 & 1 \\ 
0 & 0 & 0 & 0 & 1 \\ 
0 & 0 & 0 & 1 & 0 \\ 
0 & 0 & 1 & 0 & 1 \\ 
1 & 1 & 0 & 1 & 0 \end{matrix}$\\
$\begin{matrix}  &  &  &  &  & \\ 
0 & 0 & 0 & 0 & 0 \\ 
0 & 0 & 0 & 0 & 1 \\ 
0 & 0 & 0 & 0 & 1 \\ 
0 & 0 & 0 & 0 & 1 \\ 
0 & 1 & 1 & 1 & 0 \end{matrix}$\\
$\begin{matrix}  &  &  &  &  & \\ 
0 & 0 & 1 & 0 & 1 \\ 
0 & 0 & 0 & 1 & 1 \\ 
1 & 0 & 0 & 1 & 0 \\ 
0 & 1 & 1 & 0 & 0 \\ 
1 & 1 & 0 & 0 & 0 \end{matrix}$\\
$\begin{matrix}  &  &  &  &  & \\ 
0 & 0 & 0 & 1 & 0 \\ 
0 & 0 & 0 & 0 & 1 \\ 
0 & 0 & 0 & 1 & 1 \\ 
1 & 0 & 1 & 0 & 0 \\ 
0 & 1 & 1 & 0 & 0 \end{matrix}$\\
$\begin{matrix}  &  &  &  &  & \\ 
0 & 0 & 0 & 0 & 0 \\ 
0 & 0 & 0 & 1 & 0 \\ 
0 & 0 & 0 & 0 & 1 \\ 
0 & 1 & 0 & 0 & 1 \\ 
0 & 0 & 1 & 1 & 0 \end{matrix}$\\
$\begin{matrix}  &  &  &  &  & \\ 
0 & 0 & 0 & 1 & 0 \\ 
0 & 0 & 0 & 0 & 1 \\ 
0 & 0 & 0 & 0 & 1 \\ 
1 & 0 & 0 & 0 & 0 \\ 
0 & 1 & 1 & 0 & 0 \end{matrix}$\\
$\begin{matrix}  &  &  &  &  & \\ 
0 & 0 & 0 & 0 & 0 \\ 
0 & 0 & 0 & 0 & 0 \\ 
0 & 0 & 0 & 0 & 1 \\ 
0 & 0 & 0 & 0 & 1 \\ 
0 & 0 & 1 & 1 & 0 \end{matrix}$\\
$\begin{matrix}  &  &  &  &  & \\ 
0 & 0 & 0 & 0 & 0 \\ 
0 & 0 & 1 & 0 & 0 \\ 
0 & 1 & 0 & 0 & 0 \\ 
0 & 0 & 0 & 0 & 1 \\ 
0 & 0 & 0 & 1 & 0 \end{matrix}$\\
$\begin{matrix}  &  &  &  &  & \\ 
0 & 0 & 0 & 0 & 0 \\ 
0 & 0 & 0 & 0 & 0 \\ 
0 & 0 & 0 & 0 & 0 \\ 
0 & 0 & 0 & 0 & 1 \\ 
0 & 0 & 0 & 1 & 0 \end{matrix}$\\
$\begin{matrix}  &  &  &  &  & \\ 
0 & 0 & 0 & 0 & 0 \\ 
0 & 0 & 0 & 0 & 0 \\ 
0 & 0 & 0 & 0 & 0 \\ 
0 & 0 & 0 & 0 & 0 \\ 
0 & 0 & 0 & 0 & 0 \end{matrix}$\\
\end{longtable}
\unskip
\unpenalty
\unpenalty}
\unvbox\ltmcbox
}
\end{multicols}
\setcounter{rowcount}{0}
\newpage
\scriptsize{
\begin{longtable}{@{\stepcounter{rowcount}\therowcount.\hspace*{\tabcolsep}}p{15cm}}
$1/12 (\mu +k \mu +k^2-k) k (k-1) (\mu -1) (\mu -2)/\mu  $\\
$1/2 (\mu +k \mu +k^2-k) k (k-1) (\mu -1) (k-\mu )/\mu  $\\
$1/8 (\mu +k \mu +k^2-k) k (k-1) (\mu -1) (\mu -3 k \mu +k^2-k+2 \mu ^2)/\mu ^2 $\\
$1/24 (\mu +k \mu +k^2-k) k (k-1)/\mu  (k-2) (k-3) $\\
$1/2 (\mu +k \mu +k^2-k) k (k-1) (-2 k \mu +k^2-k+\mu +\mu ^2)/\mu  $\\
$1/6 (\mu +k \mu +k^2-k) k (k-1) (k^3-3 k^2-3 k^2 \mu +3 k \mu ^2+2 k+6 k \mu -3 \mu ^2-\mu ^3-2 \mu )/\mu ^2 $\\
$1/10 (\mu +k \mu +k^2-k) k (k-1) (k-\mu ) $\\
$1/2 (\mu +k \mu +k^2-k) k (k-1) (-3 k \mu +k^2+2 \mu ^2)/\mu  $\\
$1/2 (\mu +k \mu +k^2-k) k (k-1) (-4 k^2 \mu +k^3-k^2+6 k \mu ^2-\mu ^2-3 \mu ^3+2 k \mu )/\mu ^2 $\\
$1/4 (\mu +k \mu +k^2-k) k (k-1) (-5 k^2 \mu +k^3-k^2+3 k \mu +8 k \mu ^2-2 \mu ^2-4 \mu ^3)/\mu ^2 $\\
$1/4 (\mu +k \mu +k^2-k) k (k-1) (-7 k \mu ^2-5 \mu  k^3+k^4-2 k^3+k^2+3 \mu ^3+11 k^2 \mu ^2+6 k^2 \mu -12 k \mu ^3+5 \mu ^4-k \mu )/\mu ^3 $\\
$1/8 (\mu +k \mu +k^2-k) k (5 k \mu ^2-20 k^2 \mu ^2+10 \mu  k^3+15 k^3 \mu ^2-6 k^4 \mu +k^5-2 k^4+k^3+20 k \mu ^3-18 k^2 \mu ^3-2 \mu ^3-4 k^2 \mu -8 \mu ^4+8 \mu ^4 k)/\mu ^3 $\\
$1/12 (\mu +k \mu +k^2-k) k (k \mu ^2-k^3+8 k^2 \mu ^2-5 k \mu ^3-6 \mu  k^3-27 k^3 \mu ^2+36 k^2 \mu ^3+12 k^4 \mu -3 k^5+3 k^4-30 \mu ^4 k+12 \mu ^5-31 k^3 \mu ^3+30 k^2 \mu ^4+18 k^4 \mu ^2-6 \mu  k^5-12 \mu ^5 k+k^6)/\mu ^4 $\\
$1/120 (\mu +k \mu +k^2-k) k (6 k^2 \mu +11 k \mu ^2+6 \mu ^3+k^3-11 \mu ^4-10 k^2 \mu ^2-13 k \mu ^3-12 \mu  k^3+8 k^3 \mu ^2+10 k^2 \mu ^3+6 k^5-4 k^4-4 \mu ^4 k+15 \mu ^5+43 k^3 \mu ^3-51 k^2 \mu ^4-30 k^4 \mu ^2+12 \mu  k^5+42 \mu ^5 k-46 k^4 \mu ^3+66 k^3 \mu ^4+21 k^5 \mu ^2-6 k^6 \mu -57 k^2 \mu ^5+k^7-4 k^6+22 \mu ^6 k-22 \mu ^6)/\mu ^5 $\\
\end{longtable}
}
\setcounter{rowcount}{0}
\normalsize
order $6$:

\begin{multicols}{3}
\scriptsize
\setbox\ltmcbox\vbox{
\makeatletter\col@number\@ne
\begin{longtable} {@{\stepcounter{rowcount}\therowcount.\hspace*{\tabcolsep}}l}
$\begin{matrix}  &  &  &  &  & \\0 & 0 & 0 & 1 & 1 & 1\\
0 & 0 & 0 & 1 & 1 & 1\\
0 & 0 & 0 & 1 & 1 & 1\\
1 & 1 & 1 & 0 & 0 & 0\\
1 & 1 & 1 & 0 & 0 & 0\\
1 & 1 & 1 & 0 & 0 & 0\end{matrix}$\\
$\begin{matrix}  &  &  &  &  & \\0 & 0 & 0 & 1 & 1 & 0\\
0 & 0 & 1 & 0 & 0 & 1\\
0 & 1 & 0 & 1 & 1 & 0\\
1 & 0 & 1 & 0 & 0 & 1\\
1 & 0 & 1 & 0 & 0 & 1\\
0 & 1 & 0 & 1 & 1 & 0\end{matrix}$\\
$\begin{matrix}  &  &  &  &  & \\0 & 0 & 0 & 0 & 1 & 1\\
0 & 0 & 0 & 0 & 1 & 1\\
0 & 0 & 0 & 0 & 1 & 1\\
0 & 0 & 0 & 0 & 1 & 1\\
1 & 1 & 1 & 1 & 0 & 0\\
1 & 1 & 1 & 1 & 0 & 0\end{matrix}$\\
$\begin{matrix}  &  &  &  &  & \\0 & 0 & 0 & 0 & 0 & 1\\
0 & 0 & 0 & 0 & 1 & 1\\
0 & 0 & 0 & 0 & 1 & 1\\
0 & 0 & 0 & 0 & 1 & 1\\
0 & 1 & 1 & 1 & 0 & 0\\
1 & 1 & 1 & 1 & 0 & 0\end{matrix}$\\
$\begin{matrix}  &  &  &  &  & \\0 & 0 & 0 & 0 & 0 & 1\\
0 & 0 & 0 & 1 & 1 & 0\\
0 & 0 & 0 & 1 & 1 & 0\\
0 & 1 & 1 & 0 & 0 & 1\\
0 & 1 & 1 & 0 & 0 & 1\\
1 & 0 & 0 & 1 & 1 & 0\end{matrix}$\\
$\begin{matrix}  &  &  &  &  & \\0 & 0 & 0 & 0 & 0 & 0\\
0 & 0 & 0 & 0 & 1 & 1\\
0 & 0 & 0 & 0 & 1 & 1\\
0 & 0 & 0 & 0 & 1 & 1\\
0 & 1 & 1 & 1 & 0 & 0\\
0 & 1 & 1 & 1 & 0 & 0\end{matrix}$\\
$\begin{matrix}  &  &  &  &  & \\0 & 0 & 0 & 1 & 1 & 0\\
0 & 0 & 1 & 0 & 1 & 0\\
0 & 1 & 0 & 0 & 0 & 1\\
1 & 0 & 0 & 0 & 0 & 1\\
1 & 1 & 0 & 0 & 0 & 1\\
0 & 0 & 1 & 1 & 1 & 0\end{matrix}$\\
$\begin{matrix}  &  &  &  &  & \\0 & 0 & 0 & 1 & 1 & 0\\
0 & 0 & 0 & 0 & 1 & 1\\
0 & 0 & 0 & 0 & 1 & 1\\
1 & 0 & 0 & 0 & 0 & 1\\
1 & 1 & 1 & 0 & 0 & 0\\
0 & 1 & 1 & 1 & 0 & 0\end{matrix}$\\
$\begin{matrix}  &  &  &  &  & \\0 & 0 & 0 & 0 & 0 & 1\\
0 & 0 & 0 & 0 & 0 & 1\\
0 & 0 & 0 & 1 & 0 & 1\\
0 & 0 & 1 & 0 & 1 & 0\\
0 & 0 & 0 & 1 & 0 & 1\\
1 & 1 & 1 & 0 & 1 & 0\end{matrix}$\\
$\begin{matrix}  &  &  &  &  & \\0 & 0 & 0 & 0 & 1 & 0\\
0 & 0 & 0 & 0 & 0 & 1\\
0 & 0 & 0 & 0 & 1 & 1\\
0 & 0 & 0 & 0 & 1 & 1\\
1 & 0 & 1 & 1 & 0 & 0\\
0 & 1 & 1 & 1 & 0 & 0\end{matrix}$\\
$\begin{matrix}  &  &  &  &  & \\0 & 0 & 0 & 0 & 1 & 0\\
0 & 0 & 0 & 0 & 0 & 1\\
0 & 0 & 0 & 1 & 0 & 1\\
0 & 0 & 1 & 0 & 1 & 0\\
1 & 0 & 0 & 1 & 0 & 1\\
0 & 1 & 1 & 0 & 1 & 0\end{matrix}$\\
$\begin{matrix}  &  &  &  &  & \\0 & 0 & 0 & 0 & 1 & 0\\
0 & 0 & 0 & 1 & 0 & 1\\
0 & 0 & 0 & 1 & 0 & 1\\
0 & 1 & 1 & 0 & 0 & 0\\
1 & 0 & 0 & 0 & 0 & 1\\
0 & 1 & 1 & 0 & 1 & 0\end{matrix}$\\
$\begin{matrix}  &  &  &  &  & \\0 & 0 & 0 & 0 & 0 & 0\\
0 & 0 & 0 & 0 & 0 & 1\\
0 & 0 & 0 & 1 & 0 & 1\\
0 & 0 & 1 & 0 & 1 & 0\\
0 & 0 & 0 & 1 & 0 & 1\\
0 & 1 & 1 & 0 & 1 & 0\end{matrix}$\\
$\begin{matrix}  &  &  &  &  & \\0 & 1 & 0 & 0 & 0 & 0\\
1 & 0 & 0 & 0 & 0 & 0\\
0 & 0 & 0 & 0 & 1 & 1\\
0 & 0 & 0 & 0 & 1 & 1\\
0 & 0 & 1 & 1 & 0 & 0\\
0 & 0 & 1 & 1 & 0 & 0\end{matrix}$\\
$\begin{matrix}  &  &  &  &  & \\0 & 0 & 0 & 0 & 0 & 0\\
0 & 0 & 0 & 0 & 0 & 0\\
0 & 0 & 0 & 0 & 1 & 1\\
0 & 0 & 0 & 0 & 1 & 1\\
0 & 0 & 1 & 1 & 0 & 0\\
0 & 0 & 1 & 1 & 0 & 0\end{matrix}$\\
\end{longtable}
\unskip
\unpenalty
\unpenalty}

\unvbox\ltmcbox
\end{multicols}
\newpage
\begin{multicols}{3}
\scriptsize
\setbox\ltmcbox\vbox{
\makeatletter\col@number\@ne
\begin{longtable} {@{\stepcounter{rowcount}\therowcount.\hspace*{\tabcolsep}}l}
$\begin{matrix}  &  &  &  &  & \\0 & 0 & 0 & 0 & 0 & 1\\
0 & 0 & 0 & 0 & 0 & 1\\
0 & 0 & 0 & 0 & 0 & 1\\
0 & 0 & 0 & 0 & 0 & 1\\
0 & 0 & 0 & 0 & 0 & 1\\
1 & 1 & 1 & 1 & 1 & 0\end{matrix}$\\
$\begin{matrix}  &  &  &  &  & \\0 & 0 & 0 & 0 & 0 & 1\\
0 & 0 & 0 & 0 & 0 & 1\\
0 & 0 & 0 & 0 & 0 & 1\\
0 & 0 & 0 & 0 & 1 & 0\\
0 & 0 & 0 & 1 & 0 & 1\\
1 & 1 & 1 & 0 & 1 & 0\end{matrix}$\\
$\begin{matrix}  &  &  &  &  & \\0 & 0 & 0 & 0 & 0 & 0\\
0 & 0 & 0 & 0 & 0 & 1\\
0 & 0 & 0 & 0 & 0 & 1\\
0 & 0 & 0 & 0 & 0 & 1\\
0 & 0 & 0 & 0 & 0 & 1\\
0 & 1 & 1 & 1 & 1 & 0\end{matrix}$\\
$\begin{matrix}  &  &  &  &  & \\0 & 0 & 0 & 0 & 0 & 1\\
0 & 0 & 0 & 0 & 1 & 1\\
0 & 0 & 0 & 1 & 1 & 0\\
0 & 0 & 1 & 0 & 0 & 1\\
0 & 1 & 1 & 0 & 0 & 0\\
1 & 1 & 0 & 1 & 0 & 0\end{matrix}$\\
$\begin{matrix}  &  &  &  &  & \\0 & 0 & 0 & 0 & 1 & 0\\
0 & 0 & 0 & 0 & 1 & 0\\
0 & 0 & 0 & 0 & 0 & 1\\
0 & 0 & 0 & 0 & 0 & 1\\
1 & 1 & 0 & 0 & 0 & 1\\
0 & 0 & 1 & 1 & 1 & 0\end{matrix}$\\
$\begin{matrix}  &  &  &  &  & \\0 & 0 & 0 & 0 & 0 & 1\\
0 & 0 & 0 & 0 & 0 & 1\\
0 & 0 & 0 & 0 & 1 & 0\\
0 & 0 & 0 & 0 & 1 & 1\\
0 & 0 & 1 & 1 & 0 & 0\\
1 & 1 & 0 & 1 & 0 & 0\end{matrix}$\\
$\begin{matrix}  &  &  &  &  & \\0 & 0 & 0 & 0 & 0 & 1\\
0 & 0 & 0 & 1 & 0 & 0\\
0 & 0 & 0 & 0 & 1 & 0\\
0 & 1 & 0 & 0 & 0 & 1\\
0 & 0 & 1 & 0 & 0 & 1\\
1 & 0 & 0 & 1 & 1 & 0\end{matrix}$\\
$\begin{matrix}  &  &  &  &  & \\0 & 0 & 0 & 0 & 0 & 0\\
0 & 0 & 0 & 0 & 0 & 1\\
0 & 0 & 0 & 0 & 1 & 0\\
0 & 0 & 0 & 0 & 0 & 1\\
0 & 0 & 1 & 0 & 0 & 1\\
0 & 1 & 0 & 1 & 1 & 0\end{matrix}$\\
$\begin{matrix}  &  &  &  &  & \\0 & 0 & 0 & 0 & 1 & 0\\
0 & 0 & 0 & 0 & 0 & 1\\
0 & 0 & 0 & 0 & 0 & 1\\
0 & 0 & 0 & 0 & 0 & 1\\
1 & 0 & 0 & 0 & 0 & 0\\
0 & 1 & 1 & 1 & 0 & 0\end{matrix}$\\
$\begin{matrix}  &  &  &  &  & \\0 & 0 & 0 & 0 & 0 & 0\\
0 & 0 & 0 & 0 & 0 & 0\\
0 & 0 & 0 & 0 & 0 & 1\\
0 & 0 & 0 & 0 & 0 & 1\\
0 & 0 & 0 & 0 & 0 & 1\\
0 & 0 & 1 & 1 & 1 & 0\end{matrix}$\\
$\begin{matrix}  &  &  &  &  & \\0 & 0 & 0 & 0 & 0 & 0\\
0 & 0 & 0 & 1 & 1 & 0\\
0 & 0 & 0 & 1 & 0 & 1\\
0 & 1 & 1 & 0 & 0 & 0\\
0 & 1 & 0 & 0 & 0 & 1\\
0 & 0 & 1 & 0 & 1 & 0\end{matrix}$\\
$\begin{matrix}  &  &  &  &  & \\0 & 0 & 0 & 1 & 0 & 1\\
0 & 0 & 0 & 0 & 1 & 1\\
0 & 0 & 0 & 1 & 1 & 0\\
1 & 0 & 1 & 0 & 0 & 0\\
0 & 1 & 1 & 0 & 0 & 0\\
1 & 1 & 0 & 0 & 0 & 0\end{matrix}$\\
$\begin{matrix}  &  &  &  &  & \\0 & 0 & 0 & 0 & 1 & 0\\
0 & 0 & 0 & 0 & 0 & 1\\
0 & 0 & 0 & 1 & 1 & 0\\
0 & 0 & 1 & 0 & 0 & 1\\
1 & 0 & 1 & 0 & 0 & 0\\
0 & 1 & 0 & 1 & 0 & 0\end{matrix}$\\
$\begin{matrix}  &  &  &  &  & \\0 & 0 & 0 & 0 & 0 & 0\\
0 & 0 & 0 & 0 & 1 & 0\\
0 & 0 & 0 & 1 & 0 & 0\\
0 & 0 & 1 & 0 & 0 & 1\\
0 & 1 & 0 & 0 & 0 & 1\\
0 & 0 & 0 & 1 & 1 & 0\end{matrix}$\\
$\begin{matrix}  &  &  &  &  & \\0 & 0 & 0 & 1 & 0 & 0\\
0 & 0 & 0 & 0 & 0 & 1\\
0 & 0 & 0 & 0 & 1 & 0\\
1 & 0 & 0 & 0 & 0 & 0\\
0 & 0 & 1 & 0 & 0 & 1\\
0 & 1 & 0 & 0 & 1 & 0\end{matrix}$\\
$\begin{matrix}  &  &  &  &  & \\0 & 0 & 0 & 0 & 0 & 0\\
0 & 0 & 0 & 0 & 0 & 0\\
0 & 0 & 0 & 0 & 1 & 0\\
0 & 0 & 0 & 0 & 0 & 1\\
0 & 0 & 1 & 0 & 0 & 1\\
0 & 0 & 0 & 1 & 1 & 0\end{matrix}$\\
$\begin{matrix}  &  &  &  &  & \\0 & 0 & 0 & 0 & 1 & 0\\
0 & 0 & 0 & 0 & 1 & 0\\
0 & 0 & 0 & 0 & 0 & 1\\
0 & 0 & 0 & 0 & 0 & 1\\
1 & 1 & 0 & 0 & 0 & 0\\
0 & 0 & 1 & 1 & 0 & 0\end{matrix}$\\
$\begin{matrix}  &  &  &  &  & \\0 & 0 & 0 & 0 & 0 & 0\\
0 & 0 & 0 & 0 & 0 & 1\\
0 & 0 & 0 & 0 & 0 & 1\\
0 & 0 & 0 & 0 & 1 & 0\\
0 & 0 & 0 & 1 & 0 & 0\\
0 & 1 & 1 & 0 & 0 & 0\end{matrix}$\\
$\begin{matrix}  &  &  &  &  & \\0 & 0 & 0 & 0 & 0 & 0\\
0 & 0 & 0 & 0 & 0 & 0\\
0 & 0 & 0 & 0 & 0 & 0\\
0 & 0 & 0 & 0 & 0 & 1\\
0 & 0 & 0 & 0 & 0 & 1\\
0 & 0 & 0 & 1 & 1 & 0\end{matrix}$\\
$\begin{matrix}  &  &  &  &  & \\0 & 0 & 0 & 0 & 0 & 1\\
0 & 0 & 1 & 0 & 0 & 0\\
0 & 1 & 0 & 0 & 0 & 0\\
0 & 0 & 0 & 0 & 1 & 0\\
0 & 0 & 0 & 1 & 0 & 0\\
1 & 0 & 0 & 0 & 0 & 0\end{matrix}$\\
$\begin{matrix}  &  &  &  &  & \\0 & 0 & 0 & 0 & 0 & 0\\
0 & 0 & 0 & 0 & 0 & 0\\
0 & 0 & 0 & 0 & 0 & 1\\
0 & 0 & 0 & 0 & 1 & 0\\
0 & 0 & 0 & 1 & 0 & 0\\
0 & 0 & 1 & 0 & 0 & 0\end{matrix}$\\
$\begin{matrix}  &  &  &  &  & \\0 & 0 & 0 & 0 & 0 & 0\\
0 & 0 & 0 & 0 & 0 & 0\\
0 & 0 & 0 & 0 & 0 & 0\\
0 & 0 & 0 & 0 & 0 & 0\\
0 & 0 & 0 & 0 & 0 & 1\\
0 & 0 & 0 & 0 & 1 & 0\end{matrix}$\\
$\begin{matrix}  &  &  &  &  & \\0 & 0 & 0 & 0 & 0 & 0\\
0 & 0 & 0 & 0 & 0 & 0\\
0 & 0 & 0 & 0 & 0 & 0\\
0 & 0 & 0 & 0 & 0 & 0\\
0 & 0 & 0 & 0 & 0 & 0\\
0 & 0 & 0 & 0 & 0 & 0\end{matrix}$\\
\end{longtable}
\unskip
\unpenalty
\unpenalty}

\unvbox\ltmcbox
\end{multicols}
\setcounter{rowcount}{0}
\newpage
\scriptsize
\begin{longtable}{@{\stepcounter{rowcount}\therowcount.\hspace*{\tabcolsep}}p{15cm}}
$P_1$\\
$1/8(-20\mu k^3+k^3\mu ^4+k^4\mu ^3-5k^4\mu ^2+k^2\mu ^3+8k^2\mu -5k^2\mu ^2+5k\mu ^3-8k\mu ^2-4k^2+4k\mu +8k^3-\mu ^4k-4k^4-7k^3\mu ^3+8k^4\mu +18k^3\mu ^2-72P_1\mu )/\mu $\\
$1/48(\mu +k\mu +k^2-k)k(k-1)(\mu -1)(\mu -2)/\mu (\mu -3)$\\
$1/6(\mu +k\mu +k^2-k)k(k-1)(\mu -1)(\mu -2)(k-\mu )/\mu $\\
$1/4(k^5\mu ^2+21\mu k^3-2k^3\mu ^4-k^4\mu ^3-3\mu k^5+3k^4\mu ^2-3k^2\mu ^3-12k^2\mu +11k^2\mu ^2-8k\mu ^3+10k\mu ^2+4k^2-4k\mu -6k^3+2\mu ^4k+2k^5+12k^3\mu ^3-2k^4\mu -25k^3\mu ^2+72P_1\mu )/\mu $\\
$1/12(-15k^4\mu ^2+8k^2\mu ^4+32k^3\mu ^3+6k^5\mu ^2+3k^4\mu +6k^4\mu ^3-10k^3\mu ^2-11k\mu ^3+2k\mu ^2-3k^6\mu -4\mu ^5k-21k^3\mu ^4+6k^4-6k^5-2k^3+2k^6+13\mu ^4k-24k^2\mu ^3-3\mu k^3-3k^5\mu ^3-72P_1\mu ^2+4k^3\mu ^5+k^6\mu ^2+3\mu k^5+16k^2\mu ^2)/\mu ^2$\\
$1/4(k^5\mu ^2+20\mu k^3-2k^3\mu ^4-k^4\mu ^3-2\mu k^5+3k^4\mu ^2-3k^2\mu ^3-10k^2\mu +9k^2\mu ^2-7k\mu ^3+9k\mu ^2+4k^2-4k\mu -7k^3+2\mu ^4k+k^5+2k^4+11k^3\mu ^3-4k^4\mu -22k^3\mu ^2+72P_1\mu )/\mu $\\
$1/4(\mu +k\mu +k^2-k)k(k-1)(\mu -1)(k-\mu )$\\
$1/4(\mu +k\mu +k^2-k)k(k-1)(\mu -1)(-2k\mu +k^2-k+\mu +\mu ^2)/\mu $\\
$1/4(\mu +k\mu +k^2-k)k(k-1)(\mu -1)(-3k\mu +k^2+2\mu ^2)/\mu $\\
$1/2(-2k^5\mu ^2-17\mu k^3+3k^3\mu ^4+\mu k^5+2k^4\mu ^2+6k^2\mu ^3+10k^2\mu -12k^2\mu ^2+8k\mu ^3-9k\mu ^2-4k^2+4k\mu +7k^3-3\mu ^4k+k^5-k^6-3k^4-14k^3\mu ^3+k^4\mu +k^6\mu +21k^3\mu ^2-72P_1\mu )/\mu $\\
$1/2(-3k^5\mu ^2-16\mu k^3+4k^3\mu ^4+2\mu k^5+4k^4\mu ^2+8k^2\mu ^3+10k^2\mu -14k^2\mu ^2+9k\mu ^3-9k\mu ^2-4k^2+4k\mu +7k^3-4\mu ^4k+k^5-k^6-3k^4-17k^3\mu ^3-k^4\mu +k^6\mu +22k^3\mu ^2-72P_1\mu )/\mu $\\
$1/2(k^4\mu ^2+4k^2\mu -15k^2\mu ^4-14k^3\mu ^3+6k^5\mu ^2+11k^4\mu -16k^4\mu ^3+8k^3\mu ^2+8k\mu ^3-4k\mu ^2+\mu k^7+k^6\mu +6\mu ^5k+22k^3\mu ^4+k^4-3k^5+3k^6-k^7-10\mu ^4k+18k^2\mu ^3-10\mu k^3+4k^5\mu ^3+72P_1\mu ^2-6k^3\mu ^5+3k^4\mu ^4-4k^6\mu ^2-7\mu k^5-7k^2\mu ^2)/\mu ^2$\\
$1/16(5k^4\mu ^2+8k^2\mu -28k^2\mu ^4-33k^3\mu ^3+5k^5\mu ^2+17k^4\mu -27k^4\mu ^3+19k^3\mu ^2+16k\mu ^3-8k\mu ^2+\mu k^7+3k^6\mu +12\mu ^5k+44k^3\mu ^4+k^4-3k^5+3k^6-k^7-20\mu ^4k+35k^2\mu ^3-18\mu k^3+9k^5\mu ^3+144P_1\mu ^2-12k^3\mu ^5+4k^4\mu ^4-6k^6\mu ^2-11\mu k^5-15k^2\mu ^2)/\mu ^2$\\
$1/16(-38k^4\mu ^2-42k^2\mu ^4-4k^3\mu ^3+16k^5\mu ^2-11k^4\mu +8k^4\mu ^3+30k^3\mu ^2+8k\mu ^3+16k^3\mu ^6+k^8\mu +2\mu k^7-6k^7\mu ^2-8k^5\mu ^4-12k^4\mu ^5+13k^6\mu ^3-14k^6\mu +22\mu ^5k+10k^3\mu ^4-144P_1\mu ^3-k^4+4k^5-6k^6+4k^7-k^8-14\mu ^4k+7k^2\mu ^3+2\mu k^3-32k^5\mu ^3-54k^3\mu ^5+54k^4\mu ^4+7k^6\mu ^2+20\mu k^5-16\mu ^6k+44k^2\mu ^5-9k^2\mu ^2)/\mu ^3$\\
$1/120(\mu +k\mu +k^2-k)k(k-1)/\mu (k-2)(k-3)(k-4)$\\
$1/6(\mu +k\mu +k^2-k)k(k-1)(k^3-3k^2-3k^2\mu +3k\mu ^2+2k+6k\mu -3\mu ^2-\mu ^3-2\mu )/\mu $\\
$1/24(\mu +k\mu +k^2-k)k(k-1)(-4\mu k^3+k^4-6k^3+18k^2\mu +11k^2-22k\mu +6k^2\mu ^2-4k\mu ^3-6k+6\mu -18k\mu ^2+6\mu ^3+\mu ^4+11\mu ^2)/\mu ^2$\\
$-k^5+1/2k^4+5/2k^2\mu ^2-k\mu ^3-2\mu k^3-2k^3\mu ^2+3k^4\mu -\mu k^5+k^3\mu ^3-1/2k^4\mu ^2+1/2k^6$\\
$1/8(2k^5\mu ^2+25\mu k^3-4k^3\mu ^4+2k^4\mu ^3+7\mu k^5-6k^4\mu ^2-10k^2\mu ^3-12k^2\mu +6k^2\mu ^2-6k\mu ^3+10k\mu ^2+4k^2-4k\mu -6k^3+4\mu ^4k+5k^5-3k^6-k^4+k^7+14k^3\mu ^3-13k^4\mu -3k^6\mu -12k^3\mu ^2+72P_1\mu )/\mu $\\
$1/2(\mu +k\mu +k^2-k)k(k-1)(-5k^2\mu +k^3-k^2+3k\mu +8k\mu ^2-2\mu ^2-4\mu ^3)/\mu $\\
$1/2(4k^5\mu ^2+20\mu k^3-6k^3\mu ^4+3k^4\mu ^3+8\mu k^5-13k^4\mu ^2-15k^2\mu ^3-10k^2\mu +9k^2\mu ^2-7k\mu ^3+9k\mu ^2+4k^2-4k\mu -7k^3+6\mu ^4k+2k^5-2k^6+2k^4+k^7+19k^3\mu ^3-10k^4\mu -4k^6\mu -9k^3\mu ^2+72P_1\mu )/\mu $\\
$1/2(42k^4\mu ^2-4k^2\mu +27k^2\mu ^4-17k^3\mu ^3-30k^5\mu ^2+11k^4\mu +26k^4\mu ^3-38k^3\mu ^2-10k\mu ^3+4k\mu ^2-5\mu k^7+17k^6\mu -9\mu ^5k-21k^3\mu ^4+k^4-4k^5+6k^6-4k^7+k^8+3\mu ^4k+4k^2\mu ^3+4\mu k^3-3k^5\mu ^3-72P_1\mu ^2+9k^3\mu ^5-9k^4\mu ^4+9k^6\mu ^2-23\mu k^5+13k^2\mu ^2)/\mu ^2$\\
$1/12(\mu +k\mu +k^2-k)k(k-1)(k^4-3k^3-7\mu k^3+18k^2\mu ^2+2k^2+15k^2\mu -24k\mu ^2-20k\mu ^3-6k\mu +4\mu ^2+12\mu ^3+8\mu ^4)/\mu ^2$\\
$1/12(-92k^4\mu ^2-51k^2\mu ^4+119k^3\mu ^3+116k^5\mu ^2-24k^4\mu -140k^4\mu ^3+25k^3\mu ^2-4k\mu ^3-14k^3\mu ^6-6k^8\mu +30\mu k^7+15k^7\mu ^2-3k^5\mu ^4+21k^4\mu ^5-16k^6\mu ^3-58k^6\mu +10\mu ^5k+77k^3\mu ^4+72P_1\mu ^3-2k^4+9k^5-16k^6+14k^7-6k^8+k^9+16\mu ^4k-34k^2\mu ^3+4\mu k^3+75k^5\mu ^3+18k^3\mu ^5-39k^4\mu ^4-66k^6\mu ^2+54\mu k^5+14\mu ^6k-49k^2\mu ^5+2k^2\mu ^2)/\mu ^3$\\
$1/10(\mu +k\mu +k^2-k)k(k-1)(2k\mu +k^3-k^2-\mu ^2+9k\mu ^2-5k^2\mu -5\mu ^3)/\mu $\\
$1/12(-3k^5\mu ^2-20\mu k^3+4k^3\mu ^4+3k^4\mu ^2+8k^2\mu ^3+10k^2\mu -9k^2\mu ^2+7k\mu ^3-9k\mu ^2-4k^2+4k\mu +7k^3-4\mu ^4k-k^5-2k^4-15k^3\mu ^3+5k^4\mu +k^6\mu +18k^3\mu ^2-72P_1\mu )/\mu $\\
$1/2(6k^5\mu ^2+20\mu k^3-8k^3\mu ^4+4k^4\mu ^3+10\mu k^5-19k^4\mu ^2-20k^2\mu ^3-10k^2\mu +9k^2\mu ^2-7k\mu ^3+9k\mu ^2+4k^2-4k\mu -7k^3+8\mu ^4k+2k^5-2k^6+2k^4+k^7+23k^3\mu ^3-11k^4\mu -5k^6\mu -5k^3\mu ^2+72P_1\mu )/\mu $\\
$1/2(41k^4\mu ^2-4k^2\mu +41k^2\mu ^4-28k^3\mu ^3-39k^5\mu ^2+3k^4\mu +46k^4\mu ^3-29k^3\mu ^2-9k\mu ^3+4k\mu ^2-6\mu k^7+17k^6\mu -14\mu ^5k-33k^3\mu ^4-k^5+3k^6-3k^7+k^8+5\mu ^4k-2k^2\mu ^3+7\mu k^3-7k^5\mu ^3-72P_1\mu ^2+14k^3\mu ^5-13k^4\mu ^4+13k^6\mu ^2-17\mu k^5+10k^2\mu ^2)/\mu ^2$\\
$1/4(51k^4\mu ^2-8k^2\mu +56k^2\mu ^4-19k^3\mu ^3-47k^5\mu ^2+k^4\mu +59k^4\mu ^3-49k^3\mu ^2-18k\mu ^3+8k\mu ^2-7\mu k^7+19k^6\mu -20\mu ^5k-52k^3\mu ^4-k^5+3k^6-3k^7+k^8+12\mu ^4k-11k^2\mu ^3+14\mu k^3-11k^5\mu ^3-144P_1\mu ^2+20k^3\mu ^5-16k^4\mu ^4+17k^6\mu ^2-19\mu k^5+20k^2\mu ^2)/\mu ^2$\\
$1/4(-28k^4\mu ^2-4k^2\mu ^4+77k^3\mu ^3+78k^5\mu ^2-3k^4\mu -128k^4\mu ^3-11k^3\mu ^2-8k\mu ^3-29k^3\mu ^6-7k^8\mu +24\mu k^7+21k^7\mu ^2+5k^5\mu ^4+34k^4\mu ^5-29k^6\mu ^3-30k^6\mu -9\mu ^5k+85k^3\mu ^4+144P_1\mu ^3+k^5-4k^6+6k^7-4k^8+k^9+18\mu ^4k-21k^2\mu ^3+109k^5\mu ^3+67k^3\mu ^5-104k^4\mu ^4-68k^6\mu ^2+16\mu k^5+29\mu ^6k-92k^2\mu ^5+8k^2\mu ^2)/\mu ^3$\\
$1/8(67k^4\mu ^2-4k^2\mu +48k^2\mu ^4-47k^3\mu ^3-52k^5\mu ^2+16k^4\mu +51k^4\mu ^3-50k^3\mu ^2-11k\mu ^3+4k\mu ^2-7\mu k^7+24k^6\mu -16\mu ^5k-31k^3\mu ^4+k^4-4k^5+6k^6-4k^7+k^8-\mu ^4k+15k^2\mu ^3+3\mu k^3-8k^5\mu ^3-72P_1\mu ^2+16k^3\mu ^5-16k^4\mu ^4+16k^6\mu ^2-32\mu k^5+15k^2\mu ^2)/\mu ^2$\\
$1/4(-51k^4\mu ^2-31k^2\mu ^4+111k^3\mu ^3+111k^5\mu ^2-5k^4\mu -183k^4\mu ^3-5k^3\mu ^2-8k\mu ^3-36k^3\mu ^6-8k^8\mu +29\mu k^7+26k^7\mu ^2+7k^5\mu ^4+44k^4\mu ^5-38k^6\mu ^3-39k^6\mu +133k^3\mu ^4+144P_1\mu ^3+k^5-4k^6+6k^7-4k^8+k^9+20\mu ^4k-27k^2\mu ^3+145k^5\mu ^3+72k^3\mu ^5-129k^4\mu ^4-89k^6\mu ^2+23\mu k^5+36\mu ^6k-116k^2\mu ^5+8k^2\mu ^2)/\mu ^3$\\
$1/12(18k^4\mu ^2+33k^2\mu ^4+6k^3\mu ^3-90k^5\mu ^2+k^4\mu +132k^4\mu ^3+k^3\mu ^2+61k^3\mu ^7+30k^8\mu ^2-62k^7\mu ^3+12k^5\mu ^5-89k^4\mu ^6+61k^6\mu ^4-8\mu k^9-125k^3\mu ^6+33k^8\mu -52\mu k^7-61\mu ^7k-111k^7\mu ^2-300k^5\mu ^4+243k^4\mu ^5+229k^6\mu ^3+38k^6\mu -26\mu ^5k-196k^3\mu ^4+211k^2\mu ^6-k^5+5k^6-10k^7+10k^8-5k^9+k^10+12\mu ^4k-13k^2\mu ^3-292k^5\mu ^3-216P_1\mu ^4-280k^3\mu ^5+390k^4\mu ^4+152k^6\mu ^2-12\mu k^5+3\mu ^6k+51k^2\mu ^5)/\mu ^4$\\
$1/48(-21k^4\mu ^2+6k^2\mu ^4+123k^3\mu ^3+99k^5\mu ^2-201k^4\mu ^3-28k^3\mu ^2-16k\mu ^3-56k^3\mu ^6-9k^8\mu +27\mu k^7+33k^7\mu ^2+18k^5\mu ^4+60k^4\mu ^5-55k^6\mu ^3-27k^6\mu -20\mu ^5k+138k^3\mu ^4+288P_1\mu ^3-k^6+3k^7-3k^8+k^9+36\mu ^4k-40k^2\mu ^3+189k^5\mu ^3+132k^3\mu ^5-198k^4\mu ^4-99k^6\mu ^2+9\mu k^5+56\mu ^6k-172k^2\mu ^5+16k^2\mu ^2)/\mu ^3$\\
$1/16(9k^4\mu ^2+42k^2\mu ^4+21k^3\mu ^3-68k^5\mu ^2+112k^4\mu ^3+88k^3\mu ^7+37k^8\mu ^2-83k^7\mu ^3+8k^5\mu ^5-124k^4\mu ^6+90k^6\mu ^4-9\mu k^9-188k^3\mu ^6+32k^8\mu -42\mu k^7-88\mu ^7k-124k^7\mu ^2-416k^5\mu ^4+364k^4\mu ^5+288k^6\mu ^3+24k^6\mu -36\mu ^5k-208k^3\mu ^4+300k^2\mu ^6+k^6-4k^7+6k^8-4k^9+k^10+16\mu ^4k-16k^2\mu ^3-322k^5\mu ^3-288P_1\mu ^4-372k^3\mu ^5+476k^4\mu ^4+146k^6\mu ^2-5\mu k^5+12\mu ^6k+36k^2\mu ^5)/\mu ^4$\\
$1/48(3k^4\mu ^2+432P_1\mu ^5+22k^2\mu ^4+k^3\mu ^3+9k^5\mu ^2+164\mu ^8k+376k^3\mu ^7-141k^8\mu ^2+376k^7\mu ^3+921k^5\mu ^5-781k^4\mu ^6-709k^6\mu ^4+35\mu k^9+781k^3\mu ^6-50k^8\mu -592k^2\mu ^7+30\mu k^7-48\mu ^7k+174k^7\mu ^2+735k^5\mu ^4-967k^4\mu ^5-411k^6\mu ^3-5k^6\mu -24\mu ^5k-64k^3\mu ^4+21k^2\mu ^6-k^6+5k^7-10k^8+10k^9-5k^10+k^11+147k^5\mu ^3+264k^4\mu ^7+41k^9\mu ^2-164k^3\mu ^8-113k^8\mu ^3-148k^6\mu ^5-73k^5\mu ^6+189k^7\mu ^4-9k^10\mu +263k^3\mu ^5-173k^4\mu ^4-86k^6\mu ^2-\mu k^5+52\mu ^6k-45k^2\mu ^5)/\mu ^5$\\
$1/720(25k^4\mu ^2-26k^2\mu ^4+15k^3\mu ^3-70k^5\mu ^2+228\mu ^8k-85k^4\mu ^3-720P_1\mu ^6-345\mu ^9k-1527k^3\mu ^7+175k^8\mu ^2+1293k^2\mu ^8-395k^7\mu ^3-1316k^5\mu ^5+1735k^4\mu ^6+840k^6\mu ^4-45\mu k^9-36k^3\mu ^6-504k^2\mu ^7+45\mu k^7-165\mu ^7k-100k^7\mu ^2-120k^5\mu ^4-148k^4\mu ^5+10k^6\mu ^3-36k^6\mu -24\mu ^5k-98k^3\mu ^4+190k^2\mu ^6+k^6-6k^7+15k^8-20k^9+15k^10-6k^11+k^12+165k^5\mu ^3+1944k^4\mu ^7-150k^9\mu ^2-918k^3\mu ^8+435k^8\mu ^3+1731k^6\mu ^5-2232k^5\mu ^6-990k^7\mu ^4+36k^10\mu +345k^3\mu ^9+252k^5\mu ^7+45k^10\mu ^2-603k^4\mu ^8-145k^9\mu ^3-432k^7\mu ^5+253k^6\mu ^6+315k^8\mu ^4-9k^11\mu +224k^3\mu ^5+79k^4\mu ^4+75k^6\mu ^2+9\mu k^5+90\mu ^6k-35k^2\mu ^5)/\mu ^6$\\

\end{longtable}
\normalsize
\setcounter{rowcount}{0}

\subsection*{General SRG}

order $4$:

\begin{multicols}{4}
\scriptsize
\setbox\ltmcbox\vbox{
\makeatletter\col@number\@ne
\begin{longtable} {@{\stepcounter{rowcount}\therowcount.\hspace*{\tabcolsep}}l}
$\begin{matrix}  &  &  & \\ 0 & 0 & 0 & 1\\ 
0 & 0 & 0 & 1\\ 
0 & 0 & 0 & 1\\ 
1 & 1 & 1 & 0\end{matrix}$\\
$\begin{matrix}  &  &  & \\ 0 & 0 & 0 & 1\\ 
0 & 0 & 1 & 1\\ 
0 & 1 & 0 & 1\\ 
1 & 1 & 1 & 0\end{matrix}$\\
$\begin{matrix}  &  &  & \\ 0 & 0 & 1 & 1\\ 
0 & 0 & 1 & 1\\ 
1 & 1 & 0 & 1\\ 
1 & 1 & 1 & 0\end{matrix}$\\
$\begin{matrix}  &  &  & \\ 0 & 0 & 1 & 1\\ 
0 & 0 & 1 & 1\\ 
1 & 1 & 0 & 0\\ 
1 & 1 & 0 & 0\end{matrix}$\\
$\begin{matrix}  &  &  & \\ 0 & 0 & 1 & 0\\ 
0 & 0 & 0 & 1\\ 
1 & 0 & 0 & 1\\ 
0 & 1 & 1 & 0\end{matrix}$\\
$\begin{matrix}  &  &  & \\ 0 & 0 & 0 & 0\\ 
0 & 0 & 0 & 1\\ 
0 & 0 & 0 & 1\\ 
0 & 1 & 1 & 0\end{matrix}$\\
$\begin{matrix}  &  &  & \\ 0 & 1 & 1 & 1\\ 
1 & 0 & 1 & 1\\ 
1 & 1 & 0 & 1\\ 
1 & 1 & 1 & 0\end{matrix}$\\
$\begin{matrix}  &  &  & \\ 0 & 0 & 0 & 0\\ 
0 & 0 & 1 & 1\\ 
0 & 1 & 0 & 1\\ 
0 & 1 & 1 & 0\end{matrix}$\\
$\begin{matrix}  &  &  & \\ 0 & 0 & 0 & 1\\ 
0 & 0 & 1 & 0\\ 
0 & 1 & 0 & 0\\ 
1 & 0 & 0 & 0\end{matrix}$\\
$\begin{matrix}  &  &  & \\ 0 & 0 & 0 & 0\\ 
0 & 0 & 0 & 0\\ 
0 & 0 & 0 & 1\\ 
0 & 0 & 1 & 0\end{matrix}$\\
$\begin{matrix}  &  &  & \\ 0 & 0 & 0 & 0\\ 
0 & 0 & 0 & 0\\ 
0 & 0 & 0 & 0\\ 
0 & 0 & 0 & 0\end{matrix}$\\
\end{longtable}
\unskip
\unpenalty
\unpenalty}

\unvbox\ltmcbox
\end{multicols}
\setcounter{rowcount}{0}

\scriptsize
\begin{longtable}{@{\stepcounter{rowcount}\therowcount.\hspace*{\tabcolsep}}p{15cm}}
$P_1$\\
$1/2 (-2 k^3 \mu \lambda+k^2 \mu \lambda^2-2 k^2+2 k \mu-k^2 \mu+5 k^3-5 k^2 \lambda+k^2 \mu \lambda+3 k \mu \lambda-2 k^3 \mu+8 k^3 \lambda-4 k^2 \lambda^2-4 k^4+k^4 \mu+k^5-3 k^4 \lambda+3 k^3 \lambda^2+k \mu \lambda^2-k^2 \lambda^3-6 P_1 \mu)/\mu$\\
$-1/4 (-3 k^3 \mu \lambda+2 k^2 \mu \lambda^2-2 k^2+2 k \mu-k^2 \mu+5 k^3-6 k^2 \lambda+k^2 \mu \lambda+4 k \mu \lambda-2 k^3 \mu+10 k^3 \lambda-6 k^2 \lambda^2-4 k^4+k^4 \mu+k^5-4 k^4 \lambda+5 k^3 \lambda^2+2 k \mu \lambda^2-2 k^2 \lambda^3-6 P_1 \mu)/\mu$\\
$1/8 (-k^2 \mu^2 \lambda-5 k^3 \mu \lambda+3 k^2 \mu \lambda^2-3 k^2+3 k \mu+7 k^3-8 k^2 \lambda-k \mu^2-k \mu^2 \lambda+4 k^2 \mu \lambda+5 k \mu \lambda-5 k^3 \mu+12 k^3 \lambda-7 k^2 \lambda^2-5 k^4+k^3 \mu^2+2 k^4 \mu+k^5-4 k^4 \lambda+5 k^3 \lambda^2+2 k \mu \lambda^2-2 k^2 \lambda^3-6 P_1 \mu)/\mu$\\
$-1/2 (-k^2 \mu^2 \lambda-3 k^3 \mu \lambda+2 k^2 \mu \lambda^2-2 k^2+2 k \mu+k^2 \mu+4 k^3-5 k^2 \lambda-k \mu^2-k \mu^2 \lambda+4 k^2 \mu \lambda+3 k \mu \lambda-4 k^3 \mu+6 k^3 \lambda-4 k^2 \lambda^2-2 k^4+k^3 \mu^2+k^4 \mu-k^4 \lambda+2 k^3 \lambda^2+k \mu \lambda^2-k^2 \lambda^3-6 P_1 \mu)/\mu$\\
$1/2 (-k^3 \mu^2 \lambda+4 k^2 \mu^2 \lambda-k^2 \mu^3 \lambda+k^2 \mu^2 \lambda^2-k^3+k \mu^2+k \mu^2 \lambda-k \mu^3+2 k^2 \mu^2-3 k^3 \lambda+3 k^4-3 k^3 \mu^2-k \mu^3 \lambda+k^6-3 k^5+k^3 \mu^3-3 k^5 \lambda+3 k^4 \lambda^2+6 k^4 \lambda-k^3 \lambda^3-3 k^3 \lambda^2-6 P_1 \mu^2)/\mu^2$\\
\end{longtable}

\begin{longtable}{@{\stepcounter{rowcount}\therowcount.\hspace*{\tabcolsep}}p{15cm}}
$1/24 (-3 k^3 \mu \lambda+3 k^2 \mu \lambda^2-2 k^2+2 k \mu-k^2 \mu+5 k^3-5 k^2 \lambda+3 k \mu \lambda-2 k^3 \mu+9 k^3 \lambda-6 k^2 \lambda^2-4 k^4+k^4 \mu+k^5-4 k^4 \lambda+6 k^3 \lambda^2+3 k \mu \lambda^2-3 k^2 \lambda^3-6 P_1 \mu)/\mu$\\
$-1/6 (-k^3 \mu^2 \lambda+k^2 \mu^2 \lambda+6 k^3 \mu \lambda-3 k^4 \mu \lambda-k^2 \mu \lambda^2+2 k^3 \mu \lambda^2-2 k^2 \mu+2 k \mu^2+2 k \mu^2 \lambda-3 k^2 \mu \lambda+5 k^3 \mu-k^2 \mu^2-k^3 \lambda-2 k^3 \mu^2-4 k^4 \mu+k^4 \mu^2+\mu k^5-k^5 \lambda+2 k^4 \lambda^2+2 k^4 \lambda-k^3 \lambda^3-2 k^3 \lambda^2-6 P_1 \mu^2)/\mu^2$\\
$1/8 (4 k^2 \mu^2+k \lambda^2 \mu^2-4 k^2 \mu \lambda^2+3 k \mu^2 \lambda+8 k^2 \mu^2 \lambda-6 k^3 \mu^2-3 k^3 \mu^2 \lambda-2 k^2 \mu^3 \lambda-2 k \mu^3+3 k^2 \mu^2 \lambda^2-2 k^5-2 k \mu^3 \lambda+2 k \mu^2+2 k^3 \mu^3-5 k^2 \mu \lambda+k^6+k^4 \lambda^2-2 k^5 \lambda-6 P_1 \mu^2+2 k^4 \mu+2 k^4 \lambda+k^3 \mu-2 k^2 \mu-k^2 \mu \lambda^3+k^3 \mu \lambda-\mu k^5+k^4+2 k^4 \mu \lambda)/\mu^2$\\
$1/4 (k^2 \mu^2+6 k^3 \mu \lambda^2+2 \mu^4 k-2 k^6 \mu+2 \mu^4 k \lambda+2 k^3 \mu \lambda^3+k^2 \mu^2 \lambda+2 k^3 \mu^2+4 k^3 \mu^2 \lambda-9 k^2 \mu^3 \lambda-2 k \mu^3-5 k^2 \mu^3+3 k^5+6 P_1 \mu^3-2 k \mu^3 \lambda+6 k^3 \mu^3+k^3 \mu^2 \lambda^2+6 \mu k^5 \lambda-6 k^4 \mu \lambda^2-3 k^6+2 k^2 \mu^4 \lambda-3 k^4 \lambda^2+6 k^5 \lambda+k^4 \mu^3-2 k^3 \mu^4-6 k^4 \mu+k^3 \mu^3 \lambda-k^4 \mu^2 \lambda-2 k^2 \mu^3 \lambda^2-3 k^4 \lambda+2 k^3 \mu-3 k^6 \lambda+3 k^5 \lambda^2+k^7+6 k^3 \mu \lambda-3 k^4 \mu^2+6 \mu k^5-k^4 \lambda^3-k^4-12 k^4 \mu \lambda)/\mu^3$\\
$1/24 (-k^2 \mu^2+6 k^3 \mu \lambda^2-6 \mu k^5 \lambda^2-3 k \mu^5 \lambda+5 \mu^4 k+8 k^2 \mu^4-2 k^7 \mu-2 k^4 \mu^4+4 k^6 \mu+3 k^6 \mu^2-12 k^5 \mu^2+5 \mu^4 k \lambda+2 k^3 \mu \lambda^3-2 k^2 \mu^2 \lambda-4 k^3 \mu^2-14 k^3 \mu^2 \lambda-8 k^2 \mu^3 \lambda-2 k \mu^3-k^2 \mu^2 \lambda^2-5 k^2 \mu^3-4 k^5-2 k \mu^3 \lambda+6 k^6 \mu \lambda+k^5 \mu^3+k^4 \lambda^4+k^3 \mu^3-14 k^3 \mu^2 \lambda^2-6 \mu k^5 \lambda+6 k^6-3 k \mu^5-6 P_1 \mu^4+16 k^2 \mu^4 \lambda+6 k^4 \lambda^2-12 k^5 \lambda+5 k^4 \mu^3-10 k^5 \mu^2 \lambda-11 k^3 \mu^4-4 k^4 \mu-6 k^3 \mu^3 \lambda+26 k^4 \mu^2 \lambda-k^2 \mu^3 \lambda^2+4 k^4 \lambda+11 k^4 \mu^2 \lambda^2+2 k^3 \mu+12 k^6 \lambda-12 k^5 \lambda^2-4 k^7-4 k^3 \lambda^3 \mu^2+k^8+6 k^3 \mu \lambda+14 k^4 \mu^2-k^3 \mu^4 \lambda-4 k^7 \lambda-k^3 \mu^3 \lambda^2+4 k^4 \lambda^3-3 k^2 \mu^5 \lambda+3 k^2 \lambda^2 \mu^4+k^4+6 k^6 \lambda^2-4 k^5 \lambda^3+3 k^3 \mu^5-6 k^4 \mu \lambda+2 k^4 \mu \lambda^3)/\mu^4$\\
\end{longtable}
\setcounter{rowcount}{0}

\nocite{blog}
\nocite{spectra}
\nocite{klin}
\nocite{tfSRG}
%\newpage
%\input{suhrn.tex}

%\section*{Zoznam príloh}
%\addcontentsline{toc}{section}{\numberline{}Zoznam príloh}

%\addcontentsline{toc}{section}{\numberline{}Zoznam obrázkov}
%\listoffigures

%\addcontentsline{toc}{section}{\numberline{}Zoznam tabuliek}
%\listoftables

\end{document}